\documentclass[11pt, reqno]{amsart}
\usepackage{etex}
\usepackage{textcmds} 
\usepackage{amsmath, amssymb, amsfonts, amstext, verbatim, amsthm, mathrsfs}
\usepackage{microtype}
\usepackage[all]{xy}
\usepackage[modulo]{lineno}
\usepackage[usenames,dvipsnames]{xcolor}
\usepackage{aliascnt}
\usepackage{enumitem}
\usepackage{xspace}
\usepackage{amsfonts}
\usepackage[centertags]{amsmath}
\usepackage{rotating}
\usepackage[margin=3.5cm]{geometry}
\usepackage{dsfont}
\usepackage{bm}
\usepackage{color}
\usepackage{subfigure}
\usepackage{amsmath}
\usepackage{array}
\usepackage[all]{xy}
\usepackage{euscript}
\usepackage[T1]{fontenc}
\usepackage{mathbbol}
\usepackage{nccmath}
\usepackage{marginnote}
\usepackage{stmaryrd}
\usepackage{youngtab}
\usepackage{tikz,pgfplots}
\usetikzlibrary{calc}
\usepackage{blkarray}
\usepackage{ulem}
\usepackage[abs]{overpic}

\usepackage{graphics,graphicx}  %% comment for "draft" mode w/o pictures
\let\oldtocsection=\tocsection
\let\oldtocsubsection=\tocsubsection

\renewcommand{\tocsection}[2]{\hspace{0em}\oldtocsection{#1}{#2}}
\renewcommand{\tocsubsection}[2]{\hspace{1em}\oldtocsubsection{#1}{#2}}

\usepackage[backref,colorlinks=true,linkcolor=black,citecolor=blue,urlcolor=blue,citebordercolor={0 0 1},urlbordercolor={0 0 1},linkbordercolor={0 0 1}]{hyperref} %needs to be 

\tikzset{node distance=3cm, auto}

\makeatletter
\def\@secnumfont{\bfseries}
\def\section{\@startsection{section}{1}%
  \z@{.7\linespacing\@plus\linespacing}{.5\linespacing}%
  {\normalfont\Large\bfseries}}
\def\subsection{\@startsection{subsection}{2}%
  \z@{.5\linespacing\@plus.7\linespacing}{-.5em}%
  {\normalfont\large\bfseries}}
  \def\subsubsection{\@startsection{subsubsection}{3}%
  \z@{.5\linespacing\@plus.7\linespacing}{-.5em}%
  {\normalfont\bfseries}}
\makeatother

\newtheorem{thm}{Theorem}[subsection]
\newtheorem{lemma}[thm]{Lemma}
\newtheorem{prop}[thm]{Proposition}
\newtheorem{cor}[thm]{Corollary}

\newtheorem{conjecture}[thm]{Conjecture}

\newtheorem{mainthm}{Theorem}

\newtheorem{definition}[thm]{Definition}

\theoremstyle{remark}
\numberwithin{equation}{subsection} 
 \numberwithin{figure}{section}
\newtheoremstyle{customremark}% <name>
{3pt}% <Space above>
{3pt}% <Space below>
{}% <Body font>
{}% <Indent amount>
{\bfseries}% <Theorem head font>
{.}% <Punctuation after theorem head>
{.5em}% <Space after theorem headi>
{}% <Theorem head spec (can be left empty, meaning `normal')>
\theoremstyle{customremark}
\newtheorem{rmk_no_diamond}[thm]{Remark}
\newenvironment{rmk}{\begin{rmk_no_diamond} } {\hfill$\er$ \end{rmk_no_diamond}}
\newtheorem{example_no_diamond}[thm]{Example}
\newenvironment{example}{\begin{example_no_diamond} } {\hfill$\er$ \end{example_no_diamond}}

\newcommand{\oCP}{{\overline{{\C}P}}\!\,}

\newcommand{\id}{{\rm id}}

\newcommand{\Tt}{\mathcal{T}}

\newenvironment{itemlist}
   { \begin{list} {$\bullet$}
         { \setlength{\topsep}{.5ex}  \setlength{\itemsep}{.5ex} \setlength{\leftmargin}{2.5ex} } }
   { \end{list} }
   
    \newcommand{\vn}{{\vec{n}}}

 \newcommand{\ovx}{{\ov{x}}}
 \newcommand{\ovy}{{\ov{y}}}
  \newcommand{\ovr}{\overrightarrow}

\newcommand{\bE}{{\bf{E}}}
\newcommand{\bB}{{\bf{B}}}

\newcommand{\ov}{\overline}
\newcommand{\al}{{\alpha}}

\newcommand{\la}{{\lambda}}

\newcommand{\er}{{\Diamond}}
\newcommand{\Z}{\mathbb{Z}}

\newcommand{\R}{\mathbb{R}}

\newcommand{\C}{\mathbb{C}}
\newcommand{\CP}{\mathbb{CP}}

\newcommand{\eps}{\varepsilon}

\newcommand{\X}{\mathcal{X}}

\newcommand{\T}{\mathcal{T}}

\newcommand{\acc}{\mathrm{acc}}
\newcommand{\sembeds}{\stackrel{s}{\hookrightarrow}}

\makeatletter
\newcommand{\dashover}[2][\mathop]{#1{\mathpalette\df@over{{\dashfill}{#2}}}}
\newcommand{\fillover}[2][\mathop]{#1{\mathpalette\df@over{{\solidfill}{#2}}}}
\newcommand{\df@over}[2]{\df@@over#1#2}
\newcommand\df@@over[3]{%
  \vbox{
    \offinterlineskip
    \ialign{##\cr
      #2{#1}\cr
      \noalign{\kern1pt}
      $\m@th#1#3$\cr
    }
  }%
}
\newcommand{\dashfill}[1]{%
  \kern-.5pt
  \xleaders\hbox{\kern.5pt\vrule height.4pt width \dash@width{#1}\kern.5pt}\hfill
  \kern-.5pt
}
\newcommand{\dash@width}[1]{%
  \ifx#1\displaystyle
    2pt
  \else
    \ifx#1\textstyle
      1.5pt
    \else
      \ifx#1\scriptstyle
        1.25pt
      \else
        \ifx#1\scriptscriptstyle
          1pt
        \fi
      \fi
    \fi
  \fi
}
\newcommand{\solidfill}[1]{\leaders\hrule\hfill}
\makeatother

\date{\today}
\title{Quadrilateral Mutations and Symplectic Embeddings}
\author{Nicki Magill}
\address{Mathematics Department, UC Berkeley}
\email{nmagill@berkeley.edu}
\thanks{NM thanks the NSF for their support under the agreement DMS-2402169}

\begin{document}
 \begin{abstract}
We study the relationship between almost toric base diagrams, perfect
exceptional classes, and optimal ellipsoid embeddings for
$H_b=\CP^2_1 \#\oCP^2_b$ and $P_b=S^2_1\times S^2_b$.
Starting from a quadrilateral almost toric base diagram with one Delzant corner, we encode the three non-Delzant corners by a recursive triple. We show that every quadrilateral obtained via a well-defined sequence of mutations from the initial diagrams is encoded by a recursive triple in the same way. Moreover, geometric mutation of these diagrams corresponds to algebraic mutation of the associated triples. These algebraic mutations are the recursive operations used to generate the $(p,q)$-perfect classes for $H$. 

We apply this dictionary to realize every $(p,q)$-perfect class for $H$ by an explicit sequence of almost toric mutations for suitable values of $b$. We also prove the analogous realization result for triples of quasi-perfect classes for $P$, showing that these classes are in fact $(p,q)$-perfect. Finally, we apply these results to ellipsoid embedding problems, including visible embeddings, visible obstructions, and ATF-visible staircases.
\end{abstract}
\maketitle
   
    \section{Introduction}

This work investigates the connection between perfect exceptional classes, mutations of almost toric base diagrams (ATBDs), and optimal symplectic embeddings. We focus on the examples of $H := \CP^2\#\oCP^2$ and $P := S^2 \times S^2$, using the notation
$H_b := \CP^2_1\#\oCP^2_b$ for $b \in (0,1)$ and $P_b := S^2_1 \times S^2_b$ for $b \in [1,\infty)$ to specify the symplectic form. 
 
A series of papers~\cite{ICERM, MM, MMW, MPW} demonstrated that the $(p,q)$-perfect classes for $H$ characterize the $b$-values for which $H_b$ has an infinite staircase, a property of ellipsoid embeddings into $H_b$. Further, in~\cite{MMW}, all of the
$(p,q)$-perfect classes for $H$ were constructed via particular recursive sequences applied to specific seed classes. In~\cite{M1}, generalizing work of Casals--Vianna~\cite{CV}
and Cristofaro-Gardiner et al.~\cite{AADT}, we showed how particular sequences of perfect classes corresponded to particular sequences of mutations on ATBDs with one Delzant corner for $H_b$.

In this paper, we generalize the work in~\cite{M1} by systematically studying the mutation graph of a starting ATBD with one Delzant corner for \(H_b\) and \(P_b\); see Figure~\ref{fig:basediagr}. For each mutated quadrilateral, the three
non-Delzant corners are encoded by a triple \(\Tt\), and the nodal rays, edge directions, and affine side lengths are then prescribed by the condition that the quadrilateral is \(\Tt\)-standard. With this encoding, geometric ATF mutations are described by algebraic mutations of triples. The same recursive mutation formulas work for \(H_b\) and \(P_b\), and they agree with the recursive formulas used to generate the
\((p,q)\)-perfect classes for \(H\).

 \begin{mainthm}\label{thm:A}
For \(H_b\) and \(P_b\), the quadrilateral ATBDs obtained by well-defined mutations from the starting base diagrams have a standard form determined by a recursive triple. Under this encoding, geometric ATF mutations are described by algebraic mutations of triples. The algebraic mutation formulas are the same for \(H_b\) and \(P_b\), and agree with the recursive formulas used to generate the \((p,q)\)-perfect classes for \(H\).
\end{mainthm}

The nodal rays and direction vectors of the ATBD are given in Lemma~\ref{lem:standFrom}. To see how the side lengths are determined via the triple, see Definition~\ref{def:DecQuad}. The triples satisfy many numerical properties we consider in Section~\ref{ss:trip}. The theorem about the mutations agreeing is given in Theorem~\ref{thm:ATFMut}.

 Using work of McDuff and Siegel~\cite{McSi}, when the entries of the combinatorial triple describing an ATBD for $H$ or $P$ are positive, each element of the triples corresponds to a perfect class. We give explicit mutation sequences realizing each perfect class for $H$ in an ATBD:
 
\begin{mainthm} \label{thm:B}
For each perfect class $E$ for $H$, there exists an explicit mutation sequence exhibiting $E$ in an ATBD for a specified interval of $b$-values. One endpoint of this interval is determined from $E$.
\end{mainthm}

The precise statement is Theorem~\ref{thm:HTrip}, which is proved in Section~\ref{ss:tripPerf}. In giving these explicit mutation sequences, we show that all $(p,q)$-perfect classes for $H$ can be constructed from three fixed base classes, whereas the classification in~\cite{MMW} used a family of starting classes. We further partially address the
conjectures in~\cite{MPW, MMW, REU} regarding the connection between $P$ and $H$:
 
\begin{mainthm} \label{thm:C}
There exists a family of perfect classes for $P$ analogous to the perfect classes for $H$. For each of these classes, there exists an explicit mutation sequence exhibiting $E$ in an ATBD for a specified interval of $b$-values. One endpoint of this interval is determined from $E$.
\end{mainthm}

This is proved in Section~\ref{ss:tripPerf} with the precise statement in Theorem~\ref{thm:PTrip}. In upcoming work, McDuff and Siegel prove the perfect classes in Theorem~\ref{thm:C} are all the perfect classes for $P.$ The proofs of Theorems~\ref{thm:B} and~\ref{thm:C} are similar because the mutation structures for $P$ and $H$ are parallel.

In upcoming work, McDuff and Siegel show that the perfect classes for $H$ and $P$ can be realized via sequences of mutations on scattering diagrams. Interestingly, the mutation sequences we construct in Theorems~\ref{thm:B} and~\ref{thm:C} coincide precisely with theirs. From their perspective, there are multiple mutation sequences that obtain the perfect classes. The shortest such paths correspond to the sequences in Theorem~\ref{thm:B} and ~\ref{thm:C} that can be realized for a fixed $b$-value. Additionally, the mutations on the scattering diagrams are not linear on the $(p,q)$-coordinates, so it is not obvious that these sequences should line up. 

One advantage of the ATBD approach is its dependence on $b$. This dependence on $b$ is what connects the mutation structure to the ellipsoid embedding problem. We summarize the application to embeddings in Section~\ref{ss:introEmbed}. As an application, we develop conditions under which a sequence of ATF
mutations for \(H_b\) or \(P_b\) gives an infinite staircase; following McDuff and Siegel~\cite{McSi}, we call this an \textbf{ATF-visible staircase}. Their work showed that $H_{1/3}$ has an ATF-visible staircase. In Theorem~\ref{thm:ATFvis}, we give conditions under which an ATF-visible staircase can be constructed, and apply this to prove:
 
\begin{mainthm} \label{thm:D}
There exist infinitely many $b$-values where the ellipsoid embedding function for $H_b$ admits an increasing ATF-visible staircase.
\end{mainthm}

These are precisely the \(b\)-values for which \(H_b\) has only increasing staircases, i.e. the infinitely many nonsmooth points accumulate to the left. More information is found in Section~\ref{ss:embed4}. 

The paper is organized as follows. Section~2 develops the dictionary between \(\Tt\)-standard quadrilaterals and recursive triples, proving Theorem~\ref{thm:A}. Section~3 applies this dictionary to perfect classes for \(H\) and \(P\), proving Theorems~\ref{thm:B} and~\ref{thm:C}. Section~4 studies visible embeddings and obstructions, and proves the
ATF-visible staircase result, Theorem~\ref{thm:D}.

\subsection{Computing the ATF mutations} 
We begin by giving the necessary definitions to state Theorem~\ref{thm:A} more precisely. We will only consider cases where the ATBD for $H_b,P_b$ is a quadrilateral with one Delzant corner. To get the starting ATBD, we perform three nodal trades to the moment polytope of $H_b,P_b$. We denote these starting base diagrams as $Q_{P_b}^0$ or $Q_{H_b}^0,$ see Figure~\ref{fig:basediagr}.  In this section, we will describe the recursive formulas for determining possible mutations on $Q_{\bullet,b}^0$. 

Throughout the paper, for $\bullet\in\{H,P\}$, we use the notation
\[Q_{\bullet,b}^0=\begin{cases}
Q_{H_b}^0, \\
Q_{P_b}^0,
\end{cases}
\qquad
Q_{\bullet,b}(\Tt)=
\begin{cases}
Q_{H_b}(\Tt),\\
Q_{P_b}(\Tt).
\end{cases}
\]

\begin{figure}[h!]
    \centering
     \begin{overpic}[scale=1,unit=0.5mm]{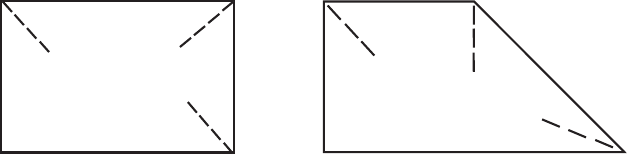}
 	\put (-6,-6) {$O$}
  \put (80,-6) {$X$}
   \put (80,55) {$V$}
    \put (-5,55) {$Y$}
    \put (10,45) {$\vn_Y$}
    \put (60,45) {$\vn_V$}
    \put (60,3) {$\vn_X$}
    \put (-4,25) {$1$}
    \put (40,-7) {$b$}
    %\put (85,25) {$1-b$}
    %\put (25,55) {$b$}

    \put (108,-6) {$O$}
      \put (210,-6) {$X$}
   \put (160,55) {$V$}
    \put (108,55) {$Y$} 
     \put (92,25) {$1-b$}
    \put (150,-7) {$1$}
     \put (130,55) {$b$}
      \put (120,45) {$\vn_Y$}
    \put (163,35) {$\vn_V$}
    \put (180,3) {$\vn_X$}
  \end{overpic}
    \caption{
    On the left is the ATBD $Q_{P_b}^0:=OXVY$. It is given by the moment polygon for $P_b$ with three nodal rays $\vn_Y=(1,-1),\vn_V=(-1,-1),$ and $\vn_X=(-1,1)$. Note that $Q_{P_b}^0=Q_{P_b}\left((1,1),(3,1),(5,1)\right).$ On the right is the ATBD $Q_{H_b}^0:=OXVY$. It is given by the moment polygon for $H_b$ with the three nodal rays $\vn_Y = (1,-1),\vn_V = (0,-1),$ and $\vn_X = (-2,1)$ inserted. Note that $Q_{H_b}^0=Q_{H_b}\left((1,1),(2,1),(4,1)\right).$
    }
    \label{fig:basediagr}
\end{figure}
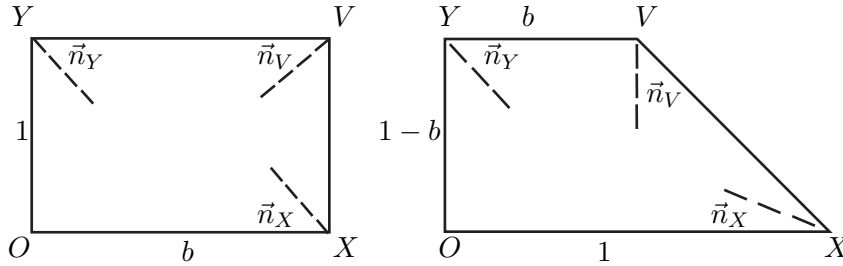

\subsubsection{Mutation on ATBDs}
We begin by describing mutations on ATBDs. Here, we assume the nodal rays are half-lines starting from a vertex extending into the interior of the polygon. Each nodal ray splits the quadrilateral into two pieces. To perform a {\bf mutation} on an ATBD about a nodal ray, we fix the piece which contains the Delzant corner, and then act on the other piece by the unique $SL_2(\Z)$ transformation which has the nodal ray as an eigenvector and aligns the edges of the vertex the nodal ray emanates from. If we assume that the nodal ray does not intersect two vertices, then the resulting ATBD remains a quadrilateral. The ATBD obtained from a mutation corresponds to the same fibration on either $P,H$.

The side of the quadrilateral that the nodal ray intersects depends on the side lengths. For instance, as seen in Figure~\ref{fig:basediagr}, for $Q_{P_b}^0:=OXVY$, the nodal ray $\vn_Y$ extends to intersect: $OX$ if $b>1,$ the vertex $X$ if $b=1,$ and $XV$ if $b<1.$ The recursive formulas for determining the result of the mutation will depend on which side the nodal ray extends to intersect. As such, we use different notation for each possible side the nodal ray can extend to intersect. 
We denote the various mutations as follows: 
\begin{definition} \label{def:mutationI}
\begin{itemlist}
    \item[{(i)}] The mutations about $\vn_Y$ are either an $s$- or a $y$-mutation. If $\vn_Y$ intersects $XV$, this is a $y$-mutation and if $\vn_Y$ intersects $OX$, this is a $s$-mutation. 
    \item[{(ii)}] The mutations about $\vn_X$ are either an $\ov{s}$- or an $x$-mutation. If $\vn_X$ intersects $VY$, this is an $x$-mutation and if $\vn_X$ intersects $OY$, this is a $\ov{s}$-mutation. 
    \item[{(iii)}] The mutations about $\vn_V$ are either an $\ov{x}$- or a $\ov{y}$-mutation. If $\vn_V$ intersects $OX$, this is a $\ov{x}$-mutation and if $\vn_V$ intersects $OY$, this is a $\ov{y}$-mutation. 
\end{itemlist}
\end{definition}
\begin{rmk} 
    Note that for $w=x,y,s$, we have that $w\ov{w}Q=Q$, which explains the notation for the $\ov{w}$-mutations. The use of $x,y$ is for consistency with the previous literature \cite{MMW,M1,REU}. The notation for $s$ will be explained later with its relation to the shift symmetry acting on perfect classes.\footnote{Rather than $\ov{x}$ the papers \cite{MMW,M1,REU} used $v$, but we change that notation here for the inverse notation.} 
\end{rmk}
These mutations can be composed, which we denote as words, so $xyQ$ will refer to first doing a $y$ mutation to $Q$ followed by an $x$-mutation to $yQ$.  As our notation requires the nodal ray to extend to intersect a particular side of the quadrilateral, not all mutation sequences are well-defined for all quadrilaterals. For example, for $Q_{H_b}^0,$ the nodal ray $\vn_V$ is vertical, so $\vn_V$ will always extend to intersect $OX$ implying that for all $b \in (0,1),$ $\ov{x}Q_{H_b}^0$ is well-defined and $\ov{y}Q_{H_b}^0$ is not well-defined. 

\begin{figure}[ht]
\centering
\begin{tikzpicture}[scale=1.05, line cap=round, line join=round]

\usetikzlibrary{calc}

% Colors
\definecolor{myblue}{RGB}{85,95,255}
\definecolor{mygreen}{RGB}{40,130,40}
\definecolor{mypink}{RGB}{255,40,170}

% Main vertices of the quadrilateral
\coordinate (O) at (0,0);
\coordinate (A) at (0,4.2);
\coordinate (B) at (3.45,2.35);
\coordinate (X) at (4.9,0);

% Points on edges
\coordinate (L1) at ($(O)!0.63!(A)$);    % point on left edge
\coordinate (L2) at ($(O)!0.43!(A)$);    % another point on left edge
\coordinate (D1) at ($(O)!0.32!(X)$);    % point on bottom edge
\coordinate (D2) at ($(O)!0.66!(X)$);    % another point on bottom edge
\coordinate (R1) at ($(B)!0.42!(X)$);    % point on right edge
\coordinate (R2) at ($(B)!0.18!(X)$);    % another point on right 

% Point on the top slanted edge A--B where the pink x-ray ends
\coordinate (T1) at ($(A)!0.72!(B)$);

% Outer quadrilateral
\draw[black, very thick] (O) -- (A) -- (B) -- (X) -- cycle;

% ---------------------------
% Rays from A (top-left vertex)
% exactly two: blue and green
% ---------------------------
\draw[myblue, very thick] (A) -- (D1);     % to bottom edge
\draw[mygreen, very thick] (A) -- (R1);    % to right edge

% ---------------------------
% Rays from B (middle-right vertex)
% exactly two: green dashed and pink dashed
% ---------------------------
\draw[mygreen, very thick, dashed, dash pattern=on 12pt off 10pt] (B) -- (L1); 
\draw[mypink, very thick, dashed, dash pattern=on 9pt off 8pt] (B) -- (D2);

% ---------------------------
% Rays from X (bottom-right vertex)
% exactly two: blue dashed and pink solid
% ---------------------------
\draw[myblue, very thick, dashed, dash pattern=on 16pt off 10pt] (X) -- (L2);
%\draw[mypink, very thick] (X) -- (R2);
\draw[mypink, very thick] (X) -- (T1);

% Labels
\node[myblue, scale=1.5] at (-0.32,3.95) {$s$};
\node[mygreen, scale=1.5] at (0.42,4.45) {$y$};

\node[mygreen, scale=1.35] at (3.20,2.8) {$\ov y$};
\node[mypink, scale=1.35] at (3.88,2.2) {$\ov x$};

\node[mypink, scale=1.5] at (5,0.42) {$x$};
\node[myblue, scale=1.5] at (4.62,-0.35) {$\ov s$};

\end{tikzpicture}
\caption{The ATBD mutations at the vertex $Y$ are called $s$ and $y$. The mutation is called $s$ when the ray intersects $\ovr{OX}$ and $y$ when the ray intersects $\ovr{XV}$. Similarly, the mutations at the vertex $V$ are called $\ov{y}$ and $\ov{x}$, and the mutations at the vertex $X$ are called $x$ and $\ov{s}$. For a fixed $b$-value, at most one of the two mutations from each vertex is well-defined. Mutations shown in the same color are inverses of each other.}
    \label{fig:mutationDefs}
\end{figure}
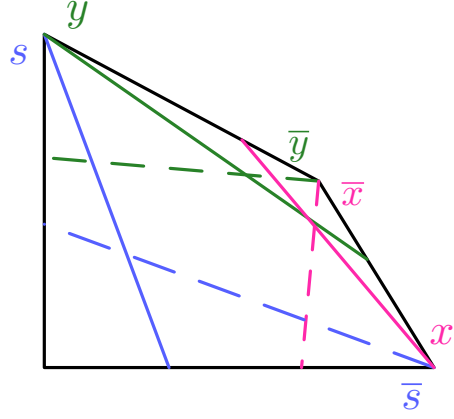
\subsubsection{Triples from an ATBD}
\label{ss:TriplesATBD} 
We now roughly describe how we assign a combinatorial triple to each ATBD. This generalizes the author's work in \cite{M1}, which did this for a family of ATBDs obtained via  mutation from $Q_{H_b}^0.$

Let $\Tt$ be a \textbf{triple} of the form
\[
\mathcal{T} := (\boldsymbol{x}_\lambda, \boldsymbol{x}_\mu, \boldsymbol{x}_\rho) \in \mathbb{Z}^3 \times \mathbb{Z}^3 \times \mathbb{Z}^3,
\]
where for $\al=\la,\mu,\rho$, $\boldsymbol{x}_\al = (p_\al, q_\al, t_\al)$. In Definition~\ref{def:DecQuad} and Lemma~\ref{lem:standFrom}, given a triple satisfying various properties, we define a quadrilateral $Q_{\bullet,b}(\Tt):=OXVY$ for $\bullet=H,P.$ We summarize this construction. 

The quadrilateral $OXVY$ has a Delzant corner at  $O$ and three non-Delzant corners at $Y$, $V$, and $X$. The polygon without the nodal rays
corresponds to a singular toric variety: each non-Delzant corner is a fixed point
of the toric action, locally modeled by a cyclic quotient singularity of type
$\frac{1}{q^2}(1, pq-1)$. From the triple, $(p_\la,q_\la)$ determines the singularity at $Y,$ $(p_\mu,q_\mu)$ at $V,$ and $(p_\rho,q_\rho)$ at $X.$ 
The nodal rays and direction vectors are then chosen so that the quadrilateral is $\Tt$-standard; see Definition~\ref{eq:tStandard}. For instance, the nodal ray at $Y$ is $(q_\la,-p_\la).$ This forces the nodal rays and direction vectors to be determined by Lemma~\ref{lem:standFrom}.

The side lengths are also encoded in the triple. Each element $(p,q,t)$ of a triple is required to lie in a set
$\mathcal{X} \subset \mathbb{Z}^3$ defined by an integral quadratic condition
(see \eqref{eq:Xdef}). In Section~\ref{ss:IntroperfectClasses}, we see for a $(p,q)$-perfect class for $H,P$ there is a $t$ such that $(p,q,t) \in \X$. From $(p,q,t) \in \mathcal{X}$, one derives
integral coordinates $(d;m)$ for $H$ and $(e,f)$ for $P$ defined by:
\begin{align}\label{eq:dmefIntro}
d &= \tfrac{1}{8}(3(p+q)+t), \quad m = \tfrac{1}{8}(p+q+3t), \notag\\ e &= \tfrac{1}{4}(p+q+t), \quad f = \tfrac{1}{4}(p+q-t).
\end{align}
These derived coordinates determine the affine side lengths of
$Q_{\bullet,b}(\mathcal{T})$ as functions of $b$. For example, in the case of $H$,
\[
|OY| = \frac{d_\lambda - m_\lambda b}{q_\lambda}, \qquad
|XV| = \pm\frac{m_\lambda - d_\lambda b}{q_\rho q_\mu},
\]
and the remaining side lengths are similar expressions
(see Definition~\ref{def:DecQuad} and Lemma~\ref{lem:XVVY}).

 Note, for instance, when $b=m_\la/d_\la,$ the length $|XV|=0.$ The points where $b=m_\la/d_\la$ correspond to bifurcation points in the mutation graph. If $Q_{\bullet,b}(\Tt)$ was obtained via a mutation at $\vn_Y$ from $Q_{\bullet,b}(\Tt'),$ then for $b=m_\la/d_\la,$ the nodal ray at $\vn_Y$ extended to intersect $X.$ For $b$ just below or above this critical value, either the mutation $y$ or $s$ is well-defined. As we see in Section~\ref{ss:IntroperfectClasses}, these coefficients $(d;m)$ (resp. $(e,f)$) correspond to homology classes in $H_2(H;\Z)$ (resp. $H_2(P;\Z)$) determining the corresponding $(p,q)$-perfect classes. The bifurcation points of the mutation graph are detecting this data.

\begin{example} \label{ex:baseDiagrams}

    We can check that the base quadrilaterals seen in Figure~\ref{fig:basediagr} can be described by triples. In particular, we have
    \[Q_{H_b}^0=Q_{H_b}\left((1,1,2),(2,1,-1),(4,1,1)\right)\] 
    and 
     \[ Q_{P_b}^0=Q_{P_b}\left((1,1,2),(3,1,0),(5,1,2)\right).\]
     Further each element in the triple associated to $Q_{H_b}^0$ has nonnegative integral solutions $(d;m)$ to \eqref{eq:dmefIntro}, and each element in the triple associated to $Q_{P_b}^0$ has nonnegative integral solutions $(e,f)$ to \eqref{eq:dmefIntro}.
\end{example}

With this example in mind, we define the base triples:
\begin{align}
    \Tt_0^H&:=\left((1,1,2),(2,1,-1),(4,1,1)\right) \\
    \Tt_0^P&:=\left((1,1,2),(3,1,0),(5,1,2)\right).
\end{align}

\noindent\textbf{Mutations on triples.}
On the algebraic side, we define six corresponding mutations of triples. The \(x\)- and \(y\)-mutations were introduced in~\cite{MMW}; they replace one entry of the triple using a linear recursion determined by the other two entries. The \(s\)-mutation, introduced here, cyclically shifts the triple and applies the linear transformation
\[
        S(p,q,t)=(6p-q,p,-t)
\]
to the displaced entry. The $\ov{x},\ov{y},\ov{s}$-mutations are the inverses of these three mutations. The explicit formulas for all six mutations are given in Definition~\ref{def:mutTrip}. In Section~\ref{ss:nodalRays}, we see that these algebraic mutations correspond precisely to the mutations on the ATBD. 

We require that our triples satisfy certain compatibility conditions (called a \textbf{recursive triple}, Definition~\ref{def:genTrip}) which ensure that, after mutation, each entry of the triple remains in the
appropriate set \(\mathcal X_H\) or \(\mathcal X_P\); see
Proposition~\ref{prop:genTrip}. Both base triples $\mathcal{T}^H_0$ and $\mathcal{T}^P_0$ are recursive triples, and Proposition~\ref{prop:genTrip} shows that all mutations of a recursive triple remain recursive. The fact that the triples are recursive also implies that the corresponding quadrilaterals $Q_{\bullet,b}(\Tt)$ are $\T$-standard, see Definition~\ref{eq:tStandard}.

We can now state the theorem on the recursive formulas for the ATBD, which is proved in Section~\ref{ss:ATFandTrip} by Proposition~\ref{prop:ATFMut}.

\begin{thm} \label{thm:ATFMut}
    Let $w$ be a word of $x,y,s,\ov{x},\ov{y},\ov{s}.$ Assume that there is a $b$-value such that $wQ_{H_b}^0$ is well-defined. Then,
    \[ wQ_{H_b}^0=Q_{H_b}(w\T_0^H).\]
    Assume that there is a $b$-value such that $wQ_{P_b}^0$ is well-defined. Then,
    \[ wQ_{P_b}^0=Q_{P_b}(w\T_0^P).\]
\end{thm}

\begin{rmk} \label{rmk:ATFmut} 
    As noted above, the $x,y$-mutations on triples were defined in \cite{MMW} in the classification of perfect classes on $H$. Further, in \cite{MM,MMW} the $S$-symmetry defined as $S(p,q,t):=(6p-q,p,-t)$ is necessary in constructing all perfect classes for $H$, which is related to the $s$-mutation. In particular, $s^3\Tt=S\Tt.$  In \cite{MMW}, it was shown that all perfect classes for $H$ could be obtained by performing the $x,y$-mutations along with the $S,R$-symmetries to a starting triple. Theorem~\ref{thm:ATFMut} implies that, for quadrilateral ATBDs obtained from the starting diagrams $Q_{H_b}^0$ and $Q_{P_b}^0$ by well-defined mutations, the $x,y,s$-mutations and their inverses determine the corresponding mutation structure. This proves Theorem~\ref{thm:A}.
\end{rmk}

Theorem~\ref{thm:ATFMut} concerns ATF mutations on the starting base diagrams $Q_{H_b}^0$ and $Q_{P_b}^0.$ One question for future investigation is: suppose that $Q \subset \R^2_{\ge 0}$ is the base diagram of an ATF for $H$ (resp. $P$) with a smooth corner at $(0,0).$ Is $Q$ necessarily obtained from $Q_{H_b}^0$ (resp. $Q_{P_b}^0$) by a sequence of mutations? Here, we fix $b$ when performing mutations, so if instead we allow $b$ to vary, which quadrilaterals do we obtain that we cannot obtain by fixing $b$? Additionally, we only consider here cases where the ATBD is a quadrilateral, so to understand the full mutation graph of $Q_{P_b}^0$, one would also have to consider cases where the ATBD is a triangle.

There are the corresponding questions about the triples. We will see that the $(p,q)$-perfect classes for $H$ appear in triples mutated from $\T_0^H$, which describe ATBDs, but what other triples appear? Are these triples geometrically meaningful?

\subsection{Perfect Classes and ATBDs} \label{ss:IntroperfectClasses}
We now give some background on perfect classes for $H$ and $P$, and explain how they can be realized in ATBDs. In particular, the derived coordinates $(d;m)$ and $(e,f)$ from \eqref{eq:dmefIntro} that controlled the side lengths in Section~\ref{ss:TriplesATBD} turn out to correspond to homology classes that obstruct symplectic embeddings. This is explained in more detail in Section~\ref{ss:tripPerf}. 

For $M=H,P,$
 a {\bf quasi-exceptional class} $\bE \in H_2(M \# (N+1) \oCP^2,\Z)$ is a class of self-intersection $-1$ with first Chern number $1$. If $\bE$ is quasi-exceptional and also can be represented by a symplectically embedded $S^2,$ we call $\bE$ an {\bf exceptional class}. Let $\mathcal{E}(M)$ denote the set of all exceptional classes for $M$.

For two coprime integers $p,q$, let $W(p,q)=(m_1,\hdots,m_\ell)$ denote the weight sequence of $p/q.$ This is a sequence determined by the continued fraction of $p/q,$ defined in \cite[Definition 1.2.5]{ball}. We say that $\bE \in H_2(M)$ is a {\bf (p,q)-perfect class} if 
$\tilde{\bE}:=\bE-\sum_{i=1}^\ell m_i e_i \in H_2(\tilde{M})$ is exceptional where
$W(p,q)=(m_1,\hdots,m_\ell)$, $\tilde{M}$ is the $\ell$-fold blow up of $M$, and $H_2(\tilde{M})=H_2(M) \oplus \langle e_1,\hdots,e_\ell \rangle$ where $e_1,\hdots, e_\ell$ are the exceptional divisors.  If $\tilde{\bE}$ has self-intersection $-1$ and first Chern number $1$, then we call $\bE$ a $(p,q)$-quasi-perfect class, i.e. $\tilde{\bE}$ isn't necessarily represented by a symplectic $S^2.$

A $(p,q)$-{\bf unicuspidal symplectic (resp. algebraic) curve} in $M$ is a curve that is a smooth, embedded symplectic (resp. complex) surface in $M$ everywhere except at one singular point. Locally near this point, it is modeled by the cusp $x^p = y^q$. In \cite[Theorem G]{MS1}, McDuff and Siegel demonstrated a correspondence between $(p,q)$-perfect classes for $M=H,P$ and rigid $(p,q)$-unicuspidal symplectic curves in $M.$  

In the special case of $H$, in \cite[Theorem F]{McSi}, McDuff and Siegel proved that the rigid unicuspidal symplectic curves in $H$ are also in correspondence with rigid unicuspidal algebraic curves in $H.$ This was shown by realizing the curves in almost toric fibrations. They conjectured in \cite[Conjecture G]{McSi} that this also holds for other del Pezzo surfaces, including $P.$

As seen in Lemma~\ref{lem:AreClasses}, assuming that $(p,q,t) \in \mathcal{X}$ with $p,q>0$, if the solutions $(d;m)$ to \eqref{eq:dmefIntro} are positive integers, then $(d;m)$ is a $(p,q)$-quasi-perfect class for $H$. Similarly, if the solutions $(e,f)$ to \eqref{eq:dmefIntro} are positive integers, then $(e,f)$ is a $(p,q)$-quasi-perfect class for $P.$ 

Utilizing the correspondence shown in \cite[Theorem D]{McSi} and \cite[Theorem G]{MS1}, when the elements of the combinatorial triple describing an ATBD are positive, we can conclude those describe perfect classes. We obtain the immediate corollary of Theorem~\ref{thm:ATFMut}:
\begin{cor} \label{cor:isPerfect}
    Let $w$ be a word of $x,y,s,\ov{x},\ov{y},\ov{s}$ and \[w\Tt_0^\bullet=:\left((p_\la,q_\la,t_\la),(p_\mu,q_\mu,t_\mu),(p_\rho,q_\rho,t_\rho)\right).\] If there is a $b$-value such that $wQ_{\bullet,b}^0$ is well-defined and each $p,q>0$, then
    there exist $(p_\la,q_\la)-,(p_\mu,q_\mu)-,(p_\rho,q_\rho)-$perfect classes for $\bullet=H$ or $\bullet=P$. 
\end{cor}

\begin{example} \label{ex:baseExceptional}
    For the triples $\Tt_0^P$ and $\Tt_0^H$, we have the $(p,q)$-quasi-perfect classes 
    \begin{align*}
      &(1;1;W(1,1)),(1;0;W(2,1)),(2;1;W(4,1)) \quad \text{and}\\  &(1,0;W(1,1)),(1,1;W(3,1)),(2,1;W(5,1)).
    \end{align*}
    In these examples, all of the classes are perfect, which can be seen from Corollary~\ref{cor:isPerfect} and the base diagram $Q_{\bullet,b}^0$. 
    Note that if $p$ is a positive integer, then $W(p,1)=(1^{\times p}).$
\end{example}

The paper \cite{MMW} classifies all perfect classes for $H$ describing them via a set of triples $\T(H)$.
These triples are generated from a collection of seed triples by applying two symmetries
$S,R$— corresponding to a shift and a reflection on triples — together with all possible $x,y$-mutations. The set $\T(P)$ is defined analogously. The precise definition is given in Definition~\ref{def:SRintro}. The shift symmetry satisfies $S = s^3$ (three applications of the $s$-mutation). McDuff, Weiler, and the author showed in \cite{MMW} that the $(p,q)$-perfect classes for $H$
are precisely the elements of the triples in $\T(H).$ The author, Pires and Weiler showed that the set $\T(P)$ consists of $(p,q)$-quasi-perfect classes for $P.$ It is conjectured that $\T(P)$ consists of all the perfect classes for $P.$

In \cite{M1}, we only focused on one family of the perfect classes for $H$, and here, we generalize to all perfect classes. We prove in Section~\ref{ss:tripPerf} the following result: 
\begin{thm} \label{thm:HTrip}
    For each $(p,q)$-perfect class $(d;m)$ in $\mathcal{T}(H)$, there is a word $w$ such that $(p,q)$ is one element of the triple \[w\left((1,1),(2,1),(4,1)\right)=\left((p_\la,q_\la),(p_\mu,q_\mu),(p_\rho,q_\rho)\right).\] Moreover, there is an open interval of \(b\)-values with one endpoint
\(m_\lambda/d_\lambda\) such that \(wQ^0_{H_b}\) is well-defined.
\end{thm} 

For the corresponding statement about $P$, we prove in Section~\ref{ss:tripPerf} the following result: 
\begin{thm} \label{thm:PTrip}
     For each $(p,q)$-quasi-perfect class $(e,f)$ in $\T(P)$, there is some word $w$ such that $(p,q)$ is one element of the triple \[w\left((1,1),(3,1),(5,1)\right)=\left((p_\la,q_\la),(p_\mu,q_\mu),(p_\rho,q_\rho)\right).\] Moreover, there is an open interval of \(b\)-values with one endpoint
\(e_\lambda/f_\lambda\) or \(f_\lambda/e_\lambda\) such that \(wQ^0_{P_b}\) is well-defined.
\end{thm}

The proof given in \cite{MMW} to show that the classes in $\T(H)$ are perfect used many different methods relying heavily on tools from ellipsoid embeddings. By realizing the classes for $\T(P)$ in the ATBD and applying Corollary~\ref{cor:isPerfect}, we conclude that
\begin{cor}
     Each class contained in a triple in $\T(P)$ is perfect. 
\end{cor}

The proofs to Theorems~\ref{thm:HTrip} and \ref{thm:PTrip} involve determining what the word $w$ is for each class as seen in Proposition~\ref{lem:perf4H} and \ref{lem:perf4P}. This implies that the perfect classes in $\T(H)$ and $\T(P)$ can all be obtained via mutation on $\T_0^H$ and $\T_0^P$. The more delicate part about the proof is determining that there are $b$-values such that $wQ_{H_b}^0$ and $wQ_{P_b}^0$ are well-defined; this is done in Section~\ref{ss:bValue}. 

To better explain these proofs and the results, we give an example of the result. This example focuses on $H$. The case of $P$ is more delicate and described in Example~\ref{eg:s^kforPandH2}.
\begin{example} \label{eg:s^kforPandH}

    The classes that appear in the triples $s^k\Tt_0^H$ correspond to perfect classes. Further, there is an interval of $b$-values such that $s^kQ_{H_b}^0$ is defined. By Theorem~\ref{thm:ATFMut}, we have that $s^kQ_{H_b}^0=Q_{H_b}(s^k\Tt_0^H)$. Let 
     \[s^k\Tt_0^H=:\left((p_{k-1},q_{k-1},t_{k-1}),(p_{k},q_{k},t_{k}),(p_{k+1},q_{k+1},t_{k+1})\right).\]
     For each $k,$ there is a $(p_k,q_k)$-perfect class for $H$ with solutions $(d_k;m_k)$ to \eqref{eq:dmefIntro}. Let $\al_k:=m_k/d_k.$ By Lemma~\ref{lem:salphaOr}, for all $k,$
    \[ \al_{2k-1} \leq \al_{2k+1} < \tfrac{1}{3} < \al_{2k+2} \leq \al_{2k}\]
    and $\lim_{k \to \infty} \al_k=1/3.$
As we see in Lemma~\ref{lem:sMutDefinedH}, the sequence of mutations $s^kQ_{H_b}^0$ is defined exactly when $(-1)^{k}\al_{k-1}<(-1)^k b<(-1)^k\al_k.$ Hence, as suggested by Theorem~\ref{thm:HTrip}, we see that $\al_k=m_k/d_k$ is appearing as the endpoints of the intervals for which $s^kQ_{H_b}^0$ is defined, whether it is a lower or upper endpoint depends on whether $\al_k<1/3$ or $\al_k>1/3.$ 

The case for $P$ is more delicate. There is no $b$-value such that $s^kQ_{P_b}^0$ is defined for all $k.$ For the case of $b=1,$ the first nodal ray extends to hit a vertex, so we consider neither $s$ nor $y$ well-defined in the case. The classes that appear in the triples $s^k\Tt_0^P$ are perfect classes for $P.$ Later on, we find a different sequence $w_k$ such that the classes in $s^k\Tt_0^P$ appear in $w_k\Tt_0^P$ and $w_kQ_{P_b}^0$ is well-defined. We defer the details to Example~\ref{eg:s^kforPandH2}.

\end{example}

\subsection{Sharp embeddings} \label{ss:introEmbed}

Results of Casals--Vianna \cite{CV} and Cristofaro-Gardiner et al. \cite{AADT} imply that for $M=P_b,H_b$ if $Q:=OXVY$ is obtained from ATF mutations on $Q_{\bullet,b}^0,$ then for $0<\eps<1,$ there is an embedding
\begin{equation} \label{eq:embedATF}
    (1-\eps)E(|OX|,|OY|) \sembeds M 
    \end{equation} 
where $E(z_1,z_2)$ denotes the $4$-dimensional ellipsoid given by \[E(z_1,z_2):=\{ (\zeta_1,\zeta_2) \in \C^2 \ | \ \pi\left(\frac{|\zeta_1|^2}{z_1}+\frac{|\zeta_2|^2}{z_2}\right)<1\}.\]

By considering sequences of mutations for a fixed $b$-value, we can construct sequences of ellipsoid embeddings. This motivates us to consider the ellipsoid embedding function. 

For a closed $4$-dimensional symplectic manifold $(M,\omega),$ define the ellipsoid embedding function to be 
  \[ c_M(z):=\inf\{ \la \ | \ E(1,z) \sembeds \la M\}\]
  where $z \geq 1$ and $\la M:=(M,\la \omega).$ A lower bound for this function comes from the volume obstruction:
  \[ V_M(z):=\sqrt{\frac{z}{Vol(M)}}\]
  giving $c_M(z) \geq V_M(z)$ and if there is equality for some $z$-value, we call this a {\bf full filling}.

  The function $c_M(z)$ is said to have a {\bf (infinite) staircase} if it has infinitely many nonsmooth points. There has been much research in determining for which $M$ does $c_M(z)$ have a staircase \cite{ball,Usher,ICERM,FM}. In fact, classifying the $b$-values such that $c_{H_b}(z)$ has an infinite staircase was recently completed in a series of papers \cite{ICERM,MM,MMW,MPW}. This classification is intricately tied to the classification of $(p,q)$-perfect classes for $H.$ The case for $c_{P_b}(z)$ has not been classified but various values of $b$ have been considered in prior work including \cite{Usher,REU,FM}. 

It is well-known that the function $c_M(z)$ is determined by considering obstructions from exceptional classes and the volume obstruction \cite{MP,Mell,CG}. For each quasi-exceptional class $\bE,$ there is a piecewise linear function $\mu_{\bE,M}(z)$ such that $c_M(z) \geq \mu_{\bE,M}(z)$ and \begin{equation} \label{eq:funSup} c_M(z)=\max\{\sup_{\bE \in \mathcal{E}(M)}\{\mu_{\bE,M}(z)\},V_M(z)\}.\end{equation} 

\begin{example}
    The monotone example $H_{1/3}$ has an infinite staircase by \cite{AADT}. Each nonsmooth point of the staircase is determined by an exceptional class realizing the supremum in \eqref{eq:funSup}.  The exceptional classes realizing this supremum are the ones considered in Example~\ref{eg:s^kforPandH} for $s^k\Tt_0^H.$ Letting
    \[ s^k\Tt_0^H=((p_{k-1},q_{k-1}),(p_k,q_k),(p_{k+1},q_{k+1})),\]
    the function $c_{H_{1/3}}(z)$ has a non-smooth point at each $p_k/q_k.$ These nonsmooth points accumulate at $z_\infty:=\lim p_k/q_k$. Note that as seen in Example~\ref{eg:s^kforPandH}, the value $b=1/3$ is $\lim_{k \to \infty} m_k/d_k$ where $(d_k;m_k)$ is the solution to \eqref{eq:dmefIntro} for each $(p_k,q_k).$
 
\end{example}

\subsubsection*{Side lengths as obstructions and embeddings}
The embedding in \eqref{eq:embedATF} gives an upper bound for the ellipsoid embedding function for an interval of $z$-values. Following \cite{McSi}, we call the embedding constructed from an ATBD in such a way a {\bf visible embedding.}

On the obstruction side, the side lengths of the quadrilateral are directly related
to the obstruction functions from the perfect classes in the triple. Specifically,
for a quadrilateral $Q_{\bullet,b}(\mathcal{T})$ where
$\mathcal{T} = (\boldsymbol{x}_\lambda, \boldsymbol{x}_\mu, \boldsymbol{x}_\rho)$,
the obstruction function $\mu_{\boldsymbol{x}_\lambda, \bullet}(z)$ from the class
$\boldsymbol{x}_\lambda$ satisfies
\begin{equation}\label{eq:obslengths}
\mu_{\boldsymbol{x}_\lambda, \bullet}(z) = \frac{z}{|OY|}
\end{equation}
for $z$ in an interval containing $p_\lambda/q_\lambda$, where $p_\lambda/q_\lambda$
is the slope of the nodal ray $\vec{n}_Y$ and corresponds to the $z$-value of a
nonsmooth point of $\mu_{\boldsymbol{x}_\lambda, \bullet}(z)$. In other words, near $p_\la/q_\la$, the
reciprocal of the side length $|OY|$ is the slope of the obstruction from the
perfect class at the vertex $Y$. Similarly, $|OX|$ is related to the obstruction
from the class of $S\bm{x}_\rho$ (see Remark~\ref{rmk:obsAndDiagrams} for the precise statement).

Recent work of McDuff--Siegel \cite{McSi} used ATBDs to construct \textbf{visible
obstructions}. Under certain conditions on the ATBD --- specifically, that the
mutation $s$ or $\overline{s}$ is well-defined (see Lemma~\ref{lem:visObs}) --- they
showed that the visible embedding in \eqref{eq:embedATF} is optimal, i.e., it computes a point on the ellipsoid embedding function. Thus, when these conditions hold, the ATBD simultaneously provides the embedding and the obstruction, both encoded in the side lengths.

If $M$ has an infinite staircase where infinitely many of the nonsmooth points come from visible embeddings and obstructions, we say that $M$ has an {\bf ATF-visible staircase}. McDuff-Siegel showed that the monotone examples $H_{1/3}$ and $P_{1}$ have ATF-visible staircases.

\subsubsection*{ATF-visible staircases}

Identifying $b$-values with ATF-visible staircases implies that the ATBDs are computing much of the ellipsoid embedding function. While the perfect classes giving the largest obstructions can be realized in the ATBD, how often do the ATBDs also realize the optimal embeddings? 

In \cite{AADT}, the authors considered the case where $M$ is a closed toric symplectic $4$-manifold. In such a setting, $M$ corresponds to some Delzant polygon $\Omega(M).$ We let $\text{Per}(M)$ be the affine perimeter of $\Omega(M)$ and $\text{Vol}(M)$ be the volume of $M$ normalized to be twice the area of $\Omega(M).$ Note that in our specific examples, we have
\begin{align*} \text{Per}(H_b)&=3-b, \quad \text{Vol}(H_b)=1-b^2, \\
\text{Per}(P_b)&=2+2b, \quad \text{Vol}(P_b)=2b.
\end{align*}
  
  The authors proved \cite[Theorem 1.13]{AADT} that if $M$ has an infinite staircase, then the staircase must accumulate to the solution $z_\infty>1$ to the following quadratic equation
  \begin{equation} \label{eq:accPt}
       \frac{1}{z}+z=\frac{Per(M)^2}{Vol(M)}-2
\end{equation}
and at the accumulation point $z_\infty,$ the function must be unobstructed, that is, \[c_M(z_\infty)=V_M(z_\infty).\]

In the paper \cite{M1}, for an ATBD $Q_b$ on $H_b$, we considered the sequence of ATBDs given by $y^kQ_b:=OXV_kY_k$. Note that $X$ stays constant as after each consecutive mutation, the nodal ray at $Y$ extended to intersect $\ovr{XV}.$ For a fixed $b$-value, the length of $\ovr{XV}$ went to zero, so the limiting shape is a triangle. In the limiting shape, the embedding in \eqref{eq:embedATF} corresponds to some $z$-value such that $c_{H_b}(z)=V_{H_b}(z).$ For the cases in \cite{M1}, this limiting ATBD corresponded to the full filling at the accumulation point. 

In fact, we show in Lemma~\ref{lem:fullFillingAccPt}, if there is a $b$-value such that  \(y^kQ_b\) is well-defined for all \(k\), \(|XV_k|\to0\), and the slope of $\ovr{V_kY_k}$ has irrational limit, then the limiting shape corresponds to a full filling at the accumulation point. By \cite{AADT}, this does not imply there is a staircase, but it shows that in these cases the ATBD recovers the accumulation point. 

However, for a fixed $k$, the ATBD $y^kQ_b$
in \cite{M1} did not correspond to an optimal embedding. Since $c_{H_b}(z)$ is strictly above the volume obstruction before the accumulation point, the ATBDs realizing these optimal embeddings must be quadrilaterals, not triangles.

Here we show this is achievable: under a mild ordering condition on the triple (Definition~\ref{def:correctlyOrdered}), for large enough $k$, the mutation $sxy^kQ_b$ is well-defined, and by the work of McDuff–Siegel \cite{McSi}, each such mutation yields an optimal embedding, producing infinitely many nonsmooth points and hence an ATF-visible staircase. Thus, these quadrilateral ATBDs produce optimal embeddings away from the
volume curve, before the accumulation point.

\begin{definition}\label{def:affineTriangularDegeneration}
Let \(Q_k=OX_kV_kY_k\) be a sequence of quadrilateral ATBDs obtained via
mutation from \(Q_{\bullet,b}^0\) for $\bullet=H,P$. We say that \(Q_k\) {\bf degenerates
affinely to a triangle} if the relevant non-axis affine length tends to
zero. More precisely:
\begin{itemize}
    \item if \(Q_k=y^kw'Q_{\bullet,b}^0=OXV_kY_k\), then this means
    \[
        |XV_k|\to 0;
    \]
    \item if \(Q_k=x^kw'Q_{\bullet,b}^0=OX_kV_kY\), then this means
    \[
        |YV_k|\to 0.
    \]
\end{itemize}
\end{definition}

\begin{thm} \label{thm:ATFvis}
     For each $k \geq 0,$ let $w_k=y^kw'$ where $w'$ is some fixed word. Assume $b$ is such that $w_kQ_{H_b}^0$ (resp. $w_kQ_{P_b}^0$) is well-defined for all $k$ and $w_k\Tt_0^H$ (resp. $w_k\Tt_0^P$) is correctly ordered. If the sequence $w_kQ_{H_b}^0$ (resp. $w_kQ_{P_b}^0$) degenerates affinely to a triangle, then $H_b$ (resp. $P_b$) has an ATF-visible staircase.  
\end{thm}

The triple being correctly ordered is what allows us to determine that a visible mutation sequence is well-defined to prove Theorem~\ref{thm:ATFvis}. A similar result holds for $w_k=x^kw'$. 

We then show that the conditions of this theorem hold for certain cases for $H.$ 
\begin{cor} \label{cor:ATFvis}
    For the $b$ values where $H_b$ has an increasing staircase and not a decreasing staircase, $H_b$ has an increasing ATF-visible staircase. In particular, for infinitely many $b$-values, $H_b$ has an ATF-visible staircase. 
\end{cor}

We will prove this in Corollary~\ref{cor:ATFvis2}. We will explain the difference for increasing and decreasing staircases in Example~\ref{ex:visStair}. 

The classification of \(b\)-values for which \(H_b\) has an infinite staircase is delicate and has involved several different methods. It remains unclear whether this classification can be recovered from ATBDs
alone. More generally, analogous classification questions remain open for families such as \(P_b\) and for blowups of \(\CP^2\) at more than two
points. The results of this paper suggest that ATBDs encode much of the relevant information in these embedding problems.

\subsubsection*{Full fillings}
Letting $\Tt=((p_\la,q_\la,t_\la),(p_\mu,q_\mu,t_\mu),(p_\rho,q_\rho,t_\rho)),$ we have that for the case of $Q_{H_b}(\Tt)=OXVY$, 
\[ |XV|=\pm \frac{m_\la-d_\la b}{q_\rho q_\mu}\]
and for the case of $Q_{P_b}(\Tt)=OXVY,$
\[ |XV|=\pm \frac{e_\la-f_\la b}{q_\rho q_\mu}.\]
Hence, setting $b=m_\la/d_\la$ or $b=e_\la/f_\la,$ the affine lengths will be zero and the resulting mutation sequence will result in a triangle. In particular, the fact that $m_\la/d_\la$ is the endpoint of the interval for Theorem~\ref{thm:HTrip} allows us to give the following result. 

\begin{prop}\label{prop:fullFilling}
    Assume there is a $(p,q)$-perfect class for $H$ with solutions $(d;m)$ to \eqref{eq:dmefIntro}. Then, there exists a full filling 
     \[ E\left(1,\frac{pq-1}{q^2}\right) \sembeds \frac{d}{q}H_{m/d}.\] 
\end{prop}

Note, in the case where $q=1$ this is clear from the toric picture, see Remark~\ref{rmk:fullFilling}. In general, we expect that for $z < z_\infty$, the functions $c_{H_b}(z)$ can be entirely computed from visible embeddings and obstructions from ATBD. If this is the case, the only times before the accumulation point where we get triangle embeddings correspond to these points where $b=m/d$ for some perfect class. 

Results of Cristofaro-Gardiner, McDuff, and the author in \cite[Corollary 1.2.7]{CGMM} imply that for $H_b$ if there is a full filling $E(1,z)$ into $H_b$ or $P_b$ where $z$ is less than the accumulation point, then $z$ must be rational. Their methods used subleading asymptotics of ECH capacities, which is quite different from our methods. Generalizing their result to this specific case, we conjecture that 

\begin{conjecture}
 Let $z_\infty^b$ denote the accumulation point for $H_b.$ The function
 $c_b(z)$ is unobstructed for $z<z_\infty^b$ if and only if $b=m/d$ for some $(d;m;p,q) \in \T(H).$  
\end{conjecture}

Note that \cite[Lemma 2.2.13]{MM} showed that there is no perfect class $(d;m;p,q)$ with $m/d=1/3$. From \cite{AADT}, the staircase for $H_{1/3}$ does not have any points that lie on the volume before the accumulation point. This is in contrast to $P,$ where all the nonsmooth points of $P_1$ before the accumulation point lie on the volume and there are infinitely many perfect classes $(e,f;p,q)$ with $e/f=1.$

It seems the ATBDs are best at detecting embeddings before the accumulation point, and it remains unclear which sharp embeddings after the accumulation point they can detect. 

\section*{Acknowledgements} I am very grateful to Dusa McDuff for many valuable discussions, for her
interest in this project, and for carefully reading multiple sections of this paper. Her comments and suggestions contributed significantly to the final version of the paper. I would also like to thank Tara Holm, Ana Rita Pires, and Morgan Weiler.

 \section{ATF Mutations from triples}
 The goal of this section is to prove Theorem~\ref{thm:ATFMut}, which shows that mutations of the relevant ATBDs are determined by algebraic mutations of triples. We begin in Section~\ref{ss:nodalRays} by studying how mutations affect the nodal rays and edge directions of quadrilaterals; this motivates the definition of mutations of triples. In Section~\ref{ss:trip}, we introduce the algebraic framework of recursive triples, generalizing the generating triples of~\cite{MMW}. Finally, in Section~\ref{ss:ATFandTrip}, we use this framework to compute the affine side lengths and prove Theorem~\ref{thm:ATFMut}.

\subsection{Mutations of Quadrilaterals}\label{ss:nodalRays}

In this section, we consider the nodal rays and direction vectors of ATBDs obtained via mutation from $Q_{H_b}^0$ and $Q_{P_b}^0$. In Definition~\ref{eq:tStandard}, we give a way to describe an ATBD by a triple, and in Lemma~\ref{lem:standard}, we show mutations of quadrilaterals agree with mutations on the triples. We defer the side lengths to Section~\ref{ss:ATFandTrip} as this requires extra care. For more background on ATBDs, we recommend \cite{Evans,LS,S}. 

We begin with some relevant background. Let \(Q=OXVY\) be a \textbf{decorated quadrilateral}, by which we mean a
quadrilateral with one Delzant corner at \(O\) and three nodal rays based at the remaining vertices \(X,V,Y\) where $Q$ is a base diagram for an ATBD for either $H$ or $P$. For two adjacent vertices $W_1,W_2$ of $Q,$ the slope of the edge between $W_1$ and $W_2$ is always rational. Hence, we let $\ovr{W_1W_2}$ denote the primitive lattice direction from $W_1$ to $W_2.$ We always assume $O$ is the Delzant corner at the origin with $\ovr{OX}=(1,0)$ and $\ovr{OY}=(0,1).$ The decorated quadrilateral $Q$ will have three nodal rays pointing into the interior of $Q$ based at each of the vertices $X,V,Y$ denoted by $\vn_X,\vn_V,\vn_Y$, respectively. Each nodal ray splits $Q$ into two pieces. At each non-Delzant vertex $W$, the mutation fixes the side of $Q$ containing the origin and acts on the other side by the matrix $M_W \in SL_2(\Z)$ defined as the unique shear fixing $\vn_W$ that aligns the two edges meeting at $W.$ 

We only consider mutations when the nodal ray does not intersect another vertex of $Q.$ Under this assumption, if $Q$ is obtained by mutation from $Q_{H_b}^0$ or $Q_{P_b}^0$, then it will always have one Delzant corner and three nodal rays with one node on each ray. By work of Symington \cite{S}, this implies that the mutation matrix $M_W$ is conjugate to $\begin{pmatrix} 1 & 1 \\ 0 & 1 \end{pmatrix}.$ Hence, if the nodal ray at some vertex is given by $(q,-p),$ then the mutation matrix must be of the form
\[M_{p,q}:=\begin{pmatrix} 1-pq & -q^2 \\ p^2 & 1+pq \end{pmatrix} .\]
As $M_{p,q}$ aligns the edges at the two vertices, if one edge at the vertex is given by $\begin{pmatrix} 0 \\ -1 \end{pmatrix},$ then the other edge is forced to be in direction $\begin{pmatrix} q^2 \\ 1-pq \end{pmatrix}.$ For each non-Delzant vertex $W$ of $Q$, there is an integral affine change of coordinates with linear part $B_W\in GL_2(\Z)$ such that
\begin{equation} \label{eq:Mmatrix} B_W\ovr{w_1}=\begin{pmatrix} 0 \\ -1 \end{pmatrix}, \quad B_W\ovr{w_2}=\begin{pmatrix} q^2 \\ 1-pq \end{pmatrix},\quad  B_W\vn_W=\begin{pmatrix} q \\ -p \end{pmatrix} \end{equation}
where $\ovr{w_1},\ovr{w_2}$ are the primitive lattice directions of the edges at $W.$ The corresponding quadrilateral with no nodal rays corresponds to a singular toric variety. Each vertex is a fixed point of the toric action and locally described by a cyclic quotient singularity of the form $\tfrac{1}{q^2}(1,pq-1)$ where $p,q$ are determined as in \eqref{eq:Mmatrix}. 

In the local coordinates after applying $B_W,$ the mutation matrix is of the form $M_{p,q}$ described above. Note that for any other vector $v$, we have $M_{p,q}$ is a shear
 \[ M_{p,q}v=v+\det\begin{pmatrix} v_1 & q \\ v_2 & -p \end{pmatrix}\begin{pmatrix} q \\ -p \end{pmatrix}.\]

The quadrilaterals we consider will always have the nodal ray at $Y$ already 
in the standard form $(q_\lambda, -p_\lambda)$ from \eqref{eq:Mmatrix}, so 
$M_Y = M_{p_\lambda, q_\lambda}$. We use a triple $\mathcal{T} = 
((p_\lambda, q_\lambda), (p_\mu, q_\mu), (p_\rho, q_\rho))$ to record the 
singularity type at each non-Delzant vertex: the vertices $Y$, $V$, $X$ carry 
singularities of type $\tfrac{1}{q_\lambda}(1, p_\lambda q_\lambda - 1)$, 
$\tfrac{1}{q_\mu}(1, p_\mu q_\mu - 1)$, and $\tfrac{1}{q_\rho}(1, p_\rho 
q_\rho - 1)$, respectively. The following definition makes precise the 
conditions under which $\mathcal{T}$ fully determines the nodal rays and edge 
directions of $Q$.

\begin{definition} \label{eq:tStandard}
For $\Tt=((p_\la,q_\la),(p_\mu,q_\mu),(p_\rho,q_\rho)) \in \Z^2 \times \Z^2 \times \Z^2$, we say a decorated quadrilateral $OXVY$ is $\Tt$-{\bf standard} if the following conditions hold: 
\begin{itemlist}
    \item[{i)}] $\vn_Y=(q_\la,-p_\la)$
    \item[{ii)}] $M_{p_\la,q_\la}\vn_V=(q_\mu,-p_\mu)$
    \item[{iii)}] $M_{p_\mu,q_\mu}M_{p_\la,q_\la}\vn_X=(q_\rho,-p_\rho).$
    \item[{iv)}] $M_{p_\rho,q_\rho}M_{p_\mu,q_\mu}M_{p_\la,q_\la}=\begin{pmatrix}
        0 & 1 \\ -1 & -6
    \end{pmatrix}$
\end{itemlist}

\end{definition}

\begin{example}
    We note that $Q_{H_b}^0$ is $\Tt_0^H$-standard and $Q_{P_b}^0$ is $\Tt_0^P$-standard. 
\end{example}

\begin{rmk}
    Assume that $Q=OXVY$ is $\Tt$-standard such that $s^2Q$ is well-defined, writing $s^iQ=OX_iV_iY_i$. The first three conditions of a $\Tt$-standard quadrilateral $Q$ imply that 
    \[\vn_{Y_0}=\begin{pmatrix}q_\la \\ -p_\la\end{pmatrix}, \ \vn_{Y_1}=\begin{pmatrix} q_\mu \\ -p_\mu\end{pmatrix}, \ \text{and} \ \vn_{Y_2}=\begin{pmatrix}q_\rho\\-p_\rho\end{pmatrix}.\] As explained in Remark~\ref{rmk:mono}, the fourth condition is necessary given (i)-(iii) for topological reasons to correspond to an ATBD on $P,H$. Furthermore, it is also a necessary condition for the algebraic mutations we defined to coincide with the geometric mutations as seen in the proof of Lemma~\ref{lem:standard}.

\end{rmk}

The next lemma shows that if $Q$ is a $\Tt$-standard ATBD for $P$ or $H$, then the nodal rays and edge directions of $Q$ are determined. 
\begin{lemma} \label{lem:standFrom}
      Let $\Tt=(\bm{x}_\la,\bm{x}_\mu,\bm{x}_\rho)=((p_\la,q_\la),(p_\mu,q_\mu),(p_\rho,q_\rho)) \in \Z^2 \times \Z^2 \times \Z^2$. Assume that $Q:=OXVY$ is $\Tt$-standard and corresponds to an ATBD on $P$ or $H$, then 
    \begin{align*}
    \notag
\ovr{OY}&=\begin{pmatrix}
0 \\ 1
\end{pmatrix}, \quad \quad \quad
\ovr{OX}=\begin{pmatrix}
1 \\ 0
\end{pmatrix},\\
    \vn_Y&=\begin{pmatrix}
q_\la \\ -p_\la
\end{pmatrix}, \quad \vn_{V}=M_Y^{-1}\begin{pmatrix} q_\mu \\ -p_\mu \end{pmatrix},
\quad \vn_X=\begin{pmatrix}
p_\rho-6q_\rho \\ q_\rho
\end{pmatrix}, \\ 
\ovr{VY}&=\begin{pmatrix}
-q_\la^2 \\ q_\la p_\la-1
\end{pmatrix}, \quad
\ovr{XV}=\begin{pmatrix}
1+p_\rho q_\rho-6q_\rho^2 \\ q_\rho^2
\end{pmatrix}
\end{align*}
Further, we have $M_{p_\mu,q_\mu}M_{p_\la,q_\la}\ovr{XV}=\begin{pmatrix}
    0 \\ -1
\end{pmatrix}.$
\end{lemma}

\begin{proof}
    The formulas for $\ovr{OY}$, $\ovr{OX}$, $\vn_Y$, and $\vn_V$ come directly from the definition. By condition iii) of $Q$ being $\Tt$-standard, we have
    \[ M_{p_\mu,q_\mu}M_{p_\la,q_\la}\vn_X=\begin{pmatrix} q_\rho \\ -p_\rho\end{pmatrix},\]
    and by condition iv), we have $M_{p_\mu,q_\mu}M_{p_\la,q_\la}=M_{p_\rho,q_\rho}^{-1} \begin{pmatrix} 0 & 1 \\ -1 & -6 \end{pmatrix}.$
    Hence, 
    \[ \vn_X=\begin{pmatrix} 0 & 1 \\ -1 & -6 \end{pmatrix}^{-1}M_{p_\rho,q_\rho}\begin{pmatrix} q_\rho \\ -p_\rho\end{pmatrix}=\begin{pmatrix} p_\rho-6q_\rho \\ q_\rho \end{pmatrix} \]
    as desired. 
    For the statement about $\ovr{VY},$ this is immediate as we must have $M_{p_\la,q_\la}\ovr{VY}=-\ovr{OY}.$ The statement about $\ovr{XV}$ follows because the computation of $\vn_X$ determines the monodromy matrix at $\vn_X$, which must align $\ovr{OX}$ and $\ovr{XV}$.
    
    For the last statement, as $Q$ is an ATBD for $H,P,$ we can do a mutation about the nodal ray at $Y.$ Let $Q_y=OX_yV_yY_y$ be the mutated diagram. Regardless of where the nodal ray $\vn_Y$ intersects $Q$, we have that $M_{p_\la,q_\la}\ovr{XV}=\ovr{V_yY_y}.$ By condition ii) of Definition~\ref{eq:tStandard}, we have the nodal ray at $Y_y$ is $(q_\mu,-p_\mu).$ This implies that the matrix $M_{p_\mu,q_\mu}$ must align $M_{p_\la,q_\la}\ovr{XV}=\ovr{V_yY_y}$ with $\ovr{OY}$ as desired. 
  
\end{proof}

\begin{definition} \label{def:QTnoL}
Given a triple $\Tt=((p_\la,q_\la),(p_\mu,q_\mu),(p_\rho,q_\rho))$, we let $Q(\Tt)$ denote the decorated quadrilateral whose nodal rays and primitive edge directions are those described in Lemma~\ref{lem:standFrom}. At this stage, $Q(\Tt)$ records only the direction data; the affine side lengths will be specified later. For this to define a quadrilateral, we require $\ovr{VY}$ and $\ovr{XV}$ to be non-parallel.

\end{definition}

\begin{lemma}\label{lem:nodalPos}
    If $q_\la,q_\mu,q_\rho>0,$ then the nodal rays of $Q(\Tt)$ point into the interior of the quadrilateral. 
\end{lemma}
\begin{proof}
    We have that $\vn_Y$ points into the positive quadrant if $q_\la>0.$ To verify that it points into the interior, we also check that $\det(\vn_Y,\ovr{YV})>0$, which follows immediately. Note that the other cases follow from this computation; the statement about the nodal ray at $\vn_V$ follows as the $SL_2(\Z)$ matrix $M_{p_\la,q_\la}$ brings $\vn_V,\ovr{XV}$ to the standard form. Similarly, the statement about $\vn_X$ follows as the $SL_2(\Z)$ matrix $M_{p_\mu,q_\mu}M_{p_\la,q_\la}$ acts similarly near the vertex at $X.$
\end{proof}

Recall from Definition~\ref{def:mutationI}, our notation for the different mutations we can perform on an ATBD. After performing a mutation $w$ to a $\Tt$-standard ATBD $Q,$ we describe in Lemma~\ref{lem:standard} for which triple $\Tt'$, the mutated quadrilateral $wQ$ is $\Tt'$-standard. To do this, we define recursive mutations on triples. These recursive mutations will be considered more in Section~\ref{ss:trip}.
\begin{definition}
\label{def:mutTrip}
 Given a triple $\Tt:=\left((p_\la,q_\la),(p_\mu,q_\mu),(p_\rho,q_\rho)\right):=(\bm{x}_\la,\bm{x}_\mu,\bm{x}_\rho)$ define 
 \[ \nu_\la:=p_\rho q_\mu-q_\rho p_\mu \quad \text{and} \quad \nu_\rho:=p_\mu q_\la-q_\mu p_\la.\] Define the three mutated triples $x\Tt,y\Tt,s\Tt$ as follows: 
  \begin{align*}
    x\Tt &:=\left(\bm{x}_\la,\nu_\la\bm{x}_\mu-\bm{x}_\rho,\bm{x}_\mu 
     \right), \\
    y\Tt &:=\left(\bm{x}_\mu,\nu_\rho \bm{x}_\mu-\bm{x}_\la,\bm{x}_\rho\right), \\
     s\Tt &:=\left(\bm{x}_\mu,\bm{x}_\rho,(6p_\la-q_\la,p_\la)\right).
 \end{align*}
 Further, we define the inverse $\ov{x},\ov{y},\ov{s}$ mutations of $x,y,s$ as: 
\begin{align*}
    \ov{x}\Tt &:=\left(\bm{x}_\la,\bm{x}_\rho,\nu_\la \bm{x}_\rho-\bm{x}_\mu
    \right), \\
    \ov{y}\Tt &:=\left(\nu_\rho \bm{x}_\la-\bm{x}_\mu,\bm{x}_\la,\bm{x}_\rho \right), \\
    \ov{s}\Tt &:=\left((q_\rho,6q_\rho-p_\rho),\bm{x}_\la,\bm{x}_\mu\right). 
\end{align*}
In particular, we have
\[ \ov{x}x\Tt=\Tt, \quad \ov{y}y\Tt=\Tt, \quad \ov{s}s\Tt=\Tt.\]
\end{definition}

\begin{rmk} \label{rmk:mono}
We make a few remarks about condition iv) of being $\Tt$-standard.
\begin{itemlist}
    \item[{i)}] If $Q$ is an ATBD for either $P$ or $H$ and has condition (i)-(iii) of being $\Tt$-standard, then condition iv) follows. To see this, consider performing a nodal trade at each of the four vertices of the moment polytope for either $M=P_b,H_b$. Let $\pi:M \to B$ be the fibration with this base diagram. The preimage of the boundary is a symplectic torus that is Poincare dual to $c_1(M),$ see \cite[Prop 8.2]{S}. Let $\gamma$ denote a counter-clockwise oriented loop in the base near the boundary that crosses each nodal ray, the preimage $\pi^{-1}(\gamma)$ can be viewed as a Lagrangian torus bundle over $\gamma$ or the boundary of the disk normal bundle of $T,$ the smoothed toric boundary. Using this, we can compute the affine monodromy of the torus bundle over $\gamma$ is $\begin{pmatrix} 1 & c_1^2(M) \\ 0&1 \end{pmatrix}$. In the case of $M=P,H,$ we have $c_1^2(M)=8$. Condition iv) of being $\Tt$-standard computes the affine monodromy of the torus bundle given three nodal trades. If we also did a nodal trade at the origin, the mutation matrix at the origin is $\begin{pmatrix} 2 & -1 \\ 1 & 0 \end{pmatrix}.$ As such, we must have that
    \[\begin{pmatrix} 2 & -1 \\ 1 & 0 \end{pmatrix} M_{p_\rho,q_\rho}M_{p_\mu,q_\mu}M_{p_\la,q_\la}=\begin{pmatrix} 1 & 8 \\ 0 & 1 \end{pmatrix}, \]
    which implies condition iv):
    \[M_{p_\rho,q_\rho}M_{p_\mu,q_\mu}M_{p_\la,q_\la}=\begin{pmatrix} 0 & 1 \\ -1 & -6 \end{pmatrix}.\] 
    \item[{ii)}]   
    Let $S=\begin{pmatrix} 6 & -1 \\ 1 & 0 \end{pmatrix}$, so $s\Tt=((p_\mu,q_\mu),(p_\rho,q_\rho),S(p_\la,q_\la)).$ We can recover this $S$ from condition iv) of Definition~\ref{eq:tStandard}. We have the matrices in Definition~\ref{eq:tStandard} are acting on the nodal rays coordinates of the form $(q,-p).$ The matrix $J=\begin{pmatrix} 0 & -1 \\ 1 & 0 \end{pmatrix}$ brings $(q,-p)$ to $(p,q).$ Note that 
    \[ J^{-1}\begin{pmatrix}
        0 & 1 \\ -1 & -6
    \end{pmatrix}J=-S.\] The negative sign is explained as after doing an $s$-mutation the new nodal ray at $X_s$ is in direction $-(q_\la,-p_\la)$. 
\end{itemlist}

\end{rmk}

To prove mutations on triples describe the mutations on quadrilaterals, we note this basic result about shear matrices, which we will regularly use. 
\begin{lemma} \label{lem:conjugation}
    If $M_{p,q}$ is the shear fixing $(q,-p),$ then for any shear matrix $M_{p',q'},$ we have
   \[ M_{p,q}M_{p',q'}M^{-1}_{p,q}=M_{p'',q''}\]
   where $(q'',-p'')=M_{p,q}(q',-p').$
    \end{lemma}
\begin{proof}
    This follows immediately by definition of the shear matrices. 
\end{proof}

We now show the recursive definition of triples in Definition~\ref{def:mutTrip} aligns with the mutation on ATBD. 
\begin{lemma} \label{lem:standard}
Suppose that $Q=OXVY$ is $\Tt$-standard for $\Tt=((p_\la,q_\la),(p_\mu,q_\mu),(p_\rho,q_\rho))$. For $w=y,x,s,\ov{x},\ov{y},\ov{s},$ assuming that $wQ$ is well-defined, we have $wQ$ is $w\Tt$-standard.

\end{lemma}
\begin{proof}

The claim about the $y$-mutation is the most straightforward, and similar to the other claims, so we omit the details of that case.

We begin by checking that $xQ = OX_xV_xY$ is $x\mathcal{T}$-standard. 
As the left element of $x\mathcal{T}$, the vertex $Y$, and $\ovr{OY}$ all stay fixed, condition i) follows 
immediately. For condition ii), we check that
\begin{equation} \label{eq:mixed} 
    M_{p_\lambda,q_\lambda}\left(-\vec{n}_X\right) = 
    \begin{pmatrix} q_{x\mu} \\ -p_{x\mu} \end{pmatrix}.
\end{equation}
By using condition iii) that $Q$ is $\mathcal{T}$-standard, we have that
\begin{align*} M_{p_\lambda,q_\lambda}\left(-\vec{n}_X\right) &= 
   -M_{p_\mu,q_\mu}^{-1}\begin{pmatrix} q_\rho \\ -p_\rho \end{pmatrix} \\
&=-\bigl(\begin{pmatrix} q_\rho \\ -p_\rho \end{pmatrix} 
   - \det\begin{pmatrix}
       q_\rho & q_\mu \\ -p_\rho & -p_\mu
   \end{pmatrix} \begin{pmatrix} q_\mu \\ -p_\mu \end{pmatrix}  \bigr) \\
&=-\bigl( \begin{pmatrix} q_\rho \\ -p_\rho \end{pmatrix} 
   - \nu_\lambda \begin{pmatrix} q_\mu \\ -p_\mu \end{pmatrix} 
    \bigr) \\
   &=-\begin{pmatrix} -q_{x\mu} \\ p_{x\mu} \end{pmatrix}
\end{align*}
gives \eqref{eq:mixed}. 
For condition iii), the new nodal ray at $X_x$ is given by $M_X\vec{n}_V$ 
where $M_X$ is the shear that fixes $\vec{n}_X$. Hence, we must check that
\[ M_{p_{x\mu},q_{x\mu}}M_{p_\lambda,q_\lambda}M_X\vec{n}_V = 
   \begin{pmatrix} q_\mu \\ -p_\mu \end{pmatrix}.\]
Using condition ii) of $Q$ being $\mathcal{T}$-standard, we have that 
$\vec{n}_V = M_{p_\lambda,q_\lambda}^{-1}\begin{pmatrix} q_\mu \\ -p_\mu 
\end{pmatrix}$, so we obtain that
\[ M_{p_\lambda,q_\lambda}M_X\vec{n}_V = 
   \begin{pmatrix} q_\mu \\ -p_\mu \end{pmatrix} 
   + \det(\vec{n}_V, \vec{n}_X)\, M_{p_\lambda,q_\lambda}\vec{n}_X.\]
Since \(M_{p_\lambda,q_\lambda}\in SL_2(\mathbb Z)\), we have
\[
\det(\vec n_V,\vec n_X)
=
\det\left(
M_{p_\lambda,q_\lambda}\vec n_V,
M_{p_\lambda,q_\lambda}\vec n_X
\right)=\det \begin{pmatrix}
    q_\mu & -q_{x\mu}\\ -p_\mu & p_{x\mu}
\end{pmatrix}=-\nu_\la.
\]
where the second equality uses condition ii) and \eqref{eq:mixed}. 
Applying 
$M_{p_{x\mu},q_{x\mu}}^{-1}$, we get
\begin{align*} 
    M_{p_{x\mu},q_{x\mu}}^{-1}\left(\begin{pmatrix} q_\mu \\ -p_\mu \end{pmatrix} 
    + \nu_\lambda\begin{pmatrix} q_{x\mu} \\ -p_{x\mu} \end{pmatrix}\right) 
    &= \begin{pmatrix} q_\mu \\ -p_\mu \end{pmatrix} 
    + \left(-\det\begin{pmatrix} q_\mu & q_{x\mu} \\ 
    -p_\mu & -p_{x\mu}\end{pmatrix} 
    + \nu_\lambda\right)\begin{pmatrix} q_{x\mu} \\ -p_{x\mu} \end{pmatrix} \\
    &= \begin{pmatrix} q_\mu \\ -p_\mu \end{pmatrix}
\end{align*}
as desired. Condition iv) follows as Lemma~\ref{lem:conjugation} implies 
that $M_{p_{x\mu},q_{x\mu}} = M_{p_\mu,q_\mu}^{-1}M_{p_\rho,q_\rho}
M_{p_\mu,q_\mu}$, so
\[ M_{p_\mu,q_\mu}M_{p_{x\mu},q_{x\mu}}M_{p_\lambda,q_\lambda} = 
   M_{p_\rho,q_\rho}M_{p_\mu,q_\mu}M_{p_\lambda,q_\lambda}.\]

We now verify that $sQ$ is $s\mathcal{T}$-standard. Unlike the other cases, 
here we must use condition iv) to verify the first three conditions. Let 
$sQ = OX_sV_sY_s$; we have that 
$\vec{n}_{X_s} = -\begin{pmatrix} q_\lambda \\ -p_\lambda \end{pmatrix}$, 
$\vec{n}_{Y_s} = M_{p_\lambda,q_\lambda}\vec{n}_V$, and 
$\vec{n}_{V_s} = M_{p_\lambda,q_\lambda}\vec{n}_X$. As $Q$ satisfies 
conditions ii) and iii) of being $\mathcal{T}$-standard, conditions i) and 
ii) hold for $sQ$ being $s\mathcal{T}$-standard. For condition iii), we 
want to check that 
$M_{p_\rho,q_\rho}M_{p_\mu,q_\mu}\begin{pmatrix}-q_\lambda \\ 
p_\lambda\end{pmatrix} = \begin{pmatrix}p_\lambda \\ 
-6p_\lambda + q_\lambda\end{pmatrix}$, but this holds because the 
left-hand side equals
\[ M_{p_\rho,q_\rho}M_{p_\mu,q_\mu}M_{p_\lambda,q_\lambda}
   \begin{pmatrix} -q_\lambda \\ p_\lambda \end{pmatrix} = 
   \begin{pmatrix} 0 & 1 \\ -1 & -6 \end{pmatrix}
   \begin{pmatrix} -q_\lambda \\ p_\lambda \end{pmatrix} = 
   \begin{pmatrix} p_\lambda \\ -6p_\lambda + q_\lambda \end{pmatrix}\]
as desired. Condition iv) follows again by Lemma~\ref{lem:conjugation}. 
For the inverse mutations, the result follows because $\ov{w}wQ = Q$, so 
if $wQ$ is $w\mathcal{T}$-standard then applying $\ov{w}$ recovers $Q$ 
being $\mathcal{T}$-standard.

\end{proof}

\begin{prop} \label{prop:nodalDir}
    Let $w$ be a word in the mutations $x,y,s,\ov{x},\ov{y},\ov{s}$. If $wQ_{\bullet,b}^0$ is well-defined, then the nodal rays and primitive edge directions of $wQ_{\bullet,b}^0$ are given by the data $Q(w\Tt_0^\bullet)$ from Definition~\ref{def:QTnoL}.
\end{prop}
\begin{proof}
    This follows immediately from Lemma~\ref{lem:standard} and Lemma~\ref{lem:standFrom}. 
\end{proof}

\subsection{Recursive Triples}\label{ss:trip}

As we aim to understand the mutation structure of $Q_{H_b}^0$ and $Q_{P_b}^0$ for fixed $b$, we must also track how affine side lengths change under mutation. Proposition~\ref{prop:nodalDir} shows that the nodal rays and edge directions are described by algebraic mutations of triples. In this subsection, we study the algebraic properties of these triples. The resulting notion of a recursive triple will later connect the mutation data to perfect classes in Section~\ref{ss:tripPerf} and will allow us to compute the affine side lengths.

We call these recursive triples, which are defined in Definition~\ref{def:genTrip}. For a recursive triple, these $(p,q)$-coordinates will also indirectly determine the side lengths via the $(d;m)$ and $(e,f)$-coordinates defined in \eqref{eq:dmefIntro}. This generalizes work of the author, McDuff, and Weiler in \cite{MMW,MM} about triples of quasi-perfect classes for $H$. We begin by developing the algebraic machinery to give the definition.

Define $\mathcal{X}:=\{ \bm{x} \in \Z^3 \ | \ \bm{x}^TA\bm{x}=8\}$ where \begin{equation} \label{eq:Xdef} A=\begin{pmatrix}
    -1 & 3 & 0 \\ 
    3 & -1 & 0 \\
    0 & 0 & 1
\end{pmatrix}, \quad \bm{x}:=\begin{pmatrix}
    p \\ q \\ t
\end{pmatrix}. \end{equation}

Note that the matrix $A$ is symmetric.
For a vector $(p,q,t) $, define $(d;m)$ and $(e,f)$ as in \eqref{eq:dmefIntro}.
We define the subsets $\X_H,\X_P$ of $\X$ as follows:  
\[\X_H:=\{ (p,q,t) \in \X \ | \ (d;m) \in \Z^2 \}
\]
and 
\[ \X_P:=\{ (p,q,t) \in \X \ | \ (e,f) \in \Z^2 \}.
\]
We often let $\X_\bullet$ refer to either $\X_P$ or $\X_H$. We say that an element $(p,q,t) \in \mathcal{X}_H$ (resp. $\in \mathcal{X}_P$) is {\bf positive} if $d, p,q$ (resp. $e,f,p,q$) are all positive and $m$ is nonnegative. We now extend the Definition~\ref{def:mutTrip} to triples with the $t$-coordinate \[\Tt:=(\bm{x}_\la,\bm{x}_\mu,\bm{x}_\rho):=((p_\la,q_\la,t_\la),(p_\mu,q_\mu,t_\mu),(p_\rho,q_\rho,t_\rho)).\] The $x,y,\ov{x},\ov{y}$-mutations will be the exact same where we extend the recursion given in Definition~\ref{def:mutTrip} to the $t$-coordinate as well. For $s,\ov{s},$ we have
\begin{align*}
    s\Tt&:=(\bm{x}_\mu,\bm{x}_\rho,(6p_\la-q_\la,p_\la,-t_\la))\\
    \ov{s}\Tt&:=((q_\rho,6q_\rho-p_\rho,-t_\rho),\bm{x}_\la,\bm{x}_\mu)
\end{align*}

In the rest of this section, we give conditions on a triple $\Tt=(\bm{x}_\la,\bm{x}_\mu,\bm{x}_\rho) \in \X^3_\bullet$ such that $w\Tt$ remains in $\X_\bullet^3$ for $w=x,y,s,\ov{x},\ov{y},\ov{s}$. The fact that triples obtained via mutation from $\Tt_0^H$ and $\Tt_0^P$ are in $\X_\bullet^3$ is useful for showing the affine length formulas work as stated.

We begin with the $s$-mutation. Note that the $s$-mutation performs a recursion by $6$ to $\bm{x}_\la.$ We define the matrix
\[S:=\begin{pmatrix}
    6 & -1 & 0 \\ 1 & 0 & 0 \\ 0 & 0 & -1
\end{pmatrix},\]
so \[s(\bm{x}_\la,\bm{x}_\mu,\bm{x}_\rho)=(\bm{x}_\mu,\bm{x}_\rho,S\bm{x}_\la).\] In the classification of perfect classes for $H$, we discuss the role of the matrix $S$ in Section~\ref{ss:tripPerf}.
The next lemma proves that if $\Tt \in \X_\bullet^3,$ then $s\Tt \in \X_\bullet^3.$
\begin{lemma} \label{lem:symTriple} \begin{itemlist}
    \item[{(i)}] If $\bm{x}:=(p,q,t) \in\X_H$, then $S\bm{x} \in\X_H.$ Further, if $(d,m)$ satisfy \eqref{eq:dmefIntro} for $\bm{x}$, then $(3p-d,p-m)$ satisfy \eqref{eq:dmefIntro} for $S\bm{x}$.
    \item[{(ii)}] If $\bm{x} \in\X_P$, then $S\bm{x} \in\X_P.$ Further, if $(e,f)$ satisfy \eqref{eq:dmefIntro} for $\bm{x}$, then $(2p-e,2p-f)$ satisfy \eqref{eq:dmefIntro} for $S\bm{x}$. 
\end{itemlist}
\end{lemma}
\begin{proof}
For either statement, we can check that $\bm{x}^TA\bm{x}=(S\bm{x})^TA(S\bm{x})$ implying that $S\bm{x} \in\X$. 
For (i), assuming $\bm{x}\in\X_H$, let $(d,m)$ satisfy \eqref{eq:dmefIntro} for $\bm{x}$.  It is a straightforward computation that $(3p-d,p-m)$ satisfy \eqref{eq:dmefIntro} for $S\bm{x}$. By assumption, $(d,m)$ are integers, so $(3p-d,p-m)$ also are integers, implying $S\bm{x} \in\X_H$. 
The case of (ii) follows similarly. 
\end{proof}

We now establish some criterion which guarantees that $x\Tt \in \X^3$ and $y\Tt \in \X^3$ assuming $\Tt \in \X^3$. This criterion is a generalization of \cite[Section 3.1]{MM} and \cite[Section 2.1]{MMW}.

\begin{definition}
    We say that two integral vectors $\bm{x}_0=(p_0,q_0,t_0),\bm{x}_1=(p_1,q_1,t_1) \in \X$ 
    are $\nu$-compatible for an integer $\nu$ if 
    \[ \bm{x}_1^TA\bm{x}_0=4\nu.\]  
\end{definition}

\begin{lemma}\cite[Lem 3.1.2]{MM}
    Suppose that $\bm{x}_0,\bm{x}_1$ are integral vectors $(p_i,q_i,t_i)$ for $i=0,1$. If $\bm{x}_0,\bm{x}_1 \in \mathcal{X}$ are $\nu$-compatible, 
         then $\bm{x}_2:=\nu \bm{x}_1-\bm{x}_0,\bm{x}_1$ are also $\nu$-compatible and $\bm{x}_2 \in\X.$  
\end{lemma}
\begin{proof}
    Since $A$ is symmetric, \[\bm{x}_2^TA\bm{x}_2=\nu^2\bm{x}_1^TA\bm{x}_1-2\nu \bm{x}_1^TA\bm{x}_0+\bm{x}_0^TA\bm{x}_0=8\nu^2-8\nu^2+8=8\]
    implying that $\bm{x}_2 \in\X.$ Additionally, we have that
    \[ (\nu \bm{x}_1-\bm{x}_0)^TA\bm{x}_1=8\nu-\bm{x}_0^TA\bm{x}_1=4\nu.\]
\end{proof}

\begin{cor}\cite[Cor 3.1.3]{MM} \label{cor:recurExtend}
    Any two integral vectors $\bm{x}_i=(p_i,q_i,t_i) \in\X_\bullet$, $i=0,1$ that are $\nu$-compatible can be extended to a sequence $\bm{x}_i, i \geq 0,$ using recursion parameter $\nu$, where each successive pair is $\nu$-compatible and $\bm{x}_i \in\X_\bullet$ for all $i \geq 0$. Further, the values $(d_i,m_i)$ that satisfy \eqref{eq:dmefIntro} or $(e_i,f_i)$ that satisfy \eqref{eq:dmefIntro} for each $\bm{x}_i$ satisfy the same recursion. 
\end{cor}

The notion of recursive triple we define next will show that as we iterate various mutations every point in the triple will continue to be in $\X^3.$ The first two conditions ensure that $\bm{x}_{y\mu},\bm{x}_{x\mu} \in \X.$ The third condition ensures we can iterate the $x,y$-mutations. In other words, it guarantees that $\bm{x}_{xy\mu},\bm{x}_{yx\mu} \in \X$ as well.

  \begin{definition} \label{def:genTrip}
        A {\bf recursive triple} is a triple $(\bm{x}_\la,\bm{x}_\mu,\bm{x}_\rho) \in \mathcal{X}_\bullet \times \mathcal{X}_\bullet \times \mathcal{X}_\bullet$ such that for some $\eps \in \{-1,1\}$ the following conditions hold: 
        \begin{itemlist}
            \item[{(i)}] $\bm{x}_\la,\bm{x}_\mu$ are $\eps t_\rho$-compatible and $\eps t_\rho=p_\mu q_\la-p_\la q_\mu$ 
            \item[{(ii)}] $\bm{x}_\mu,\bm{x}_\rho$ are $\eps t_\la$-compatible and $\eps t_\la=p_\rho q_\mu-p_\mu q_\rho$
            \item[{(iii)}] $\bm{x}_\la,\bm{x}_\rho$ are $(t_\la t_\rho-\eps t_\mu)$-compatible and $t_\la t_\rho-\eps t_\mu=p_\rho q_\la-p_\la q_\rho$
        \end{itemlist}
    If $\eps=1,$ we call the triple {\bf sign-matching}, and if $\eps=-1$, we have the triple {\bf sign-alternating}.

    \end{definition}

     Note that if $\Tt$ is sign-matching, the $x$-mutation (resp. $y$-mutation) is a recursion by $t_\la$ (resp. $t_\rho$). On the other hand, if $\Tt$ is sign-alternating, the $x$-mutation (resp. $y$-mutation) is a recursion by $-t_\la$ (resp. $-t_\rho$). 

\begin{rmk}
    This notion of a recursive triple is a generalization of the notion of generating triple from \cite{MMW}. In \cite{MMW}, the only mutations that were considered were the $x,y$-mutations. Additionally, in \cite{MMW}, we were not considering triples obtained via mutation $\Tt_0^H,$ but a more restrictive subset of triples. See Example~\ref{ex:differentNotation} for some context of why we require a more general notion of triples from the recursion perspective. 
    
    Additionally, not all triples that correspond to ATBD mutated from $Q_{H_b}^0$ and $Q_{P_b}^0$ are generating triples in the sense of \cite{MMW}. For instance, in \cite{MMW}, the authors required
    \[ p_\la/q_\la<p_\mu/q_\mu<p_\rho/q_\rho\]
    and $(p_\bullet,q_\bullet)\geq(0,0)$, but this is not the case for triples that describe ATBD. See Example~\ref{ex:unusualTrip} for such examples.  
\end{rmk}

    \begin{example} \rm \label{ex:differentNotation}

In \cite{MMW}, the notion of generating triple assumed that $t>0$ and used $\eps \in \{\pm 1\}$ to determine the sign of $t$ in finding $(d;m)$ in \eqref{eq:dmefIntro}. Note that in \cite[Lemma 2.2.1]{MM}, it was shown that given a $(p,q,\eps t) \in \X$, there was at most one value of $\eps \in \{\pm 1\}$ such that $(d;m)$ are integers. For the case of $P$, this is not the case as $(e,f)$ is a solution with $\eps=1$ and $(f,e)$ with $\eps=-1$ are solutions to \eqref{eq:dmefIntro}. This is one motivation to incorporate the sign in the $t$-variable. 

We must also incorporate the sign to ensure the recursion in performing $x,y$-mutations extends from the $(p,q)$-coordinates to the $t$-coordinate. If we consider the $(p,q)$-coordinates of the starting triple, but we instead have $t>0$ for each element, we have the triple \[(\bm{x}_\la,\bm{x}_\mu,\bm{x}_\rho):=((1,1,2),(2,1,1),(4,1,1)).\] We have $\bm{x}_\la^T A\bm{x}_\rho=12,$ which does equal $4(p_\rho q_\la-p_\la q_\rho),$ so $\bm{x}_\la,\bm{x}_\rho$ are $(p_\rho q_\la-p_\la q_\rho)$-compatible, but we do not have that \[t_\la t_\rho-t_\mu=p_\rho q_\la-p_\la q_\rho,\] so they are not $(t_\la t_\rho-t_\mu)$-compatible. Note that if we perform \[x((1,1,2),(2,1,1),(4,1,1))=((1,1,2),(0,1,1),(2,1,1))\]
and we can check that for $\bm{x}_{x\mu}^TA\bm{x}_{x\mu} \neq 8$, i.e. $\bm{x}_{x\mu} \not\in \X.$ As such, we do not insist $t$ is positive. The choice is determined by which sign gives integer values to \eqref{eq:dmefIntro}. Then, we have the starting triple
$((1,1,2),(2,1,-1),(4,1,1)),$ and now we do have that 
\[ t_\la t_\rho-t_\mu=p_\rho q_\la-p_\la q_\rho.\] 
Then, we can check that 
\[ x((1,1,2),(2,1,-1),(4,1,1))=((1,1,2),(0,1,-3),(2,1,-1)),\]
and now we do have $\bm{x}_{x\mu} \in \X.$
\end{example}

The main result of this section is: 
    \begin{prop} \label{prop:genTrip} \begin{itemlist}
        \item[{(i)}]  If $\Tt$ is a sign-matching (resp. sign-alternating) recursive triple, then $x\Tt,y\Tt,\ov{x}\Tt,\ov{y}\Tt$ are all sign-matching  (resp. sign-alternating) recursive triples. 
        \item[{(ii)}] If $\Tt$ is a sign-matching (resp. sign-alternating) recursive triple, then $s\Tt,\ov{s}\Tt$ are both sign-alternating (resp. sign-matching) recursive triples. 
    \end{itemlist}
    \end{prop}

The following corollary states that the triples that describe mutations of $Q^0_{\bullet,b}$ are always recursive triples. The corollary is immediate as $\Tt_0^H$ and $\Tt_0^P$ are both recursive triples and then applying Proposition~\ref{prop:nodalDir}: 
    \begin{cor}
        All triples obtained by mutation from $\Tt_0^H$ and $\Tt_0^P$ are recursive triples. Consequently, whenever a mutation sequence is geometrically well-defined for $Q_{\bullet,b}^0$, the resulting quadrilateral is $\Tt$-standard for the corresponding recursive triple $\Tt$. 
    \end{cor}

    We give some preliminary results before giving the proof of Proposition~\ref{prop:genTrip}. The first result is similar to what was referred to as adjacency in \cite{MMW}. 

\begin{lemma} \label{lem:7cond}
    If $\bm{x}_1=(p_1,q_1,t_1),\bm{x}_2=(p_2,q_2,t_2)$ are $(p_2q_1-p_1q_2)$-compatible, then the following identities hold:
    \begin{itemlist}
        \item[{(i)}] $t_1t_2=p_2(p_1+q_1)+q_2(q_1-7p_1)$
        \item[{(ii)}] If it is also the case that $p_2q_1-p_1q_2=t_1t_2+k,$ then 
        \begin{equation}
            k=6p_1q_2-q_1q_2-p_1p_2. 
        \end{equation} 
    \end{itemlist}
\end{lemma}
\begin{proof}
    The first is an immediate computation simplifying $\bm{x}_1^TA\bm{x}_2=4(p_2q_1-p_1q_2).$ The second follows from substituting the expression for $t_1t_2$ in (i) to $p_2q_1-p_1q_2=t_1t_2+k$ and simplifying.  
\end{proof}

  \begin{lemma} \label{lem:md13}
    For $(p,q,t) \in\X_H$ with $d,m$ satisfying \eqref{eq:dmefIntro} both positive and $p \geq q,$ we have
    \[ \frac{m}{d}>\frac{1}{3} \iff t>0 \quad \text{and} \quad \frac{m}{d}<\frac{1}{3} \iff t<0.\]

    For $(p,q,t) \in \X_P$ with $(e,f)$ satisfying \eqref{eq:dmefIntro} both positive, we have 
   \[ \frac{e}{f}>1 \iff t>0 \quad \text{and} \quad \frac{e}{f}<1 \iff t<0.\]     
\end{lemma}
\begin{proof}
    By \eqref{eq:dmefIntro}, 
    \[ \frac{m}{d}=\frac{p+q+3t}{3(p+q)+t} >\frac{1}{3} \implies 9 t>t. \]

      The statement about $e/f$ is similar as \eqref{eq:dmefIntro} gives that 
    \[ \frac{e}{f}=\frac{p+q+ t}{p+q- t}.\]
\end{proof}

\begin{lemma}\label{lem:recurDif}\cite[Cor 3.15]{MM}
    Let $p_k,q_k$,$k \geq 0$ be increasing sequences that both satisfy the recursion $r_{k+1}=tr_k-r_{k-1}$ where $p_k,q_k$ are positive for $k \geq 0.$ Then the ratios $(p_k/q_k)_{k \geq 1}$ form a monotone sequence that is strictly increasing if $p_1q_0-p_0q_1>0$ and is strictly decreasing if $p_1q_0-p_0q_1<0$. 
\end{lemma}
\begin{proof}
    We have that
    \[ p_{k+1}q_k-p_kq_{k+1}=(tp_k-p_{k-1})q_k-p_k(tq_k-q_{k-1})=p_kq_{k-1}-p_{k-1}q_k\]
    is constant in $k.$ The conclusion follows as long as the entries are positive. 
\end{proof}

We now prove Proposition~\ref{prop:genTrip}.
  \begin{proof}[Proof of Proposition~\ref{prop:genTrip}]
    For both statements, the conditions of the recursive triple ensure that by Lemma~\ref{lem:symTriple} and Corollary~\ref{cor:recurExtend}, we have $w\Tt \in\X_\bullet \times\X_\bullet \times\X_\bullet$ for $\bullet=P,H$ for $w=x,y,s,\ov{x},\ov{y},\ov{s}.$ Hence, it remains to check the other conditions of being a recursive triple.

We prove the result for $y\Tt$ first; the other cases are similar. Let $\eps=1$ if $\Tt$ is sign-matching and $\eps=-1$ if $\Tt$ is sign-alternating. Then
\[
\eps t_\rho=p_\mu q_\lambda-p_\lambda q_\mu,
\qquad
\eps t_\lambda=p_\rho q_\mu-p_\mu q_\rho,
\]
and
\[
t_\lambda t_\rho-\eps t_\mu=p_\rho q_\lambda-p_\lambda q_\rho.
\]
By definition,
\[
y\Tt=(\bm{x}_\mu,\bm{x}_{y\mu},\bm{x}_\rho),
\qquad
\bm{x}_{y\mu}=\eps t_\rho\bm{x}_\mu-\bm{x}_\lambda.
\]

We first check condition (i). By Corollary~\ref{cor:recurExtend}, $\bm{x}_\mu,\bm{x}_{y\mu}$ are $\eps t_\rho$-compatible, since $\bm{x}_\lambda,\bm{x}_\mu$ are $\eps t_\rho$-compatible. Additionally,
\begin{align*}
p_{y\mu}q_\mu-p_\mu q_{y\mu}
&=(\eps t_\rho p_\mu-p_\lambda)q_\mu
-p_\mu(\eps t_\rho q_\mu-q_\lambda)\\
&=p_\mu q_\lambda-p_\lambda q_\mu\\
&=\eps t_\rho.
\end{align*}
Hence condition (i) holds for $y\Tt$.

Next, we check condition (ii). We have
\begin{align*}
\bm{x}_{y\mu}^TA\bm{x}_\rho
&=(\eps t_\rho \bm{x}_\mu-\bm{x}_\lambda)^TA\bm{x}_\rho\\
&=\eps t_\rho \bm{x}_\mu^TA\bm{x}_\rho-\bm{x}_\lambda^TA\bm{x}_\rho\\
&=\eps t_\rho(4\eps t_\lambda)-4(t_\lambda t_\rho-\eps t_\mu)\\
&=4t_\lambda t_\rho-4t_\lambda t_\rho+4\eps t_\mu\\
&=4\eps t_\mu.
\end{align*}
Thus $\bm{x}_{y\mu},\bm{x}_\rho$ are $\eps t_\mu$-compatible. We also need to check the corresponding determinant identity. We compute
\begin{align*}
p_\rho q_{y\mu}-p_{y\mu}q_\rho
&=p_\rho(\eps t_\rho q_\mu-q_\lambda)
-(\eps t_\rho p_\mu-p_\lambda)q_\rho\\
&=\eps t_\rho(p_\rho q_\mu-p_\mu q_\rho)
-(p_\rho q_\lambda-p_\lambda q_\rho)\\
&=\eps t_\rho(\eps t_\lambda)
-(t_\lambda t_\rho-\eps t_\mu)\\
&=t_\lambda t_\rho-t_\lambda t_\rho+\eps t_\mu\\
&=\eps t_\mu.
\end{align*}
So condition (ii) holds.

Finally, we check condition (iii). In the new triple $y\Tt$, the left and right entries are $\bm{x}_\mu$ and $\bm{x}_\rho$. By assumption, $\bm{x}_\mu,\bm{x}_\rho$ are $\eps t_\lambda$-compatible, and
\[
p_\rho q_\mu-p_\mu q_\rho=\eps t_\lambda.
\]
Thus it remains only to verify that
\[
t_\mu t_\rho-\eps t_{y\mu}=\eps t_\lambda.
\]
This follows from the definition of the recursion. 

Therefore condition (iii) holds.

Thus $y\Tt$ is a recursive triple with the same value of $\eps$. The proofs for $x\Tt,\ov{x}\Tt,\ov{y}\Tt$ are similar.

    We now prove the result for $s\Tt$. Let $\eps=1$ if $\Tt$ is sign-matching and $\eps=-1$ if $\Tt$ is sign-alternating. We will show that $s\Tt$ is a recursive triple with sign $-\eps$.

Write $s\Tt=(\bm{x}_\mu,\bm{x}_\rho,\bm{x}_{s\rho}).$
By the definition of the $s$-mutation, we have $t_{s\rho}=-t_\lambda.$

We first check condition (i). We need $\bm{x}_\mu,\bm{x}_\rho$ to be $-\eps t_{s\rho}$-compatible. Since $-\eps t_{s\rho}=\eps t_\lambda,$ 
this follows from the assumption that $\bm{x}_\mu,\bm{x}_\rho$ are $\eps t_\lambda$-compatible. Moreover,
\[
p_\rho q_\mu-p_\mu q_\rho=\eps t_\lambda=-\eps t_{s\rho}.
\]
Hence condition (i) holds.

We now check condition (ii). In $s\Tt$, condition (ii) requires $\bm{x}_\rho,\bm{x}_{s\rho}$ to be $-\eps t_\mu$-compatible. Since $\bm{x}_\lambda,\bm{x}_\rho$ are $(t_\lambda t_\rho-\eps t_\mu)$-compatible, we have
\[
\bm{x}_\rho^TA\bm{x}_\lambda=4(t_\lambda t_\rho-\eps t_\mu).
\]
Thus, to show that
\[
\bm{x}_\rho^TA\bm{x}_{s\rho}=-4\eps t_\mu,
\]
it is enough to show that
\[
\bm{x}_\rho^TA\bm{x}_\lambda-\bm{x}_\rho^TA\bm{x}_{s\rho}=4t_\lambda t_\rho.
\]
Expanding the left-hand side, this is equivalent to
\begin{align*}
2(p_\lambda(p_\rho-7q_\rho)+q_\lambda(p_\rho+q_\rho)+t_\lambda t_\rho)
&=4t_\lambda t_\rho,\\
p_\lambda(p_\rho-7q_\rho)+q_\lambda(p_\rho+q_\rho)
&=t_\lambda t_\rho.
\end{align*}
This holds by Lemma~\ref{lem:7cond}. Therefore $\bm{x}_\rho,\bm{x}_{s\rho}$ are $-\eps t_\mu$-compatible.

We also need to verify the determinant identity
\[
-\eps t_\mu=p_{s\rho}q_\rho-p_\rho q_{s\rho}.
\]
Expanding the right-hand side gives
\[
p_{s\rho}q_\rho-p_\rho q_{s\rho}
=(6p_\lambda-q_\lambda)q_\rho-p_\lambda p_\rho.
\]
By Lemma~\ref{lem:7cond}, applied to $\bm{x}_\lambda$ and $\bm{x}_\rho$, we have
\[
t_\lambda t_\rho
=
p_\rho(p_\lambda+q_\lambda)+q_\rho(q_\lambda-7p_\lambda).
\]
Since
\[
t_\lambda t_\rho-\eps t_\mu=p_\rho q_\lambda-p_\lambda q_\rho,
\]
we obtain
\[
-\eps t_\mu
=
q_\rho(6p_\lambda-q_\lambda)-p_\lambda p_\rho.
\]
Thus
\[
p_{s\rho}q_\rho-p_\rho q_{s\rho}=-\eps t_\mu,
\]
as desired. So condition (ii) holds.

Finally, we check condition (iii). In $s\Tt$, condition (iii) requires $\bm{x}_\mu,\bm{x}_{s\rho}$ to be compatible with
\[
t_\mu t_{s\rho}-(-\eps)t_\rho=t_\mu t_{s\rho}+\eps t_\rho.
\]
Since $t_{s\rho}=-t_\lambda$, this quantity is
\[
-t_\lambda t_\mu+\eps t_\rho.
\]
Thus we want to show that
\[
\bm{x}_\mu^TA\bm{x}_{s\rho}=4(-t_\lambda t_\mu+\eps t_\rho).
\]
As $\bm{x}_\mu^TA\bm{x}_\lambda=4\eps t_\rho$, it is enough to show that
\[
\bm{x}_\mu^TA\bm{x}_\lambda-\bm{x}_\mu^TA\bm{x}_{s\rho}
=
4t_\lambda t_\mu.
\]
Expanding the left-hand side, we have
\[
\bm{x}_\mu^TA\bm{x}_\lambda-\bm{x}_\mu^TA\bm{x}_{s\rho}
=
2(p_\lambda(p_\mu-7q_\mu)+q_\lambda(p_\mu+q_\mu)+t_\lambda t_\mu).
\]
Thus it remains to show that
\[
p_\lambda(p_\mu-7q_\mu)+q_\lambda(p_\mu+q_\mu)=t_\lambda t_\mu,
\]
which holds by Lemma~\ref{lem:7cond}. Hence the compatibility condition in (iii) holds.

It remains to verify the determinant identity for condition (iii), namely
\[
p_{s\rho}q_\mu-p_\mu q_{s\rho}
=
t_\mu t_{s\rho}+\eps t_\rho
=
-t_\lambda t_\mu+\eps t_\rho.
\]
Equivalently, we need to show that
\[
(6p_\lambda-q_\lambda)q_\mu-p_\mu p_\lambda
=
\eps t_\rho-t_\lambda t_\mu.
\]
By Lemma~\ref{lem:7cond}, applied to $\bm{x}_\lambda$ and $\bm{x}_\mu$, we have
\[
p_\lambda(p_\mu-7q_\mu)+q_\lambda(p_\mu+q_\mu)=t_\lambda t_\mu.
\]
Subtracting this identity from
\[
p_\mu q_\lambda-p_\lambda q_\mu=\eps t_\rho
\]
gives
\[
(6p_\lambda-q_\lambda)q_\mu-p_\mu p_\lambda
=
\eps t_\rho-t_\lambda t_\mu,
\]
as desired.

Thus $s\Tt$ satisfies the recursive-triple conditions with sign $-\eps$. Hence $s\Tt$ is sign-alternating if $\Tt$ is sign-matching, and sign-matching if $\Tt$ is sign-alternating. The proof for $\ov{s}\Tt$ is similar.

\end{proof}

To prove Theorem~\ref{thm:ATFMut}, it will be helpful to have the following identities. We have chosen not to include them in the definition for simplicity. 

\begin{lemma} \label{lem:relations} Let $\Tt=(\bm{x}_\la,\bm{x}_\mu,\bm{x}_\rho)$ be a recursive triple. If $\Tt$ is a sign-matching recursive triple, set $\eps=1.$ Otherwise, set $\eps=-1.$ Then, the following identities hold: 
    \begin{itemlist}
      \item[{(i)}] $ \eps t_\mu=p_\la p_\rho+q_\la q_\rho-6 q_\rho p_\la$
            \item[{(ii)}] $ t_\la t_\rho q_\mu=\eps(t_\la q_\la+t_\mu q_\mu +t_\rho q_\rho)$

    \end{itemlist}
\end{lemma}
\begin{proof}
    We have (i) is Lemma~\ref{lem:7cond}~(ii) taking $k=-t_\mu.$

    For (ii), using condition (iii) in the definition of a recursive triple, we have
\[
t_\lambda t_\rho-\eps t_\mu=p_\rho q_\lambda-p_\lambda q_\rho.
\]
Multiplying by $q_\mu$ gives
\[
t_\lambda t_\rho q_\mu
=
p_\rho q_\lambda q_\mu-p_\lambda q_\rho q_\mu+\eps t_\mu q_\mu.
\]
Thus it remains to show that
\[
p_\rho q_\lambda q_\mu-p_\lambda q_\rho q_\mu
=
\eps(t_\lambda q_\lambda+t_\rho q_\rho).
\]
Using the determinant identities from a recursive triple, we compute
\begin{align*}
\eps(t_\lambda q_\lambda+t_\rho q_\rho)
&=(p_\rho q_\mu-p_\mu q_\rho)q_\lambda
+(p_\mu q_\lambda-p_\lambda q_\mu)q_\rho\\
&=p_\rho q_\lambda q_\mu-p_\mu q_\rho q_\lambda
+p_\mu q_\lambda q_\rho-p_\lambda q_\mu q_\rho\\
&=p_\rho q_\lambda q_\mu-p_\lambda q_\rho q_\mu.
\end{align*}
This proves (ii).
 
\end{proof}

Lastly, we give some identities which will be necessary later, but they also show how the $x,y,s$-mutations interact with each other. 
\begin{lemma} \label{lem:identS}
    For a recursive triple $\Tt=(\bm{x}_\la,\bm{x}_\mu,\bm{x}_\rho),$ we have
    \[ \bm{x}_{y\mu}=S\bm{x}_\la+\eps t_\mu \bm{x}_\rho \quad \text{and} \quad 
\bm{x}_{x\mu}=S^{-1}\bm{x}_\rho+\eps t_\mu \bm{x}_\la.\]
It follows that
\[ S\bm{x}_{\ovy\la}=-\bm{x}_{\ovx\rho}.\]
\end{lemma}
\begin{proof}
For the $(p,q)$-coordinates, we want to show 
\begin{equation*}
\begin{pmatrix} p_{y\mu} \\ q_{y\mu} \end{pmatrix} 
= S\begin{pmatrix} p_\lambda \\ q_\lambda \end{pmatrix} 
+ \varepsilon t_\mu \begin{pmatrix} p_\rho \\ q_\rho \end{pmatrix}.
\end{equation*}

The right hand side expands as
\begin{equation*}
S\begin{pmatrix} p_\lambda \\ q_\lambda \end{pmatrix} 
+ \varepsilon t_\mu \begin{pmatrix} p_\rho \\ q_\rho \end{pmatrix} 
= \begin{pmatrix} 6p_\lambda - q_\lambda \\ p_\lambda \end{pmatrix} 
+ \varepsilon t_\mu \begin{pmatrix} p_\rho \\ q_\rho \end{pmatrix} 
= \begin{pmatrix} 6p_\lambda - q_\lambda + \varepsilon t_\mu p_\rho \\ 
p_\lambda + \varepsilon t_\mu q_\rho \end{pmatrix}.
\end{equation*}
Using the definition of the $y$-mutation, we want to show
\begin{align*}
\varepsilon t_\rho p_\mu - p_\lambda &= 6p_\lambda - q_\lambda 
    + \varepsilon t_\mu p_\rho \\
\varepsilon t_\rho q_\mu - q_\lambda &= p_\lambda 
    + \varepsilon t_\mu q_\rho,
\end{align*}
or equivalently, 
\begin{align*}
    \eps(t_\rho p_\mu-t_\mu p_\rho)=7p_\la -q_\la \\
    \eps(t_\rho q_\mu-t_\mu q_\rho)=p_\la+q_\la. 
\end{align*}
In the case where $\eps=1,$ this is \cite[Lemma 4.6(i)]{M1}. The proof where $\eps=-1$ follows similarly to \cite[Lemma 4.6(i)]{M1}.  The \(t\)-coordinate also agrees, since
\[
t_{y\mu}=\eps t_\rho t_\mu-t_\lambda,
\]
while the third coordinate of
\[
S\bm{x}_\lambda+\eps t_\mu\bm{x}_\rho
\]
is
\[
-t_\lambda+\eps t_\mu t_\rho.
\]

For the second identity, we can similarly derive that in the case where $\eps=1$, these are directly the identities
\[ p_\rho+q_\rho=\eps(p_\mu t_\la-p_\la t_\mu) \quad \text{and} \quad p_\rho-7q_\rho=\eps(q_\la t_\mu-q_\mu t_\la)\] 
from 
 \cite[Lemma 4.6(ii)]{M1}. The case of $\eps=-1$ follows similarly.

 To see the last statement, we apply the left identity to $\ov{y}\Tt=(\bm{x}_{\ovy\la},\bm{x}_\la,\bm{x}_\rho),$ which gives that 
 \[ \bm{x}_\mu =S\bm{x}_{\ovy\la}+\eps t_\la \bm{x}_\rho,\]
 or equivalently
 \[-\bm{x}_{\ovx\rho}= \bm{x}_\mu-\eps t_\la \bm{x}_\rho=S\bm{x}_{\ovy\la}.\]

\end{proof}
\begin{rmk} \rm
   Note that the relation $S\bm{x}_{\ovy\la}=-\bm{x}_{\ovx\rho}$ is merely an algebraic identity given the definition of the recursive sequences. It is not the case that both $\bm{x}_{\ovy\la}$ and $\bm{x}_{\ovx\rho}$ will have positive $p,q,$ so it is not giving a relationship on quasi-perfect classes.  
\end{rmk}

\subsection{Quadrilaterals and Triples} \label{ss:ATFandTrip}

In Section~\ref{ss:nodalRays}, given a triple of $(p,q)$-coordinates, we defined a family of quadrilaterals $Q(\Tt)$ in Definition~\ref{def:QTnoL} with nodal rays and edge directions and no prescribed affine lengths. We now specify the affine lengths: 
\begin{definition} \label{def:DecQuad}
   Let $\Tt=(\bm{x}_\la,\bm{x}_\mu,\bm{x}_\rho)$ be a recursive $H$-triple where $(d_\bullet,m_\bullet)$ satisfy \eqref{eq:dmefIntro} for $\bullet=\la,\mu,\rho$ and set
   \[ d_\bullet'=d_\bullet-3q_\bullet, \quad m_\bullet'=m_\bullet-q_\bullet. \]
    Define the decorated quadrilateral $Q_{H_b}(\Tt):=OXVY$ such that the nodal rays and direction vectors are given by $Q(\Tt)$ as in Definition~\ref{def:QTnoL}
    and the affine lengths are given by
 
\begin{align}
\label{eq:QTH}
    |OY|&=\frac{d_\la-m_\la b}{q_\la} \quad \text{and} \quad |OX|=\frac{m_\rho'b-d_\rho'}{q_\rho}. 
\end{align} 
Let $\Tt=(\bm{x}_\la,\bm{x}_\mu,\bm{x}_\rho)$ be a recursive $P$-triple where $(e_\bullet,f_\bullet)$ satisfy \eqref{eq:dmefIntro} for $\bullet=\la,\mu,\rho,$ and set
   \[ e_\bullet'=e_\bullet-2q_\bullet, \quad f_\bullet'=f_\bullet-2q_\bullet. \]
  
    Define the decorated quadrilateral $Q_{P_b}(\Tt):=OXVY$ such that the nodal rays and direction vectors are given by $Q(\Tt)$ as in Definition~\ref{def:QTnoL}
    and the affine lengths are given by
\begin{align} \label{eq:QTP}
 |OY|&=\frac{e_\la+f_\la b}{q_\la} \quad \text{and} \quad |OX|=\frac{-f_\rho'b-e_\rho'}{q_\rho}. 
\end{align}
Note in both cases, the other data specifies the affine lengths of $|XV|$ and $|VY|$. 

We say that $Q_{\bullet,b}(\Tt):=OXVY$ is {\bf well-defined} if $\vn_Y,\vn_V,\vn_X$ all point to the interior of $Q_{\bullet,b}(\Tt)$, the affine lengths $|OY|,|OX|$ are positive, and $V$ is in the interior of the first quadrant.
\end{definition}

\begin{example}
    Note that $Q_{H_b}(\Tt_0^H)=Q_{H_b}^0$ and $Q_{P_b}(\Tt_0^P)=Q_{P_b}^0$, the starting quadrilaterals pictured in Figure~\ref{fig:basediagr}. 
\end{example}

\begin{lemma} \label{lem:XVVY}
    For $Q_{H_b}(\Tt_0^H),$ the remaining affine lengths are given by
    \[ |VY|=\eps(\frac{m_\rho'-d_\rho'b}{q_\la q_\mu}) \quad \text{and} \quad 
     |XV|=\eps(\frac{m_\la-d_\la b}{q_\rho q_\mu})\]
     for $\eps=1.$
    For $Q_{P_b}(\Tt_0^P),$ the remaining affine lengths are given by 
    \[ |VY|=\eps(\frac{e_\rho'-f_\rho'b}{q_\la q_\mu}), \quad \text{and} \quad
     |XV|=\eps(\frac{e_\la-f_\la b}{q_\rho q_\mu})\]
     for $\eps=1.$
\end{lemma}
\begin{proof}
    We use the condition that $Q_{\bullet,b}(\Tt_0^\bullet)$ for $\bullet=P,H$ is a quadrilateral implying that
    \[ (|OX|,0)+|XV|\ovr{XV}+|VY|\ovr{VY}=(0,|OY|).\]
     The formulas can be easily checked from this condition by substituting in the values of the triples and solving for $|XV|$ and $|VY|.$
\end{proof}

\begin{rmk} \rm 
    \begin{itemlist}
        \item[{(i)}] As seen in Lemma~\ref{lem:nodalPos}, for $Q_{P_b}(\Tt)$ and $Q_{H_b}(\Tt)$ to be well-defined, we must always have $q_\bullet>0$. On the other hand, $p_\bullet$ does not necessarily need to be positive. See Example~\ref{ex:unusualTrip}. Lemma~\ref{lem:notParallel} shows that $q_\mu = 0$ if and only if $\ovr{VY}$ and $\ovr{XV}$ are parallel. 
      
        \item[{(ii)}] If $w$ is a word in $x,y,s,\ov{x},\ov{y},\ov{s},$ then let $\eps=1$ if $w\Tt_0^\bullet$ is sign-matching and $\eps=-1$ otherwise. 
        In the proof of Theorem~\ref{thm:ATFMut}, we see that the formulas in Lemma~\ref{lem:XVVY} for the affine lengths of $|XV|$ and $|VY|$ hold for $wQ_{\bullet,b}(\Tt_0^\bullet)$ assuming it is well-defined. 
        
        \item[{(iii)}] For $\bullet=H_b,P_b$, Remark~\ref{rmk:obsAndDiagrams} explains how the affine lengths of $Q_{\bullet,b}(\Tt)$ relate to the obstruction functions of exceptional classes for the ellipsoid embedding function. 
       
    \end{itemlist}
\end{rmk}

\begin{example} \label{ex:unusualTrip}
Here, we give some examples to demonstrate why we relax some of the notions about triples seen in \cite{MMW}.  
\begin{itemlist}
    \item[{(i)}] There are examples where $w\Tt_0^H$ does not have $p_\la/q_\la \leq p_\mu/q_\mu \leq p_\rho/q_\rho$ and $wQ_{H_b}^0=Q_{H_b}(w\Tt_0^H)$ is well-defined. 
    We have that 
    \[x\Tt_0^H=\left((1,1,2),(0,1,-3),(2,1,-1)\right).\] As desired, this triple has $1=p_\la/q_\la > p_\mu/q_\mu=0$, and it can be checked that for $1/2<b<1$ the mutation $xQ_{H_b}^0$ is well-defined and $Q_{H_b}(x\Tt_0^H)=xQ_{H_b}^0.$
As $p_\la/q_\la>p_\mu/q_\mu,$ in taking a $y$-mutation
 the recursion parameter $\nu_\rho=p_\mu q_\la-p_\la q_\mu=-1$ is negative. Note that the triple $yx\Tt_0^H=((0,1,-3),(-1,-2,1),(2,1,-1))$, but there is no $b$-value such that $yx\Tt_0^H$ is well-defined.
    \item[{(ii)}] There are examples where $w\Tt_0^H$ has some $p<0$ and $wQ_{H_b}^0=Q_{H_b}(w\Tt_0^H)$ is well-defined.
    We can compute that: 
    \[ \ov{y}\ov{x}\Tt_0^H=\left((-1,2,5),(1,1,2),(6,1,3)\right).\]
    Further, $\ov{y}\ov{x}Q_{H_b}^0$ is well-defined when $0<b<1/2.$
   The negativity of $p_\la$ implies that $\ovr{YV}$ and the nodal ray $\vn_Y$ both have positive slope. 
\end{itemlist}
\end{example}

% \begin{figure}[h!]
%     \centering
%      \includegraphics[scale=0.05]{ex2.jpeg}
%   \includegraphics[scale=0.1]{ex1.jpeg}
%     \caption{On the left is the ATBD $Q_{H_b}(x\Tt_0^H)$ from Example (i) of \ref{ex:unusualTrip}. Note, we have that $p_\la/q_\la>p_\mu/q_\mu.$ In this case, if we did a $y$-mutation, the recursion parameter would be negative. 
%     On the right is the ATBD, $Q_{H_b}(\ov{y}\ov{x}\Tt_0^H)$ where we have $p_\la<0$. Note that the negativity of $p_\la$ causes $\vn_Y$ to have positive slope. }
%     \label{fig:weirdEx}
% \end{figure}

The proof of Theorem~\ref{thm:ATFMut} follows from the following proposition. The proof of the proposition relies on the lemmas following the proposition with the necessary computations. 

\begin{prop} \label{prop:ATFMut}
    Assume $\Tt$ is a recursive $H$- or $P$-triple obtained via mutation from $\Tt_0^H$ or $\Tt_0^P$ and $Q_{\bullet,b}(\Tt)$ is well-defined for $\bullet=H,P$. If $wQ_{\bullet,b}(\Tt)$ is well-defined for $w=x,y,s,\ov{x},\ov{y},\ov{s}$, then
     $wQ_{\bullet,b}(\Tt)=Q_{\bullet,b}(w\Tt)$. 
\end{prop}
\begin{proof}
    We first give the references to later in the paper for the statement about the $x,y,s$-mutations. The statement about the nodal rays and direction vectors follow from Proposition~\ref{prop:nodalDir}.

    For the affine lengths, the formulas depend on whether we have a $P$- or $H$-triple. The proofs can be found in Lemma~\ref{lem:xLengths}, Lemma~\ref{lem:yLengths} and Lemma~\ref{lem:sLengths}.

    For the statement about $\ov{x}Q(\Tt),\ov{y}Q(\Tt),$ and $\ov{s}Q(\Tt)$, we utilize the proofs of the inverse mutations. This is because we have that 
    $w\ov{w}Q(\Tt)=Q(\Tt)$. Hence, if
    \begin{align*} wQ(\Tt')=Q(\Tt) &\implies \ov{w}Q(\Tt)=Q(\Tt').
    \end{align*}
    Because $wQ(\ov{w}\Tt)=Q(\Tt),$ the result holds. 
\end{proof}

We now go through the lemmas used in the above proof. To check the claims about the side lengths, we first give the identity. Note the geometric interpretation of this identity is given in \cite[Remark 6.4]{M1}:

\begin{lemma} \label{lem:complicated} For a recursive triple $\Tt=(\bm{x}_\la,\bm{x}_\mu,\bm{x}_\rho)$ where $\eps=1$ if $\Tt$ is sign-matching and $\eps=-1$ otherwise, we have
\begin{itemlist}
    \item[{(i)}] $\eps t_\la \begin{pmatrix}
               1+p_\mu q_\mu-6q_\mu^2 \\ q_\mu^2\end{pmatrix}=q_{x\mu}\begin{pmatrix}
               p_\mu-6q_\mu \\ q_\mu \end{pmatrix}
               +q_\mu \vn_X$
    \item[{(ii)}] $-\eps t_\rho \begin{pmatrix} -q_\mu^2 \\ p_\mu q_\mu -1 \end{pmatrix}=q_{y\mu} \begin{pmatrix}
        q_\mu \\ - p_\mu
    \end{pmatrix}+q_\mu \begin{pmatrix}
        q_\la \\ -p_\la
    \end{pmatrix}$
\end{itemlist}
\end{lemma}
\begin{proof}
    For the case where $\eps=1,$ this is shown in \cite[Lemma 4.6(vi,vii)]{M1}. The case where $\eps=-1$ is similar. 
\end{proof}

One consequence of Lemma~\ref{lem:complicated} is the following result:
\begin{lemma}\label{lem:notParallel}
Let $\Tt$ be a recursive triple. Then
\[
\det(\ovr{XV},\ovr{VY})=q_\mu^2.
\]
In particular, $\ovr{XV}$ and $\ovr{VY}$ are parallel if and only if $q_\mu=0$.
\end{lemma}

\begin{proof}
We compute $\det\begin{pmatrix}
    \ovr{XV} & \ovr{VY}
\end{pmatrix}.$
        Taking the matrix 
        \[M_X:=\begin{pmatrix}
        1-p_\rho q_\rho+6q_\rho^2 & (p_\rho-6q_\rho)^2 \\ -q_\rho^2 & 1+p_\rho q_\rho-6q_\rho^2 \end{pmatrix},\]
        it can be immediately checked that $\det(M_X)=1$ and $M_X(\ovr{XV})=\ovr{OX}.$ By \cite[Lemma 6.5]{M1} for the case where $\eps=1$, we have that $M_X(\ovr{VY})=\begin{pmatrix} 1+p_\mu q_\mu-6q_\mu^2 \\ q_\mu^2 \end{pmatrix}.$ A similar argument works for the case where $\eps=-1$ using the same proof with the identities in Lemma~\ref{lem:complicated}.  It follows that
    \[ \det\begin{pmatrix}
    \ovr{XV} & \ovr{VY}
\end{pmatrix}=\det\begin{pmatrix} M_X\ovr{XV} & M_X\ovr{VY} \end{pmatrix}=q_\mu^2.\]
Thus the determinant vanishes if and only if $q_\mu=0$, which proves the claim.
\end{proof}

We consider the cases for the affine lengths for $yQ,xQ,sQ$ separately, but the proofs are very similar. We describe in detail a few different cases, but for the remaining cases, we refer the reader to a similar argument. 

\begin{lemma}\label{lem:yLengths}
Assume that $\mathcal{T}$ is a recursive $H$- or $P$-triple obtained via mutation from $\mathcal{T}_0^H$ or $\mathcal{T}_0^P$, and that the affine lengths of $Q_{\bullet,b}(\mathcal{T})$ are given by the formulas in Definition~\ref{def:DecQuad} and Lemma~\ref{lem:XVVY}. Assuming the nodal ray $\vn_Y$ extends to intersect $\ovr{XV},$ the  geometric mutation $yQ_{\bullet,b}(\mathcal{T}) := OX_yV_yY_y$ has affine lengths if $\bullet=H$:
\[
|OY_y| = \frac{d_\mu - m_\mu b}{q_\mu}, \qquad |OX_y| = \frac{m'_\rho b - d'_\rho}{q_\rho},
\]
\[
|X_yV_y| = \varepsilon\left(\frac{m_\mu - d_\mu b}{q_\rho q_{y\mu}}\right), \qquad |V_yY_y| = \varepsilon\left(\frac{m'_\rho - d'_\rho b}{q_\mu q_{y\mu}}\right)
\]
and if $\bullet=P:$
\[
|OY_y| = \frac{e_\mu + f_\mu b}{q_\mu}, \qquad |OX_y| = \frac{-f'_\rho b - e'_\rho}{q_\rho},
\]
\[
|X_yV_y| = \varepsilon\left(\frac{e_\mu - f_\mu b}{q_\rho q_{y\mu}}\right), \qquad |V_yY_y| = \varepsilon\left(\frac{e'_\rho - f'_\rho b}{q_\mu q_{y\mu}}\right)
\]
In particular, $|OX_y|$ and $|OY_y|$ agree with Definition~\ref{def:DecQuad} applied to the triple $y\mathcal{T}$, and $|X_yV_y|$ and $|V_yY_y|$ agree with the formulas of Lemma~\ref{lem:XVVY} applied to $y\mathcal{T}$.
\end{lemma}

\begin{proof}
Note, Lemma~\ref{lem:standard} and Lemma~\ref{lem:standFrom} state that in performing the geometric mutation $y$ to $Q_{\bullet,b}(\Tt)$, the nodal rays and direction vectors of $yQ_{\bullet,b}(\Tt)$ agree with the nodal rays and direction vectors of $Q_{\bullet,b}(y\Tt).$ 

Further, as $\Tt$ is obtained via mutation from $\Tt_0^H$ or $\Tt_0^P$, we use a proof by induction and assume that the formulas for $|VY|$ and $|XV|$ given in Lemma~\ref{lem:XVVY} hold for $Q_{\bullet,b}(\Tt).$ 

Hence, we must verify that 
\begin{align*} 
    |OY_y|&=|OY|+|VY| \\
    |OX_y|&=|OX| \\ 
    |V_yY_y|&=|XV|-|X_yV_y| \\ 
    |X_yV_y|&=\frac{|OY|q_\la-|OX|p_\la}{p_\la(1+p_\rho q_\rho-6q_\rho^2)+q_\la q_\rho ^2}
\end{align*}
Note that the last equality holds as setting \[c:=\frac{|OY|q_\la-|OX|p_\la}{p_\la(1+p_\rho q_\rho-6q_\rho^2)+q_\la q_\rho^2},\] we have that $c$ solves the equation 
    \[ |OY|(0,1)+k\vn_Y=|OX|(1,0)+c\ovr{XV}\]
    for some real number $k.$
    So $c$ is the affine length of $X_yV_y$ after performing $yQ_{\bullet,b}(\Tt).$
    
    We include the details for the case where $\Tt$ is a $P$-triple. The case where $\Tt$ is an $H$-triple is very similar. Note that the proof where $\Tt$ is a sign-matching $H$-triple is found in \cite[Lemma 6.2]{M1}, and the case where $\Tt$ is a sign-alternating $H$-triple follows closely to the proof we include in detail here.
   
    Set $\eps=1$ if $\Tt$ is a sign-matching recursive triple, and  otherwise, set $\eps=-1.$ Note that by Proposition~\ref{prop:genTrip}, $y\Tt$ will preserve the property of either being sign-matching or sign-alternating, so $\Tt$ and $y\Tt$ satisfy Lemma~\ref{lem:relations} and Lemma~\ref{lem:identS} for the same value of $\eps.$
    
     For the first equality, we must verify that
    \[ \frac{e_\la+f_\la b}{q_\la}+\eps(\frac{e_\rho'-f_\rho'b}{q_\la q_\mu})=\frac{e_\mu+f_\mu b}{q_\mu}.\]
    This will hold if 
    \begin{align*}
        \eps e_\rho'&=e_\mu q_\la-e_\la q_\mu \\
        \eps f_\rho'&=f_\la q_\mu-f_\mu q_\la.
    \end{align*}
    We can write $e_\bullet',f_\bullet',e_\bullet,f_\bullet$ in terms of $p_\bullet,q_\bullet,t_\bullet$ using \eqref{eq:dmefIntro}. Then, we find the equalities hold by Lemma~\ref{lem:identS} and as $\Tt$ is a recursive triple $p_\mu q_\la-p_\la q_\mu=\eps t_\la.$ Hence, the first equality in the lemma holds. 
    
    The second equality is immediate. For the third equality, we must verify that 
    \[ \eps (\frac{e_\rho'-f_\rho'b}{q_\mu q_{y\mu}})=\eps (\frac{e_\la-f_\la b}{q_\rho q_\mu})-\eps (\frac{e_\mu-f_\mu b}{q_\rho q_{y\mu}}).\]
    This will hold if
    \begin{align*}
        e_\rho'q_\rho=e_\la q_{y\mu}- e_\mu q_\mu \\
         f_\rho' q_\rho= f_\la q_{y\mu}- f_\mu q_\mu.
    \end{align*}
    Again, by writing in $p,q,t$-coordinates, 
    these are equivalent. We will verify the first one and check that 
    \begin{align*}
        (p_\rho-7q_\rho+t_\rho)q_\rho=(p_\la+q_\la+ t_\la)(\eps t_\rho q_\mu-q_\la)-(p_\mu+q_\mu+ t_\mu)q_\mu.
    \end{align*}
    By Lemma~\ref{lem:relations}~(ii),
    $t_\rho q_\rho=\eps t_\la t_\rho q_\mu-t_\la q_\la-t_\mu q_\mu.$
    It remains to check that
   \begin{align*}
        (p_\rho-7q_\rho)q_\rho=(p_\la+q_\la)(\eps t_\rho q_\mu-q_\la)-(p_\mu+q_\mu)q_\mu.
    \end{align*}
    This holds by expanding the relations in Lemma~\ref{lem:identS} to get that 
    \begin{align*}
        p_\mu+q_\mu&=\eps(p_\la t_\rho+q_\rho t_\la) \\
        p_\la+q_\la&=\eps(t_\rho q_\mu-q_\rho t_\mu) \\
        p_\rho-7q_\rho&=\eps(q_\la t_\mu-q_\mu t_\la)
    \end{align*}
    and simplifying.

    For the last equality, we can verify that the denominator is equal to $q_{y\mu}$ as 
    \[ p_\la(1+p_\rho q_\rho-6q_\rho^2)+q_\la q_\rho^2=p_\la+\eps t_\mu q_\rho=\eps t_\rho q_\mu-q_\la=q_{y\mu}\]
    by Lemma~\ref{lem:relations}~(i) and the expression for $p_\la+q_\la$ above. 
    Hence, we want to check that 
    \[ \eps \frac{e_\mu-f_\mu b}{q_\rho}=e_\la+f_\la b+\frac{p_\la(f_\rho'b+e_\rho')}{q_\rho},\]
    which follows if we verify that
    \begin{align*}
     \eps e_\mu= q_\rho e_\la+p_\la e_\rho' \\
     \eps f_\mu=-q_\rho f_\la -p_\la f_\rho'
    \end{align*}
    In terms of $(p,q,t)$-coordinates, we must verify that 
    \[ \eps (p_\mu+q_\mu+ t_\mu)=(p_\la+q_\la+ t_\la)q_\rho+p_\la(p_\rho-7q_\rho+t_\rho).\]
    This holds because by Lemma~\ref{lem:identS},
    we have that 
    \[ \eps(p_\mu+q_\mu)= t_\la q_\rho+ t_\rho p_\la\] and by Lemma~\ref{lem:relations}~(i), we have that 
    \[ \eps t_\mu=(p_\la+q_\la)q_\rho+p_\la(p_\rho-7q_\rho).\]
\end{proof}

\begin{lemma}\label{lem:xLengths}
Assume that $\mathcal{T}$ is a recursive $H$- or $P$-triple obtained via mutation from $\mathcal{T}_0^H$ or $\mathcal{T}_0^P$, and that the affine lengths of $Q_{\bullet,b}(\mathcal{T})$ are given by the formulas in Definition~\ref{def:DecQuad} and Lemma~\ref{lem:XVVY}. Then the geometric mutation $xQ_{\bullet,b}(\mathcal{T}) := OX_xV_xY_x$ has affine lengths if $\bullet=H$:
\[
|OY_x| = \frac{d_\lambda - m_\lambda b}{q_\lambda}, \qquad |OX_x| = \frac{m'_\mu b - d'_\mu}{q_\mu},
\]
\[
|V_xY_x| = \varepsilon\left(\frac{m'_\mu - d'_\mu b}{q_\lambda q_{x\mu}}\right), \qquad |X_xV_x| = \varepsilon\left(\frac{m_\lambda - d_\lambda b}{q_\mu q_{x\mu}}\right),
\]
and if $\bullet=P$:
\[
|OY_x| = \frac{e_\lambda + f_\lambda b}{q_\lambda}, \qquad |OX_x| = \frac{-f'_\mu b - e'_\mu}{q_\mu},
\]
\[
|V_xY_x| = \varepsilon\left(\frac{e'_\mu - f'_\mu b}{q_\lambda q_{x\mu}}\right), \qquad |X_xV_x| = \varepsilon\left(\frac{e_\lambda - f_\lambda b}{q_\mu q_{x\mu}}\right).
\]
In particular, $|OX_x|$ and $|OY_x|$ agree with Definition~\ref{def:DecQuad} applied to the triple $x\mathcal{T}$, and $|V_xY_x|$ and $|X_xV_x|$ agree with the formulas of Lemma~\ref{lem:XVVY} applied to $x\mathcal{T}$. 

\end{lemma}
\begin{proof}
    Similarly to the proof of Lemma~\ref{lem:yLengths} this involves checking that the following equalities hold: 
    \begin{align*} 
    |OY_x|&=|OY| \\
    |OX_x|&=|OX|+|XV| \\ 
    |V_xY_x|&=\frac{|OY|(p_\rho-6q_\rho)+|OX|q_\rho}{(p_\rho-6q_\rho)(p_\la q_\la-1)+q_\rho q_\la^2} \\ 
    |X_xV_x|&=|VY|-|V_xY_x|. 
\end{align*}
Checking the equalities is similar to Lemma~\ref{lem:yLengths} so we leave out the details. 
\end{proof}

\begin{lemma}\label{lem:sLengths}
Assume that $\mathcal{T}$ is a recursive $H$- or $P$-triple obtained via mutation from $\mathcal{T}_0^H$ or $\mathcal{T}_0^P$, and that the affine lengths of $Q_{\bullet,b}(\mathcal{T})$ are given by the formulas in Definition~\ref{def:DecQuad} and Lemma~\ref{lem:XVVY}. Then the geometric mutation $sQ_{\bullet,b}(\mathcal{T}) := OX_sV_sY_s$ has affine lengths if $\bullet=H$:
\[
|OY_s| = \frac{d_\mu - m_\mu b}{q_\mu}, \qquad |OX_s| = \frac{d_\lambda - m_\lambda b}{p_\lambda},
\]
\[
|V_sY_s| = \varepsilon\left(\frac{m_\lambda - d_\lambda b}{q_\mu q_\rho}\right), \qquad |X_sV_s| = -\varepsilon\left(\frac{m_\mu - d_\mu b}{p_\lambda q_\rho}\right),
\]
and if $\bullet=P$:
\[
|OY_s| = \frac{e_\mu + f_\mu b}{q_\mu}, \qquad |OX_s| = \frac{e_\lambda + f_\lambda b}{p_\lambda},
\]
\[
|V_sY_s| = \varepsilon\left(\frac{e_\lambda - f_\lambda b}{q_\mu q_\rho}\right), \qquad |X_sV_s| = -\varepsilon\left(\frac{e_\mu - f_\mu b}{p_\lambda q_\rho}\right).
\]
In particular, $|OX_s|$ and $|OY_s|$ agree with Definition~\ref{def:DecQuad} applied to the triple $s\mathcal{T}$, and $|V_sY_s|$ and $|X_sV_s|$ agree with the formulas of Lemma~\ref{lem:XVVY} applied to $s\mathcal{T}$.

\end{lemma}

\begin{proof}
    Let $S(\bm{x}_\la):=(p_s,q_s,-t_\la).$
    We assume that $\Tt$ is an $H$-triple. 
Set $\eps=1$ if $\Tt$ is a sign-matching recursive triple, and  otherwise, set $\eps=-1.$ Note that by Proposition~\ref{prop:genTrip}, $s\Tt$ satisfies Lemma~\ref{lem:relations} and Lemma~\ref{lem:identS} for $-\eps.$

As before, as $\Tt$ is obtained via mutation from $\Tt_0^H$ or $\Tt_0^P$, we use a proof by induction and assume that the formulas for $|VY|$ and $|XV|$ given in Lemma~\ref{lem:XVVY} hold for $Q_{\bullet,b}(\Tt).$ 

Hence, we must verify that 
\begin{align*} 
      |OY_s|&=|OY|+|VY| \\
    |OX_s|&=|OY|\frac{q_\la}{p_\la} \\
     |V_sY_s|&=|XV| \\
     |X_sV_s|&=|OX|-|OX_s|
\end{align*}
    
    We first check $|OY_s|=|OY|+|VY|:$
    \[\frac{d_\mu-m_\mu b}{q_\mu}=|OY|+|VY|=\frac{d_\la-m_\la b}{q_\la}+\eps\left(\frac{m_\rho'-d_\rho' b}{q_\la q_\mu}\right).\] This equality holds for all $b$-values if 
    both 
     \begin{align} \label{eq:sATF1}
        \eps m_\rho'&=d_\mu q_\la-d_\la q_\mu  \notag \\
        \eps d_\rho'&=m_\mu q_\la-m_\la q_\mu
    \end{align}
    We have
    \begin{align*} 8(d_\mu q_\la -d_\la q_\mu)&= 3\begin{vmatrix}  p_\mu & p_\la \\ q_\mu & q_\la \end{vmatrix}+\begin{vmatrix} 
 t_\mu &  t_\la
\\ q_\mu & q_\la
    \end{vmatrix} \\
    8(m_\mu q_\la -m_\la q_\mu)&= \begin{vmatrix}  p_\mu & p_\la \\ q_\mu & q_\la \end{vmatrix}+3\begin{vmatrix} 
 t_\mu & t_\la
\\ q_\mu & q_\la
    \end{vmatrix}
    \end{align*}
    As $8m_\rho'=p_\rho-7q_\rho+3 t_\rho$ and 
    $8d_\rho'=3p_\rho-21q_\rho+ t_\rho$, 
    then \eqref{eq:sATF1} will hold if 
    \[ \eps t_\rho=\begin{vmatrix}  p_\mu & p_\la \\ q_\mu & q_\la \end{vmatrix} \quad \text{and} \quad 
    \eps(p_\rho-7q_\rho)=\begin{vmatrix} 
t_\mu & t_\la
\\ q_\mu & q_\la
    \end{vmatrix}. 
    \]
    The first holds as $\Tt$ is a recursive triple and the second holds by expanding Lemma~\ref{lem:identS}.

    For the second equality, we check that:
    \[ \frac{m_s'b-d_s'}{q_s}=\frac{d_\la-m_\la b}{q_\la}\cdot \frac{q_\la}{p_\la},\] which holds as $q_s=p_\la,$ $m_s'=-m_\la,$ and $d_s'=-d_\la$ by Lemma~\ref{lem:symTriple}.

    For the third equality, we verify $|V_sY_s|=|XV|$:
    \[-\eps \left(\frac{m_s'-d_s'b}{q_\mu q_\rho}\right)=\eps \left(\frac{m_\la-d_\la b}{q_\mu q_\rho}\right),\]
    which holds as $m_s'=-m_\la, d_s'=-d_\la$. 
    For the last equality, we verify $|OX|-|OX_s|=|X_sV_s|:$
    \[ \frac{m_\rho'b-d_\rho'}{q_\rho}-\frac{m_s'b-d_s'}{q_s}=-\eps\left(\frac{m_\mu-d_\mu b}{q_\rho q_s}\right).\] This follows if:
    \begin{align} \label{eq:sATF2}
        -\eps m_\mu=d_s' q_\rho-d_\rho' q_s\notag \\
        -\eps d_\mu=m_s' q_\rho-m_\rho 'q_s.
    \end{align}
   These hold because 
    \[-\eps t_\mu=\begin{vmatrix}p_s & p_\rho \\ q_s & q_\rho \end{vmatrix} \quad \text{and} \quad
    -\eps(p_\mu+q_\mu)=\begin{vmatrix}
             t_s &  t_\rho \\ q_s & q_\rho
        \end{vmatrix}.
    \]
    where the first follows as $s\Tt$ is a recursive triple and the second holds by Lemma~\ref{lem:identS} for the triple $s\Tt.$
   
    The case where $\Tt$ is a $P$-triple follows similarly. 
\end{proof}

\section{Triples and Perfect Classes}\label{ss:tripPerf}

In this section, we prove Theorems \ref{thm:HTrip} and \ref{thm:PTrip}. The proofs have two distinct steps. First, for each perfect class for $H$ (resp. quasi-perfect class for $P$), we find an explicit word $w$ in $x,y,s,\ov{x},\ov{y},\ov{s}$ such that the class appears as an element of the triple $w\Tt_0^H$ (resp. $w\Tt_0^P$). These words are given in Proposition~\ref{lem:perf4H} and \ref{lem:perf4P}. Second, we verify that for each such word $w$, there is an interval of $b$-values
for which the corresponding mutation sequence $wQ_{H_b}^0$ (resp. $wQ_{P_b}^0$)
is geometrically well-defined on the ATBD. Together, these two steps complete both theorems, whose proofs appear at the beginning of Section~\ref{ss:bValue}. 

Section \ref{ss:perfectBack} provides the necessary background on the classification of perfect and quasi-perfect classes for $H$ and $P$. In Section \ref{ss:findPerfect}, we give the explicit mutation sequences on $\Tt_0^H$ and $\Tt_0^P$ that recover all perfect classes for $H$ and the corresponding quasi-perfect classes for $P.$ In Section~\ref{ss:bValue}, we show each of these sequences is realized on the ATF level for an appropriate interval of $b$-values.

\subsection{Classification of Perfect Classes} \label{ss:perfectBack}
Here, we review the classification of perfect classes for $H$ completed in \cite{MMW}. This classification is given in terms of the recursive triples we considered in Section~\ref{ss:trip}. In \cite{MMW}, the perfect classes were constructed recursively via mutations on triples and symmetries $S,R.$ In \cite{MPW}, a corresponding family of quasi-perfect classes for $P$ was defined; we show in Corollary~\ref{cor:isPerfect} that these are in fact perfect, though whether they constitute all perfect classes for $P$ remains open.

Quasi-perfect classes for $H,P$ correspond precisely to elements of $\X_H$ and $\X_P$ as we now describe.  Let $(d;m)$ denote the homology class $dL-mE_1 \in H_2(H)$ where $L$ is the homology class of the line and $E_1$ is the homology class of the exceptional divisor. Let $(e,f)$ denote the homology class $eS_1+fS_2 \in H_2(P)$. Let $(d;m;p,q)$ represent the homology class of a $k$-fold blow up of $H$ where $(p,q)$ stands for the integral weight sequence $W(p,q)$, and likewise, for  $(e,f;p,q)$ and $P.$ 

\begin{lemma} \label{lem:AreClasses}
    Assume that $(p,q,t) \in \X_H$ is a positive class. Then, $(d;m;p,q)$ is a quasi-perfect class for $H.$ Similarly, if $(p,q,t) \in \X_P$ is a positive class, then $(e,f;p,q)$ is a quasi-perfect class for $P.$ 
\end{lemma}
\begin{proof}
    We must check that the classes are quasi-perfect, i.e. they have self-intersection $-1$ and first Chern number $1$. This computation is done in \cite[Section 2.2]{MM}. It follows from the fact that the weight sequence of $W(p,q)=(a_1,\hdots,a_n)$ has the property that $\sum a_i^2=p q$ and $\sum a_i=p+q-1.$ Hence, for $H,$ the condition to be a quasi-perfect class is equivalent to 
    \[ 3d-m=p+q \quad \text{and} \quad d^2-m^2=pq-1,\]
    and for $P,$ the conditions are
    \[ 2(e+f)=p+q \quad \text{and} \quad 2ef=pq-1.\]
\end{proof}

The perfect classes for $H$ were classified in \cite{MMW} using two additional symmetries, denoted $S,R$, acting on triples beyond the $x,y$-mutations, together with a family of seed triples $\mathcal{B}_n.$ These symmetries are defined as follows: 

\begin{definition} \label{def:SRintro}
    Given a triple \[\Tt:=(\bm{x}_\la,\bm{x}_\mu,\bm{x}_\rho):=\left((p_\la,q_\la,t_\la),(p_\mu,q_\mu,t_\mu),(p_\rho,q_\rho,t_\rho) \right) \in \mathcal{X}^3,\] define the action of $S=\begin{pmatrix}
        6 & -1 & 0 \\ 1 & 0 & 0 \\ 0 & 0 & -1
    \end{pmatrix}$ and $R=\begin{pmatrix}
        6 & -35 & 0\\ 1 & -6 &0 \\ 0& 0 & -1
    \end{pmatrix}$ on $\Tt$ as: 
    \[ S\Tt:=\left(S\bm{x}_\la,S\bm{x}_\mu,S\bm{x}_\rho\right)\] and
     \[ R\Tt:=\left(R\bm{x}_\rho,R\bm{x}_\mu,R\bm{x}_\la\right).\]
\end{definition}

The seed triples $\mathcal{B}_n$ we perform the symmetries and $x,y$-mutations on are defined as 
\[ \mathcal{B}_n:=\left((n+6,1,n+3),(n^2+11n+29,n+4,n^2+8n+13),(n+8,1,n+5)\right) \]
where $n$ is a nonnegative integer. 

McDuff, Weiler, and the author showed in \cite{MMW} that the $(p,q)$-perfect classes for $H$ are precisely the elements of the triples in $\Tt(H)$ defined below. The corresponding family of triples for $P$ is denoted $\Tt(P).$

\begin{definition}  \label{def:triplesIntro}
For $\bullet\in\{H,P\}$, define $\Tt(\bullet)$ to be the set of triples obtained as follows:
\begin{itemlist}
\item[{(i)}] $S^i\Tt_0^\bullet$ for all $i\geq 0$;

\item[{(ii)}] $S^iR^\delta(\mathcal B_n)$ for $\delta\in\{0,1\}$ and $i,n\geq 0$, where $n$ is even if $\bullet=H$ and $n$ is odd if $\bullet=P$;

\item[{(iii)}] all possible $x,y$-mutations of the triples in (ii).

\end{itemlist}

\end{definition}

\begin{rmk}\label{rmk:Ordering}
    We make some remarks about the triples in Definition~\ref{def:triplesIntro} to acquaint the reader with some of the properties shown in \cite{MMW}. These properties are used in Section~\ref{ss:bValue}.

    Given a triple $\Tt=(\bm{x}_\la,\bm{x}_\mu,\bm{x}_\rho) \in \T(H),$ the triple satisfies
    \[ \frac{p_\la}{q_\la} \leq \frac{p_\mu}{q_\mu} \leq \frac{p_\rho}{q_\rho}.\]
    For the triples $S^i\Tt_0^H=(\bm{x}_{\la,i},\bm{x}_{\mu,i},\bm{x}_{\rho,i})$, the ratios $p_{\bullet,i}/q_{\bullet,i}$ are strictly increasing in $i$ with limit given by $3+2\sqrt{2}.$

    The basic triples $\mathcal{B}_{2n}$ have $p_\la/q_\la=2n+6$ and $p_\rho/q_\rho=2n+8$ for $n \geq 0$. In applying $R(\mathcal{B}_{2n})$, the classes with ratios $2n+6$ will all lie in between $6 \leq p/q \leq 7$ (except in the case where $n=0$). Applying iterations of the shift in these cases will decrease the ratios of $p/q$ with limit $3+2\sqrt{2}.$
    
    To determine where the ratios of the $x,y$-mutations lie, in \cite{MMW}, it was shown that given $\Tt=(\bm{x}_\la,\bm{x}_\mu,\bm{x}_\rho),$ which is described by the triples in (ii) or (iii) of $\T(H),$ the ordering is as follows
    \[ \frac{p_\la}{q_\la} \leq \frac{p_{x\mu}}{q_{x\mu}} \leq \frac{p_\mu}{q_\mu} \leq \frac{p_{y\mu}}{q_{y\mu}} \leq \frac{p_\rho}{q_\rho}.\]
    For instance, the $x,y$-mutations of $B_{2n}$ will all have $p/q$ between $2n+6$ and $2n+8$ and will be ordered in this interweaving way. 
   
\end{rmk}

The proof in \cite{MMW} that the classes in $\T(H)$ are perfect, and not merely quasi-perfect, used results from ATF mutations in \cite{M1} combined with methods from symplectic embeddings. Precisely, in \cite{M1}, we showed that for the triple $w\mathcal{B}_{2n}$ where $w$ is a word in $x,y$, we showed that there is a $b$-value such that $wy\ov{x}^{n+2}Q_{H_b}^0=Q_{H_b}(w\mathcal{B}_{2n})$ is well-defined. The ATBD from $wy\ov{x}^{n+2}Q_{H_b}^0$ gave results about ellipsoid embeddings into $H_b$. These embedding results were used to show the classes $w\mathcal{B}_{2n}$ are perfect. 

McDuff and Siegel gave a more direct proof that the classes were perfect in \cite{McSi}. In particular, they showed that if $Q$ is an ATBD for $H$ (resp. $P$) that is $\Tt$-standard for $\Tt=((p_\la,q_\la),(p_\mu,q_\mu),(p_\rho,q_\rho))$ where for $\bullet=\la,\mu,\rho,$ $p_\bullet,q_\bullet>0$, then there is an index zero $(p_\bullet,q_\bullet)-$unicuspidal symplectic curve in $H$ (resp. $P$). Each such unicuspidal curve corresponds to a perfect class. Hence, the result of \cite{M1} that $wy\ov{x}^{n+2}Q_{H_b}^0$ is an ATBD for $H$ implies that the classes $w\mathcal{B}_n$ are perfect. The explicit mutation sequence shows that these perfect classes can be realized in the ATBD for a fixed $b$-value.

Here, we aim to show that all of the classes in $\Tt(H)$ and $\Tt(P)$ can be realized in ATBDs by mutating $Q_{H_b}^0$ or $Q_{P_b}^0$ for a fixed $b$-value. We give an example here, which contrasts some of the differences for $P$ and $H$.

\begin{example} \label{eg:s^kforPandH2}

    In Example~\ref{eg:s^kforPandH}, we discussed that the triples $s^k\Tt_0^H$ include all the triples in $S^i\Tt_0^\bullet.$ We then explained how the $b$-values such that $s^kQ_{H_b}^0$ is well-defined are related to $m_k/d_k$ where $(d_k;m_k)$ are the solutions to \eqref{eq:dmefIntro} coming from the classes $(p_k,q_k)$ appearing in the triples $s^k\Tt_0^H.$ We alluded that the case of $P$ is more delicate. We describe the details here.

    If we let 
 \[s^k\Tt_0^P=:\left((p_{k-1},q_{k-1},t_{k-1}),(p_{k},q_{k},t_{k}),(p_{k+1},q_{k+1},t_{k+1})\right),\] we also have that $s^{3k}\Tt_0^P=S^k\Tt_0^P.$ 
 Unlike in the $H$-case, the ratios $p_k/q_k$ are not pairwise distinct.

For example, consider $s\Tt_0^P=\left((3,1,0),(5,1,2),(5,1,-2)\right).$ Here, we have that $p_1/q_1=p_2/q_2,$ but note that $-t_1=t_2.$ For the solutions $(e,f)$ to \eqref{eq:dmefIntro}, we get $(e_1,f_1)=(2,1)$ and $(e_2,f_2)=(1,2).$ For $P,$ there are two $(5,1)$-perfect classes. Note that $(3,1,0)$ only appears in the middle entry of $s^k\Tt_0^P$ for $k=0$, and there is only one $(3,1)$-perfect class for $P.$ As $t=0,$ we have $e=f$.

 In general, $p_i/q_i$ where $i \equiv 1,2 \mod{3}$ will appear in the middle entry of $s^k\Tt_0^P$ for two different values of $k$, say $k_1,k_2$ where $t_{k_1}=2$ and $t_{k_2}=-2.$ This will correspond to two different $(p_i,q_i)$-quasi-perfect classes with coordinates $(e_{k_1},f_{k_1})$ and $(f_{k_1},e_{k_1})$. 

 Unlike mutations on $Q_{H_b}^0,$ there is no $b$-value such that $s^4Q_{P_b}^0$ is well-defined. For the case of $b=1,$ the first nodal ray extends to hit a vertex, so we consider neither $s$ nor $y$ well-defined in the case. While the sequence $s^k\Tt_0^P$ provides the desired numerics of the classes, it does not give us the result we want on the ATF side. 

 We define a modified sequence of mutation words ${r_i}$ as follows. Set $r_0=\id$, let $r_i=s^i$ for $i=1,2,3$, and for $k\geq 1$ define
\begin{equation} \label{eq:SPsequence}
r_{3k+1}:=y(s^2y)^{k-1}s^3,\qquad
r_{3k+2}:=sy(s^2y)^{k-1}s^3,\qquad
r_{3k+3}:=(s^2y)^ks^3.
\end{equation}
Define the sequence $\bm{x}_i$ so that $\bm{x}_i$ is the new element appearing in the triple $r_i\Tt_0^P$. The position of $\bm{x}_i$ in the triple depends on $i$: an $s$-mutation shifts the entries to the left, while a $y$-mutation introduces the new element in the middle.

For $i=0,1,2,3$, we have
\[r_{i}\Tt_0^P:=(\bm{x}_{i-1},\bm{x}_i,\bm{x}_{i+1}).\]
We set \(\bm{x}_{-1}\) to be the left entry of \(\Tt_0^P\), so that
\(\Tt_0^P=(\bm{x}_{-1},\bm{x}_0,\bm{x}_1)\).
If $i \geq 4,$ we have that
\begin{align*}
    r_{3k+1}\Tt_0^P&=(\bm{x}_{3k},\bm{x}_{3k+2},\bm{x}_{3k+1}) \\
    r_{3k+2}\Tt_0^P&=(\bm{x}_{3k+2},\bm{x}_{3k+1},\bm{x}_{3k+3}) \\
    r_{3k+3}\Tt_0^P&=
    (\bm{x}_{3k+1},\bm{x}_{3k+3},\bm{x}_{3k+4}).  
\end{align*}

It is easy to check that the $(p,q)$-coordinates of $r_i\Tt_0^P$ are the same as the $(p,q)$-coordinates for $s^i\Tt_0^P,$ but the $t$-coordinates are only the same up to a sign. 

It is the case that there are $b$-values such that $r_k Q_{P_b}^0$ is defined for all $k.$ Letting $\al_k:=e_k/f_k$, by Lemma~\ref{lem:alPSordering}, for $i \geq 1,$
\[ \al_2<\al_{3i+1}<\al_{3i+4}<\al_0=\al_{3i}=1<\al_{3i+5}<\al_{3i+2}<\al_1\]
where $1/\al_{3i+1}=\al_{3i+2}.$ In Lemma~\ref{lem:alPSordering2}, we see that $r_{3i+1}Q_{P_b}^0,r_{3i+2}Q_{P_b}^0,r_{3i+3}Q_{P_b}^0$ are well-defined if $1<b<\al_{3i+2}.$ 

   While the cases of $P$ and $H$ are similar here, we note that the ATFs reflect the extra symmetry for the exceptional classes for $P$, which leads to some differences. In either case, the ratios of either $(d_k;m_k)$ or $(e_k,f_k)$ in the relevant triples are determining when the mutations are defined.
\end{example}

We end with one lemma from \cite{MM,MPW} about the bounds on $p/q$ for the classes $(p,q,t) \in \mathcal{T}(H),\mathcal{T}(P)$:
\begin{lemma}\cite[Lemma 2.1.5]{MM},\cite[Theorem 3.1.3]{MPW} \label{lem:pOverq6}
    Let $w$ be a word in $x,y.$ Let $(p,q,t)$ be an entry of the triple $\Tt$. The following hold:  
    \begin{itemlist}
        \item[{(i)}] if $\Tt=S^i\Tt_0^\bullet$ for $i \geq 0$, then $1\leq p/q<3+2\sqrt{2}$;
        \item[{(ii)}] if $\Tt=w\mathcal{B}_n$ for $n \geq 0$, then $6 \leq p/q$;
        \item[{(iii)}] if $\Tt=wR\mathcal{B}_n$ for $n \geq 0$, then $6 < p/q \leq 7$, except when $n=0,$ where $(p_\rho,q_\rho)$ may equal $(1,0)$;
        \item[{(iv)}] if $\Tt=wS^iR\mathcal{B}_n$ for $n \geq 0$ and $i \geq 1,$ then $3+2\sqrt{2} \leq p/q < 6$;
        \item[{(v)}] if $\Tt=wS^i\mathcal{B}_n$ for $n \geq 0$ and $i \geq 1$, then $3+2\sqrt{2} \leq p/q < 6.$
    \end{itemlist}
\end{lemma}
\subsection{Perfect classes constructed via mutations on triples} \label{ss:findPerfect}
Here, we complete the first step of the proofs of Theorems~\ref{thm:HTrip} and \ref{thm:PTrip} where we construct explicit words $w$ such that the classes in $\T(H)$ and $\T(P)$ are an element of $w\Tt_0^H$ or $w\Tt_0^P$. These words are given in Proposition~\ref{lem:perf4H} for $\Tt(H)$ and Proposition~\ref{lem:perf4P} for $\Tt(P).$ As we show in Example~\ref{eg:s^kforPandH2}, these words are not unique. In Section~\ref{ss:bValue}, we show the words given are well-defined on the ATBD.

We first must establish some additional notation. As previously noted, for $P$ if $(p,q,t)$ corresponds to the quasi-perfect class $(e,f)$, then $(p,q,-t)$ also gives a quasi-perfect class with coordinates $(f,e).$ See Example~\ref{eg:s^kforPandH2}. While there is this symmetry, for a $P$-triple $\Tt=(\bm{x}_\la,\bm{x}_\mu,\bm{x}_\rho)$, the quadrilateral $Q_{P_b}(\Tt)$ defined in \eqref{eq:QTP} is not invariant under this symmetry. For mutations on $Q_{P_b}^0$, to be well-defined it is essential that we have the right sign on $t.$  In particular, for the triples in $\Tt(P)$ that appear after applying iterations of $S,$ from the ATBD perspective, we want $t<0.$ Let $I_t$ denote the matrix \[I_t:=\begin{pmatrix} 1&0&0 \\ 0&1&0\\ 0&0&-1 \end{pmatrix}.\] 
Hence, $I_t(p,q,t)=(p,q,-t)$ so if $(e,f)$ is the solution to \eqref{eq:dmefIntro}, then $(f,e)$ is the solution for $I_t(p,q,t).$

 We now state a preliminary lemma which will help us with the classes obtained by iterating $S$ before giving the general sequences. Let $r_{i}$ be the sequence defined in \eqref{eq:SPsequence}. As seen in Example~\ref{eg:s^kforPandH2}, this is the sequence that will realize the classes $s^i\Tt_0^P$ in the ATBD.

\begin{lemma} \label{lem:commuteS}
    Let $w$ denote a word in $x,y,\ov{x},\ov{y},s$.
    For a recursive triple $\Tt$, we have the equality \[s^3w\Tt=ws^3\Tt.\]  
    Additionally, we have that for $j\geq 1$,
    \[wr_{3j}\Tt_0^P=I_t^{j-1}s^{3j}w\Tt_0^P.\] In other words, the $(p,q)$-coordinates of $wr_{3j}\Tt_0^P$ and $s^{3j}w\Tt_0^P$ agree and the $t$-coordinates agree up to sign. 
\end{lemma}
\begin{proof}
    Let $\Tt=(\bm{x}_\la,\bm{x}_\mu,\bm{x}_\rho).$ We first show both equalities assuming $\Tt$ is sign-matching and $w=x.$
    We have that $x\Tt=(\bm{x}_\la,t_\la \bm{x}_\mu-\bm{x}_\rho,\bm{x}_\mu),$ so 
    \[ s^3x\Tt=S(\bm{x}_\la,t_\la \bm{x}_\mu-\bm{x}_\rho,\bm{x}_\mu).\]
    Alternatively, if we first applied $s^3,$ we obtain
    \[ xs^3\Tt=xS\Tt=(S(\bm{x}_\la),t_\la S(\bm{x}_\mu)-S(\bm{x}_\rho),S\bm{x}_\mu).\]
    
    A similar argument holds for the case where $\Tt$ is sign-alternating and for the mutations given by $y,\ov{x},\ov{y},s.$ 

    For the second statement, we first show that the claim holds when $w$ is the empty word.  By direct computation
    \[ r_3\Tt_0^P=s^3\Tt_{0}^P=\left( (5,1,-2),(17,3,0),(29,5,-2) \right).\] 
    Let $\Tt(j)$ denote $r_{3j}\Tt_0^P=((p_\la,q_\la,t_\la),(p_\mu,q_\mu,t_\mu),(p_\rho,q_\rho,t_\rho)).$  It suffices to show that $\Tt(j+1)=I_tS\Tt(j)$. We claim this follows if for all $j>1,$ $\Tt(j)$ satisfies
    \begin{itemlist}
        \item[{(i)}] $t_\rho=-2,t_\mu=0,$
        \item[{(ii)}] $|t_\rho|\bm{x}_\mu-\bm{x}_\la=(p_\rho,q_\rho,-t_\rho),$
        \item[{(iii)}] $S(p_\la,q_\la,t_\la)=(p_\rho,q_\rho,-t_\rho)$ 
    \end{itemlist}
    To see that this is sufficient, note that for all $j>1,$ as $\Tt(j+1)=s^2y\Tt(j),$ the triple $\Tt(j)$ is sign-alternating. Hence, property (ii)  implies that 
 \[ y\Tt(j)=(\bm{x}_\mu,I_t\bm{x}_\rho,\bm{x}_\rho).\]
 Using the fact that $t_\mu=0$ and property (iii), we get that 
 \[\Tt(j+1)=s^2y\Tt(j)=(\bm{x}_\rho,S(\bm{x}_\mu),S(I_t\bm{x}_\rho))=(I_tS\bm{x}_\la,I_tS\bm{x}_\mu,I_tS\bm{x}_\rho)\]
as desired. 

It remains to show that $\Tt(j)$ satisfies properties (i),(ii), and (iii), which we prove by induction. The base case when $j=2$ can be checked directly. Now assuming the properties hold for $\Tt(j)$, we have that $\Tt(j+1)=s^2y\Tt(j)=(\bm{x}_\rho,S(\bm{x}_\mu),S(I_t\bm{x}_\rho)).$ As $t_\mu=0,$ the $t$-value of $S(\bm{x}_\mu)$ remains zero. As $SI_t$ act trivially on $t$, we have the $t$-value of $S(I_t\bm{x}_\rho)$ remains $-2$ concluding that property (i) holds. To verify property (ii), as $2\bm{x}_\mu-\bm{x}_\la=(p_\rho,q_\rho,-t_\rho)$ applying $S$ to both sides we get
    $2S\bm{x}_\mu-S\bm{x}_\la=SI_t\bm{x}_\rho.$ Using property (iii), we have 
    $2S\bm{x}_\mu-I_t\bm{x}_\rho=SI_t\bm{x}_\rho$. 
    Multiplying both sides by $I_t$ and using that $t_\mu=0,$ we get
    \[ 2S\bm{x}_\mu-\bm{x}_\rho=S\bm{x}_\rho,\]
    which is precisely (ii) for the triple $\Tt(j+1).$
    
     Hence, the statement holds when $w$ is empty. If $w$ is a word of $x,y,s,\ov{x}, \ov{y}$ the statement follows as $ s^{3j}w\Tt_0^P=ws^{3j}\Tt_0^P$. 
   
\end{proof}
We now give the mutation sequences on $\Tt_0^H$ to obtain the perfect classes for $H$ described in Definition~\ref{def:triplesIntro}. This will be the first part of the proof of Theorem~\ref{thm:HTrip}. Note that the triple
\[R(\mathcal{B}_0) =((13,2,-5),(34,5,-13),(1,0,-3)),\]
so has $q_\rho=0$ implying that $Q_{H_b}(R(\mathcal{B}_0))$ is not defined for any $b$-value, which explains why in the proposition the classes obtained via mutation for $R(\mathcal{B}_0)$ are handled differently. See Corollary~\ref{cor:rTrip} for the precise sequences for $R(\mathcal{B}_0).$

\begin{prop}
    \label{lem:perf4H}
    Every perfect class for $H$ appears in some triple obtained from doing some combination of 
    $x,y,s,\ov{x},\ov{y}$ mutations on 
    $\Tt_{0}^H$. In particular, if $w$ is a word in $x,y$, we have
     \begin{align*}
        ws^{3j}\Tt_{0}^H&=wS^j(\Tt_{0}^H),\\
    wy\ov{x}^{n+2}s^{3j}\Tt_{0}^H&=wS^j(\mathcal{B}_{2n}), \quad \text{for $n \geq 0, j \geq 0$} \\
        wyx^ny\ov{x}s^{3j+1}\Tt_{0}^H&= wS^jR(\mathcal{B}_{2n}), \quad \text{for $n \geq 0, j \geq 1$}\\
         wyx^{n-1}y\ov{x}\ov{y}s\Tt_{0}^H&=wR(\mathcal{B}_{2n}), \quad \text{for $n \geq 1$}. 
    \end{align*}
    Lastly, the perfect classes obtained from \(wR(\mathcal B_0)\) are realized by the mutation sequences described in Corollary~\ref{cor:rTrip}.
\end{prop}
\begin{proof}
    Throughout this proof, we apply Lemma~\ref{lem:commuteS} implicitly using that $s^3$ commutes with any word $w$ in $x,y,\ov{x},\ov{y}$. For the first equality, note that by definition $s^3\Tt=S\Tt$ for any triple $\Tt$, so $ws^{3j}\Tt_0^H=wS^j\Tt_0^H$ as desired.

For the second equality, since the left element of $\ov{x}^k\Tt_0^H$ stays constant at $(1,1,2)$ each $\ov{x}$-mutation performs a recursion by $2$, giving 

 \[\ov{x}^{k}\T_0^H=\left((1,1,2),(2k+2,1,2k-1),(2k+4,1,2k+1)\right).\] 
 For $k \geq 2,$ we compute that $y\ov{x}^k\Tt_0^H=\mathcal{B}_{2k-4}.$

Hence using the commutativity of $s^{3j}$, \[y\ov{x}^k s^{3j} \Tt_0^H = S^j(\mathcal{B}_{2k-4}),\] and relabeling $k = n+2$ gives the second equality.

For the third equality, the strategy is to show that $yx^ny\ov{x}s^4\Tt_0^H$ and $SR(\mathcal{B}_{2n})$ agree by computing both sides in $(p,q,t)$-coordinates. Beginning with the left hand side,  we have 
\[ s^4\Tt_{0}^H=((11,2,1),(23,4,-1),(29,5,2)),\]
and applying $y\ov{x}$ gives
\[ y\ov{x}s^{4}\Tt_{0}^H=((29,5,2),(76,13,5),(6,1,3)). 
\]
Since each $x$-mutation performs a recursion of $2,$ we obtain $x^ny\ov{x}s^{4}\Tt_{0}^H$ is equal to
\begin{align*} \left((29,5,2),(76+70n,
13+12n,5+2n),(76+70(n-1),13+12(n-1),3+2n)\right).
\end{align*}
In performing the remaining $x,y$-mutations, the $t$-coordinate stays positive, so we just write the $(p,q)$-coords for brevity now. 
A further $y$-mutation gives:
\begin{align*} &yx^ny\ov{x}s^{4}\Tt_{0}^H= \\ &\left( (76+70n,13+12n),(140n^2+362n+199,24n^2+62n+34), (6+70n,1+12n)
\right).
\end{align*}
For the right hand side, note that $R$ reverses the order of the triple so $SR(\mathcal{B}_{2n})$ is equal to 
\begin{align*}
&SR\left(\left((2n+6,1),(4n^2+22n+29,2n+4),(2n+8,1)\right)\right) \\
&=S((12n+13,2n+2),(24n^2+62n+34,4n^2+10n+5),(12n+1,2n)) \\
&=
\left( (76+70n,13+12n),(140n^2+362n+199,24n^2+62n+34), (6+70n,1+12n)
\right)
\end{align*} 
matching the left hand side as desired. 
The proof of $yx^{n-1}y\ov{x}\ov{y}s\Tt_0^H=R(\mathcal{B}_{2n})$ is similar where we check both sides agree by performing the mutations. For the proof of the statement about $wR(\mathcal{B}_0)$, see Corollary~\ref{cor:rTrip} and Lemma~\ref{lem:R'andR}. This needs to be considered differently because $q_\rho=0$ in $R(\mathcal{B}_0).$

\end{proof}

We now give the details for the $x,y$-mutations of $R(\mathcal{B}_0).$ Letting
$R(\mathcal{B}_0):=(\bm{x}_0,\bm{x}_1,\bm{x}_2)$, the right element $\bm{x}_2$ of $R(\mathcal{B}_0)$ is $(1,0,-3),$ which does not correspond to a perfect class. Instead, we define \begin{align*} R'(\mathcal{B}_0):=((0,1,-3),\bm{x}_0,\bm{x}_1)=\left((0,1,-3),(13,2,-5),(34,5,-13)\right).\end{align*}

Below we verify that $R'(\mathcal{B}_0)=\ov{x}^2\ov{y}s\Tt_0^H$, and in Lemma~\ref{lem:bDefH5}, we show there is a $b$-value such that $\ov{x}^2\ov{y}sQ_{H_b}^0$ is defined. Further, the elements of $wR(\mathcal{B}_0)$ can be obtained via mutation from $R'(\mathcal{B}_0).$ 

\begin{lemma} \label{lem:R'andR}
Let $y^nR(\mathcal{B}_0):=\left((p_\la,q_\la,t_\la),(p_\mu,q_\mu,t_\mu),(1,0,-3)\right)$. Then, 
\[ \ov{x}^nR'(\mathcal{B}_0)=\left((0,1,-3),(p_\la,q_\la,t_\la),(p_\mu,q_\mu,t_\mu)\right).\]
Further
\[ y\ov{x}^nR'(\mathcal{B}_0)=xy^nR(\mathcal{B}_0).\]
Finally, the triple $R'(\mathcal{B}_0)$ is equal to $\ov{x}^2\ov{y}s\Tt_0^H.$ 

\end{lemma}
\begin{proof}
    The last statement is a simple computation, and note that it implies that $R'(\mathcal{B}_0)$ is a sign-alternating recursive triple by Proposition~\ref{prop:genTrip}. 
    The first statement is clear by induction as it holds for the base case. For $n \geq 1$, the statement holds as each time we perform a recursion by $3$ to the same elements. 

    For the second statement, as the triples are sign-alternating, we have 
    \[ xy^nR(\mathcal{B}_0)=((p_\la,q_\la,t_\la),(-t_\la p_\mu-1,-t_\la q_\mu,-t_\la t_\mu+3),(p_\mu,q_\mu,t_\mu))\]
    and
    \[y\ov{x}^nR'(\mathcal{B}_0)=((p_\la,q_\la,t_\la),(-t_\mu p_\la,-t_\mu q_\la-1,-t_\mu t_\la+3),(p_\mu,q_\mu,t_\mu)).\]
    It remains to check that
    \[ (-t_\la p_\mu-1,-t_\la q_\mu,-t_\la t_\mu+3)=(-t_\mu p_\la,-t_\mu q_\la-1,-t_\mu t_\la+3),\]
    or equivalently that
    \begin{align*}
        1&=t_\mu p_\la-t_\la p_\mu \\
        1&=t_\la q_\mu-t_\mu q_\la.  
    \end{align*}
    For $n=0,$ this is easily checked, and the cases of $n>0$ follow as terms on the right will remain the same under the recursion (see for example Lemma~\ref{lem:recurDif}). 
\end{proof}

\begin{cor} \label{cor:rTrip}
    The perfect classes obtained from $w'R(\mathcal{B}_0)$ where $w'$ is a word of $x,y$ are contained in the triples
    \[w\ov{x}^2\ov{y}s\Tt_0^H\] where the word $w$ satisfies: 
    if $w'=uxy^n,$ then $w=uy\ov{x}^n$ and if $w'=ux,$ then $w=uy$ where $u$ is a word in $x,y.$
\end{cor}

We now give the mutation sequences on $\Tt_0^P$ to obtain the quasi-perfect classes for $P$ described in Definition~\ref{def:triplesIntro} where here $r_k$ is defined in \eqref{eq:SPsequence}.

\begin{prop} \label{lem:perf4P} The quasi-perfect classes from $P$ in $\T(P)$ appear in some triple obtained from doing some combinations of $x,y,s,\ov{x}$ mutations on 
$\Tt_{0}^P.$   In particular, if $w$ is a word in $x,y$, we have
 \begin{align*}
        wy\ov{x}^{n+2}\Tt_{0}^P&=w\mathcal{B}_{2n+1}, \quad \text{for $n \geq 0$} \\
        wyx^ny\ov{x}\Tt_{0}^P&=wR(\mathcal{B}_{2n+1}), \quad \text{for $n \geq 0$} 
    \end{align*}
    and for $j \geq 1,$ we have 
       \begin{align*}
wr_{3j}\Tt_{0}^P&=wI_t^{j-1}S^j\Tt^P_{0} \\
        wy\ov{x}^{n+2}r_{3j}\Tt_{0}^P&=wI_t^{j-1}S^j(\mathcal{B}_{2n+1}), \quad \text{for $n \geq 0$} \\
        wyx^ny\ov{x}r_{3j}\Tt_{0}^P&=wI_t^{j-1}S^jR(\mathcal{B}_{2n+1}), \quad \text{for $n \geq 0$}.
    \end{align*}
\end{prop}
\begin{proof}
We begin with the statement where we applied no shift. Verifying that $y\ov{x}^{n+2}\T_0^P=\mathcal{B}_{2n+1}$ and $yx^ny\ov{x}\T_0^P=R(\mathcal{B}_{2n+1})$ for $n \geq 0$ is similar to the computations in Proposition~\ref{lem:perf4H}. Now, for the cases where $j \geq 1$, these follow from the cases with no shift and Lemma~\ref{lem:commuteS}, which implies that $wr_{3j}\Tt_0^P=I_t^{j-1}ws^{3j}\Tt_0^P$. 

\end{proof}

Our goal is to show that the mutation sequences in Proposition~\ref{lem:perf4H} and Proposition~\ref{lem:perf4P} are also defined on the ATF level. To do this, we require certain numerical conditions about these triples. We finish this section by giving these results. 

\begin{definition}
    Here, we let $w,u,v$ denote (possibly empty) words in $x,y,\ov{x},\ov{y},s,\ov{s}$. Define the set
    \[ \text{suffix}(w):=\{v \ | \ uv=w\}.\]
    
    We now restrict $w$ to be a (possibly empty) word in $x,y$, then define the sets \[\mathcal{W}_H=\text{suffix}\{ wy\ov{x}^{n+2}s^{3j},wyx^ny\ov{x}s^{3j+1},wyx^{n}y\ov{x}\ov{y}s,wy\ov{x}^{n+2}\ov{y}s\}\]
   
    and 
    \[ \mathcal{W}_P=\text{suffix}\{r_{3j},wy\ov{x}^{n+2}r_{3j},wyx^ny\ov{x}r_{3j}\}.\]
    Additionally, we define the corresponding set of triples
    \[ \mathcal{T}_H=\{ w\T_0^H \ | \ w \in \mathcal{W}_H\}\]
    and
    \[ \mathcal{T}_P=\{ w\T_0^P \ | \ w \in \mathcal{W}_P\}.\]

\end{definition}

\begin{cor}\label{cor:Positive}
Each triple in $\mathcal{T}_H$ and $\mathcal{T}_P$ is positive. 
\end{cor}
\begin{proof}
    This follows from the computations in Proposition~\ref{lem:perf4H} and Proposition~\ref{lem:perf4P}.
\end{proof}

\begin{cor} \label{cor:triples}
    The triples \begin{align*} wy\ov{x}^{n+2}r_{3j}&\Tt_0^P, \quad wyx^ny\ov{x}r_{3j}\Tt_0^P,\\
wy\ov{x}^{n+2}s^{3j}&\Tt_0^H,\quad wyx^ny\ov{x}s^{3j+1}\Tt_0^H, \quad wyx^{n-1}y\ov{x}\ov{y}s\Tt_0^H
    \end{align*}
    where $w$ is a (possibly empty) word in $x,y$ denoted $(\bm{x}_\la,\bm{x}_\mu,\bm{x}_\rho)$ satisfy the following properties: 
    \begin{itemlist}
        \item[{(i)}] For $\bullet=\la,\rho$, 
$(0,0,0)<(p_\bullet,q_\bullet,|t_\bullet|)<(p_\mu,q_\mu,|t_\mu|).$
        \item[{(ii)}] $|t_\la|,|t_\mu|,|t_\rho|\geq 2$
        \item[{(iii)}] $p_\la/q_\la < p_\mu/q_\mu <p_\rho/q_\rho$
        \item[{(iv)}] All $t$-values are positive if the triple is sign-matching and all $t$-values are negative if the triple is sign-alternating.
    \end{itemlist}
\end{cor}
\begin{proof}
    The statement where $w$ is empty and $j=0$ follows from the computations in Proposition~\ref{lem:perf4H}~and~\ref{lem:perf4P}. For the statement where $w$ is empty and $j \neq 0,$ the statement holds by the equalities in Proposition~\ref{lem:perf4H}~and~\ref{lem:perf4P} considering how $S,R,I_t$ changes the sign of $t$ and act on the $p,q$ coordinates. 

    The claim now follows by induction for non-empty words $w.$ Assuming that the statement holds for a triple $\Tt,$ we show it holds for $x\Tt$. Assume that $\Tt$ is sign-matching. Then, we have $t_\la=p_\rho q_\mu-p_\mu q_\rho>0.$
    We have $x\Tt=(\bm{x}_\la,\bm{x}_{x\mu},\bm{x}_\mu)$ where $\bm{x}_{x\mu}=t_\la\bm{x}_{\mu}-\bm{x}_\rho.$ Hence as the claim holds for $\Tt$ we conclude that, \[(0,0,0)<(p_\bullet,q_\bullet,t_\bullet)<\bm{x}_{x\mu}=(t_\la p_{\mu}-p_\rho,t_\la q_\mu-q_\rho,t_\la t_\mu-t_\rho)\] for $\bullet=\mu,\la$. Now, by Lemma~\ref{lem:recurDif}, the recursion implies that $p_{x\mu}/q_{x\mu}<p_\mu/q_\mu<p_\rho/q_\rho.$ To see that $p_\la/q_\la<p_{x\mu}/q_{x\mu},$ note that this holds if $p_{x\mu}q_\la-p_\la q_{x\mu}=t_\mu>0$, which was assumed. A similar proof holds for a $y$-mutation and the case where $\Tt$ is sign-alternating. 
\end{proof}

\begin{cor} \label{cor:signOfT}
    Let $w \in \mathcal{W}_H$ where $w=w's^j$ with $w'$ nonempty and let $w\Tt_0^H=(\bm{x}_\la,\bm{x}_\mu,\bm{x}_\rho).$ Then $t_\la,t_\mu,t_\rho>0$ if $j$ is even and $t_\la,t_\mu,t_\rho<0$ if $j$ is odd. Let $w \in \mathcal{W}_P$ where $w=w'r_{3j}$ with $w'$ nonempty and let $w\Tt_0^P=(\bm{x}_\la,\bm{x}_\mu,\bm{x}_\rho).$ Then $t_\la,t_\mu,t_\rho<0$ if $j \geq 1$ and $t_\la,t_\mu,t_\rho>0$ if $j=0.$
\end{cor}
\begin{proof}
    This follows from the computation of each of the triples in the proof of Proposition~\ref{lem:perf4H} and Proposition~\ref{lem:perf4P}. Then once we take combinations of $x,y$-mutations this follows by Corollary~\ref{cor:triples}~(iv). 
\end{proof}

\subsection{There is a \texorpdfstring{$b$}{b}-value}\label{ss:bValue}

In this section, we show that for the words $w \in \mathcal{W}_H,\mathcal{W}_P$, there is an interval of $b$-values such that $wQ_{H_b}^0$ (resp. $wQ_{P_b}^0$) is well-defined. The proof of this is given in a case-by-case setting. Before doing each case separately, we give the proofs of both Theorem~\ref{thm:HTrip} and Theorem~\ref{thm:PTrip} deferring the more technical lemmas after the proofs. 

\begin{proof}[Proof of Theorem \ref{thm:HTrip}]
In Proposition~\ref{lem:perf4H}, we completed the first statement of the theorem by directly giving the words $w$. For each word $w$ given in Proposition~\ref{lem:perf4H}, we then show in Lemma~\ref{lem:sMutDefinedH}, Lemma~\ref{lem:bDefH1}, Lemma~\ref{lem:bDefH2}, Lemma~\ref{lem:bDefH3}, Lemma~\ref{lem:bDefH4}, and Lemma~\ref{lem:bDefH5} that there is an interval of $b$-values such that $wQ_{H_b}^0$ is well-defined with one endpoint given by $m/d.$
\end{proof}

\begin{proof}[Proof of Theorem \ref{thm:PTrip}]
In Proposition~\ref{lem:perf4P}, we completed the first statement of the theorem by directly giving the words $w.$ For each word $w$ given in Proposition~\ref{lem:perf4P}, we then show in Lemma~\ref{lem:alPSordering2}, Lemma~\ref{lem:bDefP1}, Lemma~\ref{lem:poly1Q}, Lemma~\ref{lem:bDefP3}, and Lemma~\ref{lem:bDefP4} that there is an interval of $b$-values such that $wQ_{P_b}^0$ is well-defined with one endpoint given by either $e/f$ or $f/e.$
\end{proof}

We first give numerical conditions on the triple to check when $wQ_{H_b}(\Tt)$ is well-defined for a particular $b$-value. 
\begin{lemma}\label{lem:definedH}
    Let $\Tt=(\bm{x}_\la,\bm{x}_\mu,\bm{x}_\rho)\in \X_H^3$ be a recursive triple where $\eps=1$ if $\Tt$ is sign-matching and otherwise, $\eps=-1$. Assume $Q(\Tt)$ is well-defined for some $b$-value. Then
    \begin{itemlist}
        \item[{(i)}] 
        If \(q_{y\mu}>0\) and $d_\mu>0$, then \(sQ(\Tt)\) is well-defined precisely when
\[
        \eps\frac{m_\mu}{d_\mu}<\eps b,
\]
and \(yQ(\Tt)\) is well-defined precisely when
\[
        \eps\frac{m_\mu}{d_\mu}>\eps b.
\]

        \item[{(ii)}] $xQ(\Tt)$ is well-defined if $p_\rho/q_\rho>6$. 
        \item[{(iii)}] Let $S^{-1}(\bm{x}_\mu):=\bm{x}=(p,q,-t_\mu) \in\X_H$ with solutions $(d;m)$ to \eqref{eq:dmefIntro}. Assume that $d>0,m\geq 0, \eps t_\mu>0,$ and $q_{x\mu}>0$. Under these assumptions, $xQ(\Tt)$ is well-defined if
        \[ \eps b>\eps 1/3.\] 
  \item[{(iv)}] $\ovx Q(\Tt)$ is well-defined if $q_{\ovy\la}<0$, or equivalently $p_{\ov{x}\rho}-6q_{\ov{x}\rho}<0.$
        \item[{(v)}] $\ovy Q(\Tt)$ is well-defined if $q_{\ovx\rho}<0$, or equivalently $p_{\ovy\la}>0$
    \end{itemlist}
\end{lemma}

\begin{proof}
    For (i), let $X_{cr}$ denote the intersection of $\vn_Y$ and the line through $X$ and $V$. In the proof of Lemma~\ref{lem:yLengths}, we show that the solution $c$ to the following equation
    \[ |OY|(0,1)+k\vn_Y=|OX|(1,0)+c\ovr{XV}\]
    is 
    \[c=\frac{|OY|q_\la-|OX|p_\la}{p_\la(1+p_\rho q_\rho-6q_\rho^2)+q_\la q_\rho^2}=\eps\left(\frac{m_\mu-d_\mu b}{q_\rho q_{y\mu}}\right).\] Hence, if $c>0$, then $yQ(\Tt)$ is well-defined and if $c<0,$ then $sQ(\Tt)$ is well-defined. 
Note that for $Q(\Tt)$ to be well-defined, then $q_\rho$ must be positive as if not $V$ would have negative $y$-coordinate. The result follows. 

    For (ii), if $p_\rho/q_\rho>6,$ then the nodal ray $\vn_X$ has positive slope, so $xQ(\Tt)$ will be well-defined for any $b$ such that $Q(\Tt)$ is well-defined. 
    
    For (iii), similarly to the argument in (i), we have $xQ(\Tt)$ is well-defined from Lemma~\ref{lem:xLengths}, we have that if \[\frac{|OY|(p_\rho-6q_\rho)+|OX|q_\rho}{(p_\rho-6q_\rho)(p_\la q_\la-1)+q_\rho q_\la^2}=\eps\left(\frac{m_\mu'-d_\mu'b}{q_\la q_{x\mu}}\right)>0\]
    then $xQ(\Tt)$ is well-defined and if it's negative, then $\ov{s}Q(\Tt)$ is well-defined. As we are assuming that $q_\la,q_{x\mu}>0$, the mutation $x$ will be well-defined if $\eps (m_\mu'-d_\mu'b)>0$.  For $\bm{x}=(p,q,-t_\mu)$ and $S\bm{x}=(p_\mu,q_\mu,t_\mu)$, we have that $m_\mu'=-m$ and $d_\mu'=-d$ where $(d,m)$ satisfy \eqref{eq:dmefIntro} for $\bm{x}$. We get that
    \[ \eps(-m+db)=\eps (m_\mu'-d_\mu'b)>0\] implying that $x$ will be well-defined if $\eps b>\eps m/d$. As $\eps t_\mu>0$, by Lemma~\ref{lem:md13}, we have that
    \[ \eps 1/3>\eps m/d.\]
    The statement then follows as we get the chain of inequalities $\eps b>\eps 1/3>\eps m/d.$

    For (iv) and (v), we are considering which side of the quadrilateral $\vn_V$ will extend to intersect. By definition, $\vn_V=\begin{pmatrix} -q_{\ovy\la} \\ p_{\ovy\la} \end{pmatrix}=-\begin{pmatrix} p_{\ovx\rho}-6q_{\ovx\rho} \\ q_{\ovx\rho} \end{pmatrix}$. First, note that it is not the case that $q_{\ovy\la}<0$ and $q_{\ovx\rho}<0$ because then $\vn_V$ would not point into the interior of $Q(\Tt)$, so $Q(\Tt)$ would not be well-defined. 

    If $q_{\ovy\la}<0$ (resp. $q_{\ovx\rho}<0$) then $\vn_V$ will extend to intersect $\ov{OX}$  (resp. $\ov{OY}$) due to the positivity of the first (resp. second) coordinate of $\vn_V$, so $\ov{x}Q$ (resp. $\ov{y}Q$) is well-defined.

\end{proof} 

We give the analogous statement for when $wQ_{P_b}(\Tt)$ is well-defined. 
\begin{lemma} \label{lem:definedP}
    Let $\Tt=(\bm{x}_\la,\bm{x}_\mu,\bm{x}_\rho) \in \X_P^3$ be a recursive triple where $\eps=1$ if $\Tt$ is sign-matching and otherwise, $\eps=-1$. Assume $Q(\Tt)$ is well-defined for some $b$-value. Then
    \begin{itemlist}
        \item[{(i)}]  If \(q_{y\mu}>0\) and $f_\mu>0$, then \(sQ(\Tt)\) is well-defined precisely when
\[
        \eps\frac{e_\mu}{f_\mu}<\eps b,
\]
and \(yQ(\Tt)\) is well-defined precisely when
\[
        \eps\frac{e_\mu}{f_\mu}>\eps b.
\]
If \(q_{y\mu}<0\), the two inequalities are reversed.

        \item[{(ii)}] $xQ(\Tt)$ is well-defined if $p_\rho/q_\rho>6$
        \item[{(iii)}] Let $S^{-1}(\bm{x}_\mu):=\bm{x}=(p,q,-t_\mu) \in\X_P$ with solutions $(e,f)$ to \eqref{eq:dmefIntro}. If $(e,f)$ are both positive and $\eps t_\mu>0$, then $xQ(\Tt)$ is well-defined if
        \[ \eps b>\eps.\]
  \item[{(iv)}] $\ovx Q(\Tt)$ is well-defined if $q_{\ovy\la}<0$, or equivalently 
  $p_{\ov{x}\rho}-6q_{\ov{x}\rho}<0.$
        
        \item[{(v)}] $\ovy Q(\Tt)$ is well-defined if $q_{\ovx\rho}<0$, or equivalently $p_{\ovy\la}>0$

    \end{itemlist}
\end{lemma}
\begin{proof}
    The proof is similar to Lemma~\ref{lem:definedH}. 
\end{proof}

From Lemma~\ref{lem:definedH}~(i) and Lemma~\ref{lem:definedP}~(i), in checking whether an $s$ or a $y$-mutation is well-defined, we compare the ratio $e/f$ and $m/d$ to $b$. For a class $(d;m;p,q)$ or $(e,f;p,q),$ we will let $\al:=m/d$ or $\al:=e/f$ where the context will be clear whether we are talking about a $P$-class or a $H$-class. This motivates the following definition. 
\begin{definition}
For a recursive triple $\Tt:=(\bm{x}_\la,\bm{x}_\mu,\bm{x}_\rho)$ where $\eps=1$ if $\Tt$ is sign-matching and otherwise, $\eps=-1$, we will say that $\Tt$ is {\bf $\al$-ordered} if the following condition holds:
\[\eps\al_\rho,\eps\al_\la>\eps\al_\mu.\]

\end{definition} 

\begin{example}
	The triple $\Tt_0^H$ and $\Tt_0^P$ are $\al$-ordered. Note, in \cite{MMW}, in their definition of a recursive triple, they required the triple to be $\al$-ordered. Here, we do not make this requirement as it is not the case that all possible mutations of $\Tt_0^H$ and $\Tt_0^P$ are $\al$-ordered.

\end{example}

We now give some relevant lemmas to show that all triples in $\mathcal{W}_H,\mathcal{W}_P$ are $\al$-ordered.

\begin{lemma} \label{lem:mDrelations} 
    Let $\Tt=(\bm{x}_\la,\bm{x}_\mu,\bm{x}_\rho)$ be a recursive $H$-triple. If $\Tt$ is sign-matching, let $\eps=1$ and otherwise let $\eps=-1.$ The following identities hold:
    \begin{align*}
        m_\rho d_\mu-m_\mu d_\rho&=\eps p_\la \\
        m_\la d_\mu-m_\mu d_\la&=\eps q_\rho \\
        m_\rho d_\la-m_\la d_\rho&=\eps p_{\ov y\la}=-\eps q_{\ovx\rho} 
    \end{align*}
\end{lemma}
\begin{proof}
  For the first identity, converting the $(d;m)$-coordinates to the $(p,q,t)$-coordinates, we have
    \[ \begin{vmatrix} m_\rho & m_\mu \\  d_\rho & d_\mu \end{vmatrix}=1/8((p_\mu+q_\mu)t_\rho-(p_\rho+q_\rho)t_\mu) ,\]
    so we want to show that $(t_\rho q_\mu-q_\rho t_\mu)+(p_\mu t_\rho-p_\rho t_\mu)=8p_\la$. This follows from the identities in Lemma~\ref{lem:identS}. A similar argument applies to the second identity.    The third is implied by the first two by applying an $\ov{x}$ or $\ov{y}$ mutation to the triple $\Tt.$

\end{proof}

\begin{lemma} \label{lem:eFrelations}
       Let $\Tt=(\bm{x}_\la,\bm{x}_\mu,\bm{x}_\rho)$ be a recursive $P$-triple. If $\Tt$ is sign-matching, let $\eps=1$ and otherwise let $\eps=-1.$ The following identities hold:
    \begin{align*}
        e_\rho f_\mu-e_\mu f_\rho&=\eps p_\la \\
        e_\la f_\mu-e_\mu f_\la&=\eps q_\rho \\
        e_\rho f_\la-e_\la f_\rho&=\eps p_{\ov y\la}=-\eps q_{\ovx\rho} 
    \end{align*}
\end{lemma}
\begin{proof}
    This proof is similar to Lemma~\ref{lem:mDrelations}. 
\end{proof}

\begin{cor} \label{cor:ordering}
    Assume that $\Tt=(\bm{x}_\la,\bm{x}_\mu,\bm{x}_\rho)$ is a positive $P$- or $H$-triple. Then, $\Tt$ is $\al$-ordered. 
    Further, if $(p_{\ovx\rho},q_{\ovx\rho})$ are positive, then
    \[ \eps \al_\la >\eps \al_\rho\] and if  $(p_{\ovy\la},q_{\ovy\la})$ are positive, then
    \[ \eps \al_\rho >\eps \al_\la\]
    There is equality exactly when $p_{\ovy\la}=q_{\ovx\rho}=0.$
\end{cor}
\begin{proof}
    This follows immediately from Lemma~\ref{lem:mDrelations} if $\Tt$ is a $H$-triple and from Lemma~\ref{lem:eFrelations} if $\Tt$ is a $P$-triple. 
\end{proof}

\begin{cor} \label{cor:alOrdered} 
For each $w \in \mathcal{W}_H$, the triple $w\Tt_0^H$ is $\al$-ordered. For each $w \in \mathcal{W}_P$, the triple $w\Tt_0^P$ is $\al$-ordered. 
\end{cor} 
\begin{proof}

    This follows from Corollary~\ref{cor:Positive} and Corollary~\ref{cor:ordering}. 
\end{proof}

We now begin by showing that for the words $w \in \mathcal{W}_H$, there is an interval of $b$-values such that $wQ_{H_b}^0$ is well-defined. This first two lemmas concern the classes $s^k\Tt_0^H.$
\begin{lemma} \label{lem:salphaOr}
    For triples of the form $s^k\Tt_0^H:=(\bm{x}_k,\bm{x}_{k+1},\bm{x}_{k+2}),$ let $\eps_k=1$ if $k$ is even and $\eps_k=-1$ if $k$ is odd. Then, for all $k$, letting $\al_k=m_k/d_k$, we have 
    \[\eps_k \al_{k+1} < \eps_k\tfrac{1}{3} < \eps_k\al_{k+2} \leq \eps_k \al_{k}   \]
\end{lemma}
\begin{proof}
    By Proposition~\ref{prop:genTrip}, we have that $s^k\Tt_0^H$ is sign-matching if $k$ is even and sign-alternating if $k$ is odd. Further, as $|t_k|$ stays constant under $s$ and changes sign after each $s$, by Lemma~\ref{lem:md13}, we have that \[\eps \al_{k+1} <\eps\tfrac{1}{3} < \eps\al_{k+2},\eps \al_{k}.\] It remains to check that $\eps\al_{k+2}\leq \eps \al_{k}$. As $(d_k;m_k)$ are positive for each $k,$ this holds if $q_{k,\ovx\rho} \geq 0$ where $\ov{x}s^k\Tt_0:=(\bm{x}_k,\bm{x}_{k+2},\bm{x}_{k,\ovx\rho})$ by Corollary~\ref{cor:ordering}. We show this is the case. 
    
    For $s^k\Tt_0$, we have the following properties: 
    \begin{itemlist}
        \item[{(i)}] $\{p_k\}_{k \geq 0}$ and $\{q_k\}_{k \geq 0}$ are nondecreasing sequences. 
        \item [{(ii)}] $p_k/q_k < p_{k+1}/q_{k+1} < p_{k+2}/q_{k+2}$
    \end{itemlist}
     Property (ii) implies that the recursion parameter, $\nu_k:=p_{k+2}q_{k+1}-p_{k+1}q_{k+2}$, for the $\ov{x}$-mutation is positive. Note that $\nu_k$ is either $1,2$, but we will not use this fact. Then, the positivity of $\nu_k$ along with property (i) implies that \[(p_{k,\ovx\rho},q_{k\ovx\rho})=(\nu_k p_{k+2}-p_{k+1},\nu_k q_{k+2}-q_{k+1})\] will both be nonnegative as desired.

\end{proof}

\begin{lemma} \label{lem:sMutDefinedH}
    For the triple $s^k\Tt_0^H$, let $(\al_k,\al_{k+1},\al_{k+2})$ correspond to the ratios $m/d$ of the degree coordinates. The mutation sequence $s^kQ_{H_b}^0$ is well-defined for $b$ if $b$ satisfies
    \[ (-1)^{k}\al_{k-1} < (-1)^{k} b < (-1)^{k}\al_{k} .\]
\end{lemma}
\begin{proof}
 Note that by Lemma ~\ref{lem:salphaOr}, we have that 
	\begin{equation}\label{eq:alt}
	    \eps_k\al_{k+1} <\eps_k \tfrac13 < \eps_k \al_{k+2} \leq \eps_k\al_{k}
	\end{equation}  
    where $\eps_{k}=1$ if $s^{k}\Tt_0^H$ is sign-matching and $-1$ otherwise. By Proposition~\ref{prop:genTrip}, we have that $s$-mutations alternate between sign-matching and sign-alternating where $\Tt_0^H$ is sign-matching. Hence, $\eps_{k}=(-1)^{k}.$

    By Lemma~\ref{lem:definedH}~(i), assuming that $s^{k-1}Q_{H_b}^0=(\bm{x}_{k-1},\bm{x}_k,\bm{x}_{k+1})$ is well-defined, then $s^kQ_{H_b}^0$ is well-defined if  
    \[ \eps_{k-1} \al_{k} <\eps_{k-1} b.\] Note, for $ys^k\Tt_0^H$, we have $q_{y\mu}$ is positive as $q_k\leq q_{k+1}$ and $p_k/q_k<p_{k+1}/q_{k+1}$. 
    
 Combining this with \eqref{eq:alt}, we get that $s^kQ_{H_b}^0$ is well-defined if 
  \[ (-1)^{k}\al_{k-1} <  (-1)^{k} b < (-1)^{k}\al_{k} .\]
\end{proof}

The sequences considered in the next lemma corresponds to the classes $w\mathcal{B}_{2n}$ as seen in Proposition~\ref{lem:perf4H}. 
\begin{lemma} \label{lem:bDefH1}
Let $w=w'y\ov{x}^k$ where $w'$ is a word containing only $x,y$, and $w\Tt_0^H=:(\bm{x}_\la,\bm{x}_\mu,\bm{x}_\rho).$ For $b\in (\frac{k-1}{k},m_\la/d_\la)$, the mutation sequence $wQ_{H_b}^0$ is defined. 
\end{lemma}
\begin{proof}
    By \cite[Proposition 3.9]{M1}, we have that 
    $y\ov{x}^kQ_{H_b}^0$ is defined if and only if $\frac{k-1}{k} < b < \frac{k}{k+1}$. To show that $y\ov{x}^kQ_{H_b}^0$ is defined for the claimed $b$-interval, it suffices to show that $\frac{k-1}{k}<\al_\la<\frac{k}{k+1}$. 
    
    We first check the upper bound. We can compute by a simple induction argument that the right entry of $y\ov{x}^k\Tt_0^H$ is $(2k+2,1,2k-1)$ with solutions $(d;m)=(1+k;k).$ By Corollary~\ref{cor:alOrdered}, the triples we obtain from doing $x,y$-mutation to $y\ov{x}^k\Tt_0^H$ will all be $\al$-ordered and sign matching. Hence, the $\al$-value for the middle entry will form a strictly decreasing sequence, and we have the $\al_\la<\frac{k}{1+k}.$ 

    To see why $\frac{k-1}{k}<\al_\la$, it does not suffice to make an $\al$-ordering argument. Instead, this follows from \cite[Lemma 5.4(i,ii)]{M1}. Hence, we conclude that $y\ov{x}^kQ_{H_b}^0$ is defined for the claimed $b$-interval. 

    The argument about $\al$-ordering also explains why each $y$-mutation is well-defined by  Lemma~\ref{lem:definedH}~(i). Lastly, each $x$ mutation is well-defined as by Lemma~\ref{lem:pOverq6}, for all $i,$ as the triple $\Tt_i$ has $p_{\bullet,i}/q_{\bullet,i}>6$ for $\bullet=\la,\mu,\rho,$ this implies that the nodal ray $\vn_X$ has positive slope, so each $x$-mutation is well-defined. See the criteria in Lemma~\ref{lem:definedH}~(ii). 
 
\end{proof}

\begin{rmk}
    Note that in Lemma~\ref{lem:bDefH1}, we have that $wQ^0_{H_b}$ is well-defined for $b \in (1/3,m_\la/d_\la).$ Note that $wQ^0_{H_b}$ is not well-defined for $m_\la/d_\la$ as then $|XV|=0$. This implies that the nodal ray at some point intersected the vertex. 
\end{rmk}

The sequence considered in the next lemma corresponds to the classes $wS^j(\mathcal{B}_{2n})$ as seen in Proposition~\ref{lem:perf4H}. 
\begin{lemma} \label{lem:bDefH2}
Let $w=w'y\ov{x}^ks^{3j}$ where $j \geq 1,$ $w'$ is a word containing only $x,y$, and $w\Tt_0=:(\bm{x}_\la,\bm{x}_\mu,\bm{x}_\rho).$ Then, for \[ (-1)^j \tfrac13 < (-1)^jb<(-1)^jm_\la/d_\la,\]
 $wQ_{H_b}^0$ is defined. 
\end{lemma}
\begin{proof}

    For ease of notation, we begin by assuming that $3j$ is even. A very similar argument holds for the case where $3j$ is odd. This assumption implies by Proposition~\ref{prop:genTrip} that $s^{3j}\Tt_0$ is a sign-matching recursive triple. 

    {\bf Claim 1:} $s^{3j}Q_{H_b}^0$ is defined for the $b$-interval. 
    Let $s^{3j}\Tt_0^H:=(\bm{x}_{3j},\bm{x}_{3j+1},\bm{x}_{3j+2})$ with corresponding ratios $m_j/d_j=\al_j.$ By Lemma~\ref{lem:sMutDefinedH}, we have that $s^{3j}Q_{H_b}^0$ is well-defined if $   \al_{3j-1} < b < \al_{3j}.$ Note, by Lemma~\ref{lem:salphaOr} that $\al_{3j-1}<\tfrac13<\al_{3j}$. Hence, $s^{3j}Q_{H_b}^0$ will be defined for the $b$-interval assuming that \begin{equation} \label{eq:sDefined}
        \tfrac13<\al_\la<\al_{3j}.
    \end{equation}
 
	Let us denote the ratios corresponding to the triple $\ov{x}^ks^{3j}\Tt_0^H$ as 
	$(\al_{3j},\al_{j,k},\al_{j,k+1})$
	where $\al_{3j+1}=\al_{j,-1}$ and $\al_{3j+2}=\al_{j,0}$. By Corollary~\ref{cor:alOrdered}, we have that $\ov{x}^ks^{3j}\Tt_0$ is $\al$-ordered for each $k$-value. Hence, we get the sequence:  
	\[ 1/3 < \al_{3j+2}=\al_{j,0} < \al_{j,1} < \hdots <\al_{j,k} < \al_{j,k+1} < \al_{3j}.\]
  
	Then, we take a $y$-mutation the triple of the $m/d$ ratios is  $(\al_{j,k},\al_{y,j,k},\al_{j,k+1})$,
	which remains $\al$-ordered implying that 
	\[ \al_{y,j,k}< \al_{j,k} <\al_{j,k+1} < \al_{3j}.\]
    Note that by Corollary~\ref{cor:triples}, the $t$-value for each entry of $y\ov{x}^ks^{3j}\Tt_0$ is positive as we are assuming $j$ is even, so $1/3<\al_{y,j,k}$ by Lemma~\ref{lem:md13}. 
    For the $x,y$-mutations in $w'$, the triples remain $\al$-ordering with positive $t$-values, so the left entry will be
    \[ 1/3<\al_\la \leq \al_{j,k}<\al_{3j}\]
    as desired in \eqref{eq:sDefined} for Claim 1.

    {\bf Claim 2:} $\ov{x}^ks^{3j}Q_{H_b}^0$ is defined for the claimed $b$-interval.

    Recall we are assuming that $j>0.$ In this case, we have that \[\ov{x}^ks^{3j}\Tt_0^H=S^j\left((1,1,2),(2+k,1,k-1),(4+k,1,1+k)\right),\] 
    so each entry in the triple has $p/q \leq 6$ as $S(p/q)=6-q/p.$ Hence, by Lemma~\ref{lem:definedH}, for any $k$, $\ov{x}^ks^{3j}Q_{H_b}^0$ is defined for any $b$ such that $s^{3j}Q_{H_b}^0$ is well-defined. This proves Claim 2 when $j>0.$

    {\bf Claim 3:} $w'y\ov{x}^ks^{3j}Q_{H_b}^0$ is defined for the claimed $b$-interval.

    The $y$-mutations being well-defined follows due to the $\al$-ordering of the triples and Lemma~\ref{lem:definedH}~(i).

     We now consider each $x$-mutation in $w'.$ We check that in this case the assumptions of Lemma~\ref{lem:definedH}~(iii) are satisfied.  If $j> 0,$ then $S^{-1}(w_i\hdots w_1 y \ov{x}^ks^{3j}\Tt_0^H)=w_i\hdots w_1 y \ov{x}^ks^{3(j-1)}\Tt_0^H \in \mathcal{W}_H$, so is positive by Corollary~\ref{cor:Positive}. By Corollary~\ref{cor:triples}~(iv), each $t$-value of $w_i\hdots w_1 y \ov{x}^ks^{3j}\Tt_0^H$ is positive and $\eps=1$ as the triples are sign-matching. Hence, the assumptions of Lemma~\ref{lem:definedH}~(iii) are satisfied. This implies that each $x$-mutation in $w'$ will be well-defined assuming that $b>1/3,$ which does hold for the claimed $b$-interval.  
\end{proof} 

The sequence considered in the next lemma corresponds to the classes $wS^jR(\mathcal{B}_{2n})$ for $j \geq 1$ as seen in Proposition~\ref{lem:perf4H}. 

\begin{lemma} \label{lem:bDefH3}
Let $w=w'yx^ny\ov{x}s^{3j+1}$ where $w'$ is a word containing only $x,y$. Let $w\Tt_0=(\bm{x}_\la,\bm{x}_\mu,\bm{x}_\rho)$ and $\eps=(-1)^{j+1}.$ Then, for
\[\eps\tfrac13<\eps b< \eps\tfrac{m_\la}{d_\la},\]
 $wQ_{H_b}^0$ is defined. 
\end{lemma}
\begin{proof}
Note that the value of $\eps=1$ when $w\Tt_0$ is sign-matching and $-1$ otherwise by Proposition~\ref{prop:genTrip}. 
Denote $s^{3j+1}\Tt_0^H$ by $(\bm{x}_{3j+1},\bm{x}_{3j+2},\bm{x}_{3j+3})$ with corresponding ratios $m/d$ given by $(\al_{3j+1},\al_{3j+2},\al_{3j+3})$.

{\bf Claim 1:} $s^{3j+1}Q_{H_b}^0$ is defined for the claimed $b$-interval. 
By Lemma~\ref{lem:sMutDefinedH}, we have that $s^{3j+1}Q_{H_b}^0$ is well-defined if  
    \[ \eps \al_{3j} < \eps b < \eps \al_{3j+1}.\]

This follows similarly to Claim 1 in Lemma~\ref{lem:bDefH2} where we use the $\al$-ordering of the triples. In this case, we get that
\[ \eps 1/3<\eps \al_\la <\eps \al_{3j+1}.\]

        {\bf Claim 2: $\ov{x}s^{3j+1}Q_{H_b}^0$ is well-defined for the claimed $b$-interval}
        To check that we can perform $\ov{x}$ to $s^{3j+1}Q_{H_b}^0,$ we use the fact that $\ov{x}s^{3j+1}Q_{H_b}^0$ has $p_{\ovx\rho}-6q_{\ovx\rho}<0$, see Proof of Proposition~\ref{lem:perf4H}. 

        {\bf Claim 3: Each $y$-mutation is well-defined for the claimed $b$-interval.}
        By Lemma~\ref{lem:definedH}~(i), this again follows by $\al$-ordering. Note that the last $y$-mutation we perform will move $\bm{x}_\la$ to the left entry. By Lemma~\ref{lem:definedH}~(i) for this to be well-defined we need $\eps m_\la/d_\la>\eps b.$ Previous $y$-mutations will have the middle entry smaller than $\eps m_\la/d_\la.$

        {\bf Claim 4: Each $x$-mutation is well-defined for the claimed $b$-interval.}

The claimed interval implies that $\eps\frac13<\eps b.$
Suppose first that \(j\ge 2\). Before any \(x\)-mutation occurring after the
block \(yx^ny\ov{x}s^{3j+1}\), the \(S^{-1}\)-shift of the relevant middle
entry is obtained by replacing \(s^{3j+1}\) with \(s^{3j-2}=s^{3(j-1)+1}\).
Since \(j-1\ge1\), the resulting triple lies in the family already contained
in \(\mathcal W_H\). Hence it is positive by Corollary~\ref{cor:Positive}.
Moreover, by Corollary~\ref{cor:triples}, all \(t\)-coordinates have sign
\(\eps\). Therefore the hypotheses of Lemma~\ref{lem:definedH}~(iii)
are satisfied, and the inequality $\eps\frac13<\eps b$
implies that each such \(x\)-mutation is well-defined.

It remains to handle the exceptional case \(j=1\). Here the above argument
does not apply, because applying \(S^{-1}\) replaces \(s^4\) by \(s\), and
the triple \(y\ov{x}s\Tt_0^H\) contains the non-positive entry $(1,0,-3)$. However, Lemma~\ref{lem:definedH}~(iii) only uses \(S^{-1}\) of the middle
entry.

For the \(r\)-th \(x\)-mutation in the block \(x^ny\ov{x}s^4\), where
\(0\le r\le n-1\), the \(S^{-1}\)-shift of the middle entry is
\[
       (d;m;p,q;t)= (5(r+1);r;13+12r,\;2+2r,\;-(2r+5)).
\]
Thus \(d>0\), \(m\ge0\), and
\[
        \frac{m}{d}=\frac{r}{5(r+1)}<\frac13.
\]
Since in the case \(j=1\) we have \(\eps=1\) and the claimed interval
gives \(b>1/3\), Lemma~\ref{lem:definedH}~(iii) implies that these
\(x\)-mutations are well-defined.

After the block \(yx^ny\ov{x}s^4\), we have the middle entry of $S^{-1}\big(yx^ny\ov{x}s^4\Tt_0^H\big)$ is $(24n^2+62n+34,4n^2+10n+5,-4n^2-16n-13)$

The middle entry has $(d;m)=(10n^2+25n+13;\, n(2n+3)),$
so it is positive for \(n>0\), and nonnegative for \(n=0\). The same
direct check, together with the ordering of the \(p/q\)-coordinates, shows
that the hypotheses of Lemma~\ref{lem:definedH}~(iii) continue to hold for
the subsequent \(x\)-mutations appearing in the word \(w'\).
	 
\end{proof}

The sequence considered in the next lemma corresponds to the classes $wR(\mathcal{B}_{2n})$ for $n \geq 1$ as seen in Proposition~\ref{lem:perf4H}. 
\begin{lemma} \label{lem:bDefH4}
Let $w=w'yx^{n-1}y\ov{x}\ov{y}s$ where $w'$ is a word containing only $x,y$. Let $w\Tt_0=(\bm{x}_\la,\bm{x}_\mu,\bm{x}_\rho).$ Then, for $b \in (m_\la/d_\la,1/3)$, $wQ_{H_b}^0$ is defined. 
\end{lemma}
\begin{proof}
	We can check by direct computation that $y\ov{x}\ov{y}s$ is defined for $0< b<1/2$. Now, we take $x^{n-1},$ by Proposition~\ref{lem:perf4H}, we have that 
    \[ x^{n-1}y\ov{x}\ov{y}s\Tt_0^H=\left((5,1,-2),(13+12n,2+2n,-(5+2n)),(1+12n,2n,-(3+2n))\right).\]
    In particular, the last entry has $p/q>6$ for all $n.$ Hence, for any $b$-values as $\vn_X$ has positive slope the $x$-mutation will be well-defined on the ATBD. The rest of the argument follows similarly to Lemma~\ref{lem:bDefH3}.

\end{proof}

 The sequence considered in the next lemma corresponds to the classes $wR(\mathcal{B}_{0})$ as seen in Proposition~\ref{lem:perf4H} and Corollary~\ref{cor:rTrip}. 
\begin{lemma} \label{lem:bDefH5}
    Let $w=w'y\ov{x}^{n+2}\ov{y}s$ where $w'$ is a word containing only $x,y.$ Let $w\Tt_0=(\bm{x}_\la,\bm{x}_{\mu},\bm{x}_\rho).$ Then, for $b \in (m_\la/d_\la,1/3)$ $wQ_{H_b}^0$ is defined. 
\end{lemma}
\begin{proof}
    We can check directly that $\ov{x}^{n}\ov{y}sQ_{H_b}^0$ is well-defined for $0<b<d_{n-2}/d_{n-1}$
    where $\ov{x}^{k}\ov{y}s\Tt_0^H=((0,1,-3),(p_{k-1},q_{k-1},t_{k-1}),(p_k,q_k,t_k))$ with corresponding $(d_k;0).$ Note, we can easily verify from the base case that $d_k$ is the odd-index Fibonacci numbers, i.e. $1,2,5,13,34,\hdots$. Hence, $d_{k-1}/d_k >1/3$ for all $k.$ While $\ov{x}^{n}\ov{y}s\Tt_0^H$ is not a positive triple, after performing $y\ov{x}^{n+2}\ov{y}s,$ all other triples are positive by the computation in Lemma~\ref{lem:R'andR} and Lemma~\ref{lem:pOverq6}. Hence, each triple is $\al$-ordered by Corollary~\ref{cor:ordering}.

    For each later $x$-mutation, we have that $\vn_X$ will always have positive slope with the third element having $p/q>6$ by Lemma~\ref{lem:R'andR} and Lemma~\ref{lem:pOverq6}. For each later $y$-mutation the argument is similar as before due to $\al$-ordering. 

\end{proof}
This completes the cases for $w \in \mathcal{W}_H.$ We now move on to the case of $w \in \mathcal{W}_P.$

Recall, the sequence of mutation $\{r_i\}$ from \eqref{eq:SPsequence}, which we use to obtain the triples $S^j\Tt_0^P.$

The sequence considered in the next two lemma corresponds to the classes $S^j\Tt_0^P$ for as seen in Proposition~\ref{lem:perf4P}. 
 
\begin{lemma} \label{lem:alPSordering}
  For triples of the form $r_i\Tt_0^P$ defined above, letting $\al_i=e_i/f_i$,
for $i \geq 1,$ we have
\[ \al_2 < \al_{3i+1} <\al_{3i+4} <\al_0=\al_{3i}=1< \al_{3i+5}<\al_{3i+2}<\al_1.\]
Further, for all $i\geq 0,$ we have that $1/\al_{3i+1}=\al_{3i+2}.$

\end{lemma}
\begin{proof}
    Here, a triple of integers $(t_\la,t_\mu,t_\rho)$ will represent the $t$-coordinate of a triple $\Tt=(\bm{x}_\la,\bm{x}_\mu,\bm{x}_\rho).$
    For $\Tt_0^P,$ the $t$-values are $(2,0,2).$ As taking $s$-mutations switches the sign of the $t$-variables, we have that $r_1\Tt_0^P$, $r_2\Tt_0^P,$ and $r_3\Tt_0^P$ will have the following $t$-values $(0,2,-2),(2,-2,0),(-2,0,-2).$ As these are all positive triples, i.e. they have positive $(e,f;p,q)$-coordinates, we have the Lemma~\ref{lem:md13} implies the statement for $\al_0,\al_1,\al_2$ as for $P$-triples
    \[ e/f>1 \iff t>0, \quad e/f=1 \iff t=0, \quad e/f<1 \iff t<0.\]

   It is immediate by induction that for $i \geq 1,$ the $t$-values of $r_{3k+1}\Tt_0^P,r_{3k+2}\Tt_0^P,r_{3k+3}\Tt_0^P$ are in order $(0,2,-2),(2,-2,0),(-2,0,-2).$ Hence, the ordering follows by Lemma~\ref{lem:md13}.

\end{proof}

\begin{lemma} \label{lem:alPSordering2}
For $i \leq 3,$ if $1 < b < 2$, the mutation $r_iQ_{P_b}^0$ is well-defined. For $k \geq 1,$ the mutation sequences
$r_{3k+1}Q_{P_b}^0,r_{3k+2}Q_{P_b}^0,r_{3k+3}Q_{P_b}^0$ are well-defined if $1 < b < \al_{3k+2}.$

\end{lemma}
\begin{proof}
The claim for $i \leq 3$ can be checked directly. 

For the other claim, note that by Proposition~\ref{prop:genTrip}, we have that if $i \geq 4$, $r_i\Tt_0^P$ is sign-alternating if $i \equiv 1,0 \mod{3}$ and sign-matching otherwise. Hence, the claim follows by the ordering of the $\al_i$ given in Lemma~\ref{lem:alPSordering} and Lemma~\ref{lem:definedP}~(i), which states when a $y$-mutation or $s$-mutation is well-defined.

\end{proof}

 The sequence considered in the next two lemma corresponds to the classes $w\mathcal{B}_{2n+1}$ for as seen in Proposition~\ref{lem:perf4P}. 
\begin{lemma} \label{lem:bDefP1}
Let $w=w'y\ov{x}^k$ where $w'$ is a word containing only $x,y$. Let $w\Tt_0^P=(\bm{x}_\la,\bm{x}_\mu,\bm{x}_\rho).$ Then, for $b \in (k,e_\la/f_\la)$, $wQ_{P_b}^0$ is well-defined. 
\end{lemma}
\begin{proof}
    A straightforward proof by induction shows that $\ov{x}^kQ_{P_b}^0$ is well-defined for $b>k.$ Note that $\ov{x}^k\Tt_0^P=\left((1,1,2),(2k+3,1,2k),(2k+5,1,2k+2)\right)$
    with corresponding $(e,f)$-values $((1,0),(k+1,1),(k+2,1)).$ As $\ov{x}^k\Tt_0^P$ is sign-matching, by Lemma~\ref{lem:definedP}~(i),  $y\ov{x}^kQ_{P_b}^0$ is well-defined for $k<b<k+1,$ so we must check that $k<e_\la/f_\la<k+1.$

    We compute that $y\ov{x}^k\Tt_0^P$ has $(e,f)$-values 
    \[ \left((k+1,1),(2k^2+4k+1,2k+2),(k+2,1)\right).\]
    This triple is $\al$-ordered, and we have
    \[ \frac{2k^2+4k+1}{2k+2}<k+1<k+2.\]
    As we apply the mutations in $w$, the triples remain $\al$-ordered, so for each new $\al$ appearing, we have that  $\al< k+1<k+2$ as desired implying that $y\ov{x}^kQ_{P_b}^0$ is well-defined for the claimed $b$-interval. Hence, $e_\la/f_\la \leq k+1.$

    Since the triples are $\al$-ordered by Lemma~\ref{lem:definedP}~(i), each $y$-mutation in $w$ is well-defined. To see that each $x$-mutation, we claim that $\vn_X$ always has positive slope for each mutation after $y\ov{x}^{k+2}.$ To see this note that \[y\ov{x}^{k+2}\Tt_0^P=((2n+7,1,4+2n),(4n^2+26n+41,2n+5,4n^2+20n+22),(2n+9,1,2n+6))\]
    and by Corollary~\ref{cor:triples}~(iii) we have that the $p/q$-coordinates will continue to satisfy $2n+7 \leq p/q \leq 2n+9.$ Hence, the slope of $\vn_X$ will be positive. 
\end{proof}

 The sequence considered in the next lemma corresponds to the classes $I_t^{j-1}S^j(\mathcal{B}_{2n+1})$ for $n \geq 0$ as seen in Proposition~\ref{lem:perf4P}. 
\begin{lemma} \label{lem:poly1Q}
    Assume that $i\geq 1.$
    We have $wy\ov{x}^kr_{3i}Q_{P_b}^0$ where $w$ is a word of $x,y$ is well-defined for $b \in (1,f_\la/e_\la)$ where $wy\ov{x}^kr_{3i}\Tt_{P}^0:=(\bm{x}_\la,\bm{x}_\mu,\bm{x}_\rho).$
\end{lemma}
\begin{proof}

    First note, as $i \geq 1$, the triple $wy\ov{x}^kr_{3i}$ is sign-alternating for all $i$, so by Corollary~\ref{cor:triples}~(iv) and Lemma~\ref{lem:md13}, we have that $t_\la$ is negative implying that $e_\la/f_\la=\al_\la<1.$ Here, note the notation convention implies the endpoint of the claimed interval is $1/\al_\la=f_\la/e_\la$.  

    \textbf{Claim 1:} $r_{3i}Q_{P_b}^0$ is well-defined. 
    We have that by Lemma~\ref{lem:alPSordering2}, $r_{3i}Q_{P_b}^0$ is well-defined if $1<b<\al_{3i-1}$ where $r_{3i}\Tt_0^P=(\bm{x}_{3i+1},\bm{x}_{3i},\bm{x}_{3i+4})$ where:
  
    \[ 1/\al_{3i-1}=\al_{3i-2}<\al_{3i+1}<\al_{3i+4}<\al_{3i}=1.\]
    Hence, to show Claim 1, it suffices to show that $\al_{3i-2}<\al_\la<1$ as this implies that $1<1/\al_\la=f_\la/e_\la<\al_{3i-1}.$
Let \[\ov{x}^kr_{3i}\Tt_0^P:=(\bm{x}_{3i+1},\bm{x}_{3i,k-1},\bm{x}_{3i,k})\]
where $\bm{x}_{3i,-1}=\bm{x}_{3i}$ and $\bm{x}_{3i,0}=\bm{x}_{3i+4}.$
By Corollary~\ref{cor:alOrdered}, these triples are $\al$-ordered, so we have $\{\bm{x}_{3i,k}\}_k$ is a decreasing sequence and for all $k >1$, $\al_{3i,k} \in [\al_{3i+1},\al_{3i+4}].$
Taking a $y$-mutation, we get 
\[y\ov{x}^kr_{3i}\Tt_0^P:=(\bm{x}_{3i,k-1},\bm{x}_{y,3i,k-1},\bm{x}_{3i,k})\] is $\al$-ordered, so $\al_{3i+1}<\al_{3i,k}<\al_{y,3i,k-1}.$
When we keep taking $x,y$-mutations in $w,$ the triples will remain $\al$-ordered, so we conclude that $\al_{3i-2}<\al_{3i+1}\leq \al_\la$ as desired. 

    \textbf{Claim 2:} $\ov{x}^kr_{3i}Q_{P_b}^0$ is well-defined. Here, we show that the conditions of Lemma~\ref{lem:definedP}~(iv) hold. It suffices to show that for the triples $\ov{x}^kr_{3i}\Tt_0^P$ each entry has $p/q<6.$ This follows by Proposition~\ref{lem:perf4P} as on the $(p,q)$-coordinates $\ov{x}^kr_{3i}\Tt_0^P=S^j(\ov{x}^k\Tt_0^P)$ and $S(p/q)=6-q/p.$ 

   \textbf{Claim 3:} Each $y$-mutation after performing $r_{3i}$ in $wy\ov{x}^kr_{3i}Q_{P_b}^0$ is well-defined. Note in this case by Lemma~\ref{lem:definedP}~(i) we need $\al_\la<b,$ but as $\al_\la<1$ this always holds.

   \textbf{Claim 4:} Each $x$-mutation in $wy\ov{x}^kr_{3i}Q_{P_b}^0$ is well-defined. 
    For each $x$-mutation, as we assumed that $i \geq 1,$ we have that \[S^{-1}\ov{x}^kr_{3i}\Tt_0^P= \ov{x}^k(s^2y)^i\Tt_0^P.\]
    Note, for some fixed $i$, that $\ov{x}^k(s^2y)^i\Tt_0^P$ is a sign-matching triple for all $k$ and the $t$-values for these triples will be positive, so we have that $e/f>1.$ Hence, they are well-defined by Lemma~\ref{lem:definedP}~(iii). 
    \end{proof}

     The sequence considered in the next lemma corresponds to the classes $R(\mathcal{B}_{2n+1})$ for $n \geq 0$ as seen in Proposition~\ref{lem:perf4P}.
\begin{lemma} \label{lem:bDefP3}
    We have that $wyx^ny\ov{x}Q_{P_b}^0$ where $w$ is a word of $x,y$ is well-defined for $b \in (1,e_\la/f_\la)$ where  $wyx^ny\ov{x}\Tt_0^P:=(\bm{x}_\la,\bm{x}_\mu,\bm{x}_\rho).$
\end{lemma}
\begin{proof}
    First note that $e_\la/f_\la>1$ as the triple is sign-matching, so $t>0$, which by Lemma~\ref{lem:md13} implies $e_\la/f_\la>1.$ We can check directly that $y\ov{x}$ is well-defined if $1<b<2.$ The sequence of $x$-mutations will be well-defined as $\vn_X$ has a positive slope as the continued fractions of $p/q$ of the triples $x^ny\ov{x}\Tt_0^P$ are
    \[ \left([5],[6,2n+3],[6,2n+1]\right).\] 
    As usual, the remaining $y$-mutations will be defined as each triple is $\al$-ordered. The remaining $x$-mutations in the triple will be well-defined as the slope of $\vn_X$ will continue to be positive as $yx^ny\ov{x}\Tt_0^P=R(\bB_{2n+1})$, which has $p/q > 6$ for the middle entry, see Lemma~\ref{lem:pOverq6}. 
\end{proof}

   The sequence considered in the next lemma corresponds to the classes $I_t^{j-1}S^jR(\mathcal{B}_{2n+1})$ for $n \geq 0$ as seen in Proposition~\ref{lem:perf4P}.
    \begin{lemma} \label{lem:bDefP4}
        For $j \geq 1$, we have that $wyx^ny\ov{x}r_{3j}Q_{P_b}^0$ where $w$ is a word of $x,y$ is well-defined for $b \in (1,f_\la/e_\la)$ where  $wyx^ny\ov{x}r_{3j}\Tt_0^P:=(\bm{x}_\la,\bm{x}_\mu,\bm{x}_\rho).$
    \end{lemma}
    \begin{proof}

The argument that $r_{3j}Q_{P_b}^0$ is defined for the $b$-interval follows closely to Claim 1 in Lemma~\ref{lem:poly1Q}. To see that $\ov{x}r_{3j}Q_{P_b}^0$ is defined note that the continued fractions of the $p/q$-coordinates of $r_{3j}\T_0^P$ are given by 
\[\left([5,\{1,4\}^j],[5,\{1,4\}^j,1,2],[5,\{1,4\}^{j+1}]\right).\] This can be checked by induction. Then, we have that $p_{\ovx\rho}/q_{\ovx\rho}$ will have continued fraction $[5,\{1,4\}^j,1,6]$ as the recursion parameter is $2$. Hence, by Lemma~\ref{lem:definedP}~(iv), $\ov{x}r_{3j}Q_{P_b}^0$ is well-defined. 

The claim that the remaining $y$-mutations and $x$-mutations are well-defined follows similarly to Claim 3 and Claim 4 of Lemma~\ref{lem:poly1Q}. 
    \end{proof}

\section{Connection to Embeddings}\label{ss:embed4}
In this section, we use the idea of a visible obstruction from McDuff-Siegel to prove Theorem~\ref{thm:ATFvis} and Corollary~\ref{cor:ATFvis}.

\subsection{Embeddings and ATBDs}
We begin by giving some background information to motivate why the ATBDs give us embedding results. 

Let $OXVY$ be an almost toric base diagram for $M=H_b,P_b$ where $O$ is the standard Delzant corner at the origin. Then, by \cite{AADT,CV} for all $0<\eps<1$, there exists an embedding of 
\[ (1-\eps)E(|OX|,|OY|) \sembeds M.\]
This embedding corresponds to an upper bound for one particular point of the function $c_{M}(z).$ For $\bullet=P,H$, we often let $c_{\bullet,b}(z)$ denote the function for $H,P$ accordingly. If this embedding corresponds to a point on the volume obstruction, we can conclude that it is sharp. Before moving on to visible obstructions, we will first prove Proposition~\ref{prop:fullFilling}. First, we give an example of this in the case $q=1$ and then give the general proof.
\begin{example} \label{rmk:fullFilling}
    This case was pointed out by Dusa McDuff. Note that in the case where $(p,q)=(2n,1)$, this full filling can be seen immediately from the toric picture using the Traynor trick \cite{T}. In particular, if $q=1,$ Proposition~\ref{prop:fullFilling} is 
\[ E(1,p-1) \sembeds dH_{m/d}.\]
In this case, \eqref{eq:dmefIntro} give $p=2d$ and $m=d-1.$ The vertices of the moment polygon of $dH_{m/d}$
are $(0,0),(d,0),(d-1,1),(0,1)$ using that $d-m=1.$ We can clearly cut this into $2d-1=p-1$ standard triangles of size $1.$ In the case where $q \neq 1,$ the claim is not immediate from the toric pictures. 
\end{example}

\begin{proof}[Proof of Proposition \ref{prop:fullFilling}]
We first note that every perfect class $(p,q)$ in $\T(H)$ appears as the left entry of some triple $\Tt=((p_\la,q_\la,t_\la),(p_\mu,q_\mu,t_\mu),(p_\rho,q_\rho,t_\rho)) \in \T(H)$. From Theorem~\ref{thm:HTrip}, there is some word $w$ such that \[wQ_{H_b}^0=Q_{H_b}(w\Tt_0^H)=Q_{H_b}(\Tt):=OXVY\] is defined for some interval of $b$-values denoted $(b_1,b_2)$ where $m_\la/d_\la$ is either $b_1,b_2.$ Hence, for all $b \in (b_1,b_2),$
for $0 < \eps <1,$ there is an embedding 
\[ (1-\eps)(E(|OX|_b,|OY|_b)) \sembeds H_b\]
where the formulas for $|OX|_b,|OY|_b$ are given in \eqref{eq:QTH} where here we are emphasizing that $|OX|_b,|OY|_b$ are functions of $b$. As the ellipsoid embeddings are continuous in $b,$ taking the limit as $b \to m_\la/d_\la$ implies that there is an embedding
\[ (1-\eps)(E(|OX|_{m_\la/d_\la},|OY|_{m_\la/d_\la})) \sembeds H_{m_\la/d_\la}.\]

Note that when $b=m_\la/d_\la,$ $|XV|_{m_\la/d_\la}=0,$ which implies that the above embedding is a full filling. 

Using the expression for $|OX|_{m_\la/d_\la},|OY|_{m_\la/d_\la},$ we obtain 
\[ |OY|_{m_\la/d_\la}=\frac{p_\la q_\la-1}{d_\la q_\la}\]
where we used that $d_\la^2-m_\la^2=p_\la q_\la-1.$

For $|OX|_{m_\la/d_\la}$, we get
\begin{equation*}
    |OX| = \frac{m_\rho' \left(\frac{m_\lambda}{d_\lambda}\right) - d_\rho'}{q_\rho} = \frac{m_\rho' m_\lambda - d_\rho' d_\lambda}{d_\lambda q_\rho}.
\end{equation*}
Expanding the numerator gives
\begin{align*}
    m_\rho' m_\lambda - d_\rho' d_\lambda &= (m_\rho - q_\rho)m_\lambda - (d_\rho - 3q_\rho)d_\lambda \\
    &= -(d_\lambda d_\rho - m_\lambda m_\rho) + q_\rho(3d_\lambda - m_\lambda).
\end{align*}
Translating into $(p,q,t)$-coordinates, for the $q_\rho$ coefficient, we have $3d_\lambda - m_\lambda = p_\lambda + q_\lambda.$
For the cross term, expanding the product gives
\begin{align*}
    d_\lambda d_\rho - m_\lambda m_\rho = \frac{1}{8} \left( (p_\lambda+q_\lambda)(p_\rho+q_\rho) - t_\lambda t_\rho \right).
\end{align*}
Substituting these back into the numerator expression, we obtain
\begin{align*}
    m_\rho' m_\lambda - d_\rho' d_\lambda
    &= \frac{1}{8} \left( t_\lambda t_\rho - (p_\lambda+q_\lambda)(p_\rho - 7q_\rho) \right).
\end{align*}
Because $\bm{x}_\lambda$ and $\bm{x}_\rho$ are $(p_\rho q_\lambda - p_\lambda q_\rho)$-compatible, by Lemma~\ref{lem:7cond}, we simplify to 
\begin{align*}
    m_\rho' m_\lambda - d_\rho' d_\lambda &= q_\la q_\rho.
\end{align*}
Finally, substituting this result back into the fractional expression for $|OX|$, we conclude
\begin{equation*}
    |OX| = \frac{q_\lambda q_\rho}{d_\lambda q_\rho} = \frac{q_\lambda}{d_\lambda}.
\end{equation*}
Hence, we obtain 
\[ (1-\eps)E\left( \frac{q_\la}{d_\la},\frac{p_\la q_\la-1}{d_\la q_\la}\right) \sembeds H_{m_\la/d_\la} \iff (1-\eps)E\left(1,\frac{p_\la q_\la-1}{q_\la^2}\right) \sembeds \frac{d_\la}{q_\la}H_{m_\la/d_\la}\]
as desired. 
\end{proof}

The case that the embedding produced in Proposition~\ref{prop:fullFilling} is optimal is immediate as it agrees with the volume bound. Here, we consider cases where the embedding detected from the ATBD is not on the volume bound. From \eqref{eq:funSup}, we have that if $\bm{x}=(p,q,t)$ is a quasi-perfect class for $M$ with integral solution, then  $\bm{x}$ determines a piecewise linear function $\mu_{\bm{x},M}(z)$ which is a lower bound for $c_{M}(z).$

In certain cases, these lower bounds are exactly the reciprocal of the affine lengths in Definition~\ref{def:DecQuad}. We first consider formulas for the $\mu_{\bm{x},M}(z)$:

\begin{lemma}\label{lem:obsFunction}
    \begin{itemlist}
        \item[{(i)}]\cite[Lemma 16]{ICERM} Let $\bm{x}=(p,q,t) \in \X_H$ be a positive class with integral solutions $(d;m)$. Then, there exist real numbers $z_1<p/q<z_2$ such that 
        \[ c_{H_b}(z) \geq \mu_{\bm{x},H_b}(z):=\begin{cases}
            \frac{qz}{d-mb} & \text{if $z_1<z<p/q$} \\
            \frac{p}{d-mb} & \text{if $p/q \leq z < z_2$}
        \end{cases}\]
        \item[{(ii)}] Let $\bm{x}=(p,q,t) \in \X_P$ be a positive class with integral solutions $(e,f)$. Then, there exist real numbers $z_1<p/q<z_2$ such that 
        \[ c_{P_b}(z) \geq \mu_{\bm{x},P_b}(z):= \begin{cases}
            \frac{qz}{e+fb} & \text{if $z_1<z<p/q$} \\
            \frac{p}{e+fb} & \text{if $p/q \leq z < z_2$}
        \end{cases}\]
    \end{itemlist}
\end{lemma}

\begin{rmk} \label{rmk:obsAndDiagrams}
    Note that for $Q_{\bullet,b}(\Tt)$ where $\Tt=(\bm{x}_\la,\bm{x}_\mu,\bm{x}_\rho),$ we have that for some interval of $z$-values, the affine length of $OY$ satisfies: 
    \[ \mu_{\bm{x}_\la,\bullet}(z)=\frac{z}{|OY|}.\]
    Further, note the slope of the nodal ray $\vn_Y$ is $p_\la/q_\la$ where this is the $z$-value of the nonsmooth point in $\mu_{\bm{x}_\la,\bullet}(z).$
    To see the relationship between $|OX|$ and the obstruction functions, we use Lemma~\ref{lem:symTriple}. As $S\bm{x}_\rho=(6p_\rho-q_\rho,p_\rho,-t_\rho),$ by Lemma~\ref{lem:symTriple} and the definition of $m_\rho',d_\rho',e_\rho',f_\rho'$ we have that for some $z_2$, for $p/q \leq z \leq z_2,$ the reciprocal of the affine length of $OX$ satisfies:
    \[ \mu_{S\bm{x}_\rho,\bullet}(z)=\frac{1}{|OX|}.\] 
\end{rmk}

ATBDs can not only produce embeddings, but also give obstructions called {\bf visible obstructions} by McDuff and Siegel. A restatement of \cite[Corollary 6.1.4]{McSi} in the terminology of this paper:  
\begin{lemma}\cite[Corollary 6.1.4]{McSi} \label{lem:visObs}
Let $Q:=OXVY$ denote an ATBD for either $P_b,H_b$ with $\bullet=P,H$. Assume that $Q=Q_{\bullet, b}(\Tt)$ where $\Tt=(\bm{x}_\la,\bm{x}_\mu,\bm{x}_\rho).$ If $sQ$ is well-defined, then \[c_{\bullet,b}(z)=\begin{cases}
    \frac{z}{|OY|} & \text{for all $z \in \left[\frac{|OY|}{|OX|},\frac{p_\la}{q_\la}\right]$ if $|OY|>|OX|$} \\
    \frac{1}{|OY|} & \text{for all $z \in \left[\frac{q_\la}{p_\la},\frac{|OX|}{|OY|}\right]$ if $|OY|<|OX|$}
\end{cases}\]
If $\ov{s}Q$ is well-defined, then \[c_{\bullet,b}(z)=\begin{cases}
    \frac{z}{|OX|} & \text{for all $z \in \left[\frac{|OX|}{|OY|},\frac{6q_\rho-p_\rho}{q_\rho}\right]$ if $|OY|>|OX|$} \\
    \frac{1}{|OX|} & \text{for all $z \in \left[\frac{q_\rho}{6q_\rho-p_\rho},\frac{|OY|}{|OX|}\right]$ if $|OX|<|OY|$}
\end{cases}\]

\end{lemma}
\begin{proof}
    Assume that $sQ$ is well-defined. Let $X_s$ denote the point on $OX$ such that $\vn_Y$ extends to intersect $OX.$ Note that by definition, we have that $|OY|/|OX_s|=p_\la/q_\la.$

    By \cite[Corollary 6.1.4]{McSi}, for $|OX_s| \leq z' \leq |OX|,$ there is an embedding \[E\left(\frac{|OY|}{\la},\frac{z'}{\la}\right) \sembeds H_b\] for all $\la>1$ and no values of $\la$ such that $\la<1$. If $|OX_s|>|OY|,$ this implies that $c_{\bullet,b}(z)=\frac{1}{|OY|}$ for  $q_\la/p_\la=|OX_s|/|OY| \leq z \leq |OX|/|OY|$.

    Now, in the case where $|OY|>|OX|,$ we have $E(z,|OY|) \sembeds H_b$ is optimal. Setting $x=|OY|/z,$ we have that $c_{\bullet,b}(z)=z/|OY|$ for $|OY|/|OX| \leq z \leq |OY|/|OX_s|.$ The case where $\ov{s}Q$ is well-defined is similar. 
\end{proof}

We can use Lemma~\ref{lem:visObs} to give criteria on mutations which imply that $H_b$ or $P_b$ has a staircase. 
\begin{cor} \label{cor:visStair}
Fix \(\bullet\in\{H,P\}\) and \(b\). Let \(\{w_k\}_{k\ge0}\) be an infinite
sequence of words such that
\[
        w_kQ_{\bullet,b}^0:=OX_kV_kY_k
\]
is well-defined for every \(k\). If \(sw_kQ_{\bullet,b}^0\) is well-defined for every \(k\), and the affine
lengths \(|OY_k|\) are pairwise distinct, then \(c_{\bullet,b}\) has an infinite staircase. Similarly, if \(\ov{s}w_kQ_{\bullet,b}^0\) is
well-defined for every \(k\), and the affine lengths \(|OX_k|\) are pairwise
distinct, then \(c_{\bullet,b}\) has an infinite staircase.
\end{cor}
\begin{proof}
    For the case when $\ov{s}$ is well-defined, by Lemma~\ref{lem:visObs}, for each $k,$ the function $c_{\bullet,b}(z)$ 
    is equal to $\frac{z}{|OX_k|}$ or $\frac{1}{|OX_k|}$ on some interval. 
    For each \(k\), the visible-obstruction interval has an endpoint at which the slope of \(c_{\bullet,b}\) changes, and the value of this linear piece is determined by \(|OX_k|\). Since the \(|OX_k|\) are pairwise distinct, these visible pieces cannot all give the same nonsmooth point. A similar proof holds for the case where $s$ is well-defined. 
\end{proof}

We aim to use Corollary~\ref{cor:visStair} to prove Theorem~\ref{thm:ATFvis} and Corollary~\ref{cor:ATFvis}. By \cite[Theorem 1.13]{AADT} if $M$ has an infinite staircase, then it must accumulate to the solution to the quadratic equation in \eqref{eq:accPt} and further, the function is unobstructed at the accumulation point. We first show that under certain conditions if the limit of an infinite sequence of ATF mutations is a triangle, then this corresponds to being unobstructed at the accumulation point. In these cases, the ATBD can detect the accumulation point. 

\begin{lemma} \label{lem:fullFillingAccPt}
   Let $w_k=x^kw'$ (resp. $y^kw'$) where $w'$ is some word in $x,y,s,\ov{x},\ov{y},\ov{s}.$  Assume that $w_kQ_{\bullet,b}^{0}:=OX_kV_kY_k$ is well-defined for each $k$ for either $\bullet=P,H$. If the $\lim_{k \to \infty} w_kQ_{\bullet,b}^{0}$ is such that the limit of the affine length of $V_kY_k$ (resp. $X_kV_k$) goes to zero and 
   the limit of the slope of $X_kV_k$ (resp. $V_kY_k$) is irrational, then $c_{\bullet,b}(z_\infty)$ is unobstructed. 
\end{lemma}
\begin{proof}
    Consider the case where $w_k=y^kw'.$ Let $\ell_k$ denote the slope of the edge $V_kY_k$, which if $w_k\Tt_0^\bullet=(\bm{x}_{k},\bm{x}_{k+1},\bm{x}_{\rho})$ is 
    \[\ell_k=p_{k}/q_{k}-1/q_{k}^2.\] 
    Since the slopes of \(V_kY_k\) converge to an irrational number, the
primitive direction vectors of \(V_kY_k\) have norm tending to infinity.
The Euclidean lengths of \(V_kY_k\) remain bounded because the vertices
remain in a bounded region. Hence the affine lengths of \(V_kY_k\) tend to
zero.
 Further, by assumption, the affine length of $X_kV_k$ is going to zero. Let $|OY_\infty|:=\lim_{k \to \infty} |OY_k|$ and note that $|OX|$ is constant at each stage of mutation. As mutations preserve the volume and perimeter, due to the above observations, we have that
    \begin{equation} \label{eq:perVolQ}
         |OY_\infty|+|OX|=Per(Q) \quad \text{and} \quad |OY_\infty|\cdot |OX|=Vol(Q).
    \end{equation}
    By \eqref{eq:accPt}, we have $z_\infty$ satisfies
    \[ \frac{1}{z_\infty}+z_\infty=\frac{Per^2}{Vol}-2.\]
    Hence, by \eqref{eq:perVolQ}, we can see that $\frac{|OY_\infty|}{|OX|}$ and $\frac{|OX|}{|OY_\infty|}$ are the solutions to the accumulation point quadratic equation. From the limit of ATBD, we obtain the embedding 
    \[ (1-\eps)E(|OX|,|OY_\infty|) \sembeds M_b\] for $M=P,H$ corresponding to a full filling at the accumulation point. 
    The case of $w_k=x^kw'$ is identical. 
\end{proof}

\begin{rmk} \rm 
 
    Note that the condition in Lemma~\ref{lem:fullFillingAccPt} about the limit of the slope of either $X_kV_k$ or $V_kY_k$ being irrational will hold if $t_\rho,t_\la \geq 3$ by \cite[Remark 2.1.3]{MM}. 
\end{rmk}

\subsection{Correctly Ordered Triples}
In order to prove Theorem~\ref{thm:ATFvis}, we use Corollary~\ref{cor:visStair}. This corollary concerns when $s$ or $\ov{s}$ is well-defined. The assumptions in Theorem~\ref{thm:ATFvis} about correctly ordered triples ensure that we can guarantee that $s$ or $\ov{s}$ is well-defined for a sequence of mutations.

\begin{definition} \label{def:correctlyOrdered}
    We say a recursive triple $\Tt=(\bm{x}_\la,\bm{x}_\mu,\bm{x}_\rho)$ where $\eps=1$ if $\Tt$ is sign-matching and $\eps=-1$ otherwise is {\bf correctly ordered} if the following ordering conditions hold: 
    \begin{itemlist}
        \item[{i)}]  $3+2\sqrt{2}<p_\la/q_\la<p_\mu/q_\mu<p_\rho/q_\rho$, 
        \item[{ii)}] $(0,0,2)< (p_\la,q_\la,\eps t_\la),(p_\rho,q_\rho,\eps t_\rho) \leq (p_\mu,q_\mu,\eps t_\mu)$.
    \end{itemlist}
    Here inequalities between triples are interpreted coordinatewise.
\end{definition}

Under the assumption that a triple is correctly ordered, we now deduce various consequences about sequences of mutations of that triple. 
\begin{lemma} \label{lem:ordered}
    Let $\Tt=(\bm{x}_\la,\bm{x}_\mu,\bm{x}_\rho)$ be a recursive triple such that $\eps t_\la,\eps t_\rho>0$ and $p_\bullet,q_\bullet>0$ where $\eps=1$ if $\Tt$ is sign-matching and $-1$ otherwise. Then, the triple is ordered, i.e. we have that
    \[ p_\la/q_\la < p_\mu/q_\mu <p_\rho/q_\rho.\]
    \end{lemma}
\begin{proof}
    Since \(q_\lambda,q_\mu,q_\rho>0\), the identities of a recursive triple
\[
\eps t_\lambda=p_\rho q_\mu-p_\mu q_\rho,\qquad
\eps t_\rho=p_\mu q_\lambda-p_\lambda q_\mu
\]
imply
\[
\frac{p_\mu}{q_\mu}<\frac{p_\rho}{q_\rho},
\qquad
\frac{p_\lambda}{q_\lambda}<\frac{p_\mu}{q_\mu}.
\]
\end{proof}

\begin{lemma}\label{lem:correctlyOrdered}
    If $\Tt=(\bm{x}_\la,\bm{x}_\mu,\bm{x}_\rho)$ is correctly ordered, then $w\Tt$ is correctly ordered where $w$ is any word in $x,y.$
\end{lemma}
\begin{proof}
    It is sufficient to check for when $w=x,y.$ Assume $w=x.$ Then, we have $x\Tt=(\bm{x}_\la,\bm{x}_{x\mu},\bm{x}_\mu)$ where
    \[ \bm{x}_{x\mu}=\eps t_\la \bm{x}_\mu-\bm{x}_\rho.\] 
    We first show condition (ii) of being correctly ordered holds.
    Since \(\eps t_\lambda>2\) and \(\bm{x}_\rho\leq \bm{x}_\mu\) coordinate wise,
\[
\bm{x}_{x\mu}-\bm{x}_\mu
=
(\eps t_\lambda-1)\bm{x}_\mu-\bm{x}_\rho
\geq
(\eps t_\lambda-2)\bm{x}_\mu>0.
\]
    
    Condition (i) holds for $x\Tt$ by Lemma~\ref{lem:ordered} as $\Tt$ is correctly ordered. A similar proof holds when $w=y.$ 
\end{proof}

In this section, $\acc_\bullet(b)$ is defined via the solution $z_\infty \geq 1$ given in \eqref{eq:accPt} for $\bullet=P,H$. Hence, $\acc_H(b)$ is the solution $z_\infty \geq 1$ to 
\[
        z+\frac{1}{z}
        =
        \frac{(3-b)^2}{1-b^2}-2
\]
and
$\acc_P(b)$ is the solution $z_\infty \geq 1$ to 
\[  z+\frac{1}{z}=
        \frac{(2+2b)^2}{2b}-2.\]
We outline some properties of this function now. 
\begin{lemma}\label{lem:accMonotone}
The function $\acc_H$ is strictly decreasing on $(0,1/3)$ and strictly increasing on $(1/3,1)$. The function $\acc_P$ is strictly decreasing on $(0,1)$ and strictly increasing on $(1,\infty)$. 
\end{lemma}

\begin{proof}
The function
\[
        z\mapsto z+\frac{1}{z}
\]
is strictly increasing for $z>1$. Therefore the monotonicity of
$\acc_H$ is the same as the monotonicity of $ f(b):=\frac{(3-b)^2}{1-b^2}-2.$ The sign of $f'(b)$ is the same as the sign of $3b-1$ implying the result. A similar result follows for $P.$
\end{proof}

\begin{lemma} \label{lem:pqAlOrder}
Let $(d;m;p,q)$ or $(e,f;p,q)$ be a quasi-perfect class for $H$ or $P$ such that $p/q>1$ where $\al=m/d,e/f$ and $\bullet=H,P$. 
Then, \[p/q<\acc_\bullet(\al).\]
\end{lemma}
\begin{proof}
    Let $z_0=\acc_\bullet(\al).$ Then, by \eqref{eq:accPt}, we have that
    \begin{align*}
         z_0+\frac{1}{z_0}&=\left(\frac{(3-m/d)^2}{1-(m/d)^2}-2\right) \\
         &=\frac{p^2+q^2+2}{pq-1}
    \end{align*}
    As the function $z \to z+1/z$ increases for $z>1$ and the inequality
    \[ \frac{p}{q}+\frac{q}{p}<\frac{p^2+q^2+2}{pq-1}\] holds,
    the result follows. 
\end{proof}

\begin{lemma} \label{lem:pqAlOrder2}
    Let $\Tt=(\bm{x}_\la,\bm{x}_\mu,\bm{x}_\rho)$ be a correctly ordered triple for $\bullet=H,P$. Then, 
    \[ p_\la/q_\la<p_\mu/q_\mu<\acc_\bullet(\al_\mu)<p_\rho/q_\rho<\acc_\bullet(\al_\rho)\]
\end{lemma}
\begin{proof}
    The ordering of $p_\la/q_\la,p_\mu/q_\mu,p_\rho/q_\rho$ is by assumption as the triple is ordered. The inequalities about $p_\bullet/q_\bullet<\acc_\bullet(\al_\bullet)$ is precisely Lemma~\ref{lem:pqAlOrder}. The inequality about $\acc_\bullet(\al_\mu)<p_\rho/q_\rho$ follows from \cite[Lemma 2.1.10]{MMW} for the case $(d;m;p,q)$. 
    The inequality for a class $(e,f;p;q)$ follows precisely the same proof because for the quasi-perfect class we also have if $w=\acc(e/f)$, then $w$ satisfies $w+\frac{1}{w}=\frac{p^2+q^2+2}{pq-1}.$
\end{proof}

The next lemma is a key result as in Theorem~\ref{thm:ATFvis}, we are considering a sequence of mutations given by $y^k\Tt.$ 

\begin{lemma} \label{lem:moreAlOrder}
Let $\Tt$ denote a correctly ordered recursive triple where $\eps=1$ if $\Tt$ is sign-matching and $\eps=-1$ otherwise. 
Further, let $y^k\Tt:=(\bm{x}_{k-1},\bm{x}_k,\bm{x})$ and $xy^k\Tt:=(\bm{x}_{k-1},\bm{x}_{xk},\bm{x}_k)$. Then, we have that 
   \[ \eps\al_{xk} <\eps\al_i\]
   for all $i \geq k.$
\end{lemma}
\begin{proof}
By Lemma~\ref{lem:correctlyOrdered}, the triples $y^k\Tt$ and $xy^k\Tt$ are correctly ordered for all $k$, so we obtain the ordering
\[ p_{k-1}/q_{k-1}<p_{xk}/q_{xk}<p_k/q_k<p/q.\]
Additionally, by Lemma~\ref{lem:pqAlOrder2}, we have that
    \[ \frac{p_{k-1}}{q_{k-1}} < \frac{p_{xk}}{q_{xk}}<\acc_\bullet(\al_{xk}) <\frac{p_k}{q_k} <\acc_\bullet(\al_k) <p/q.\]

We obtain for all $i \geq k,$ we have that 
\[ p_k/q_k \leq p_i/q_i < \acc_\bullet(\al_i)\]
implying that $\acc_\bullet(\al_{x_k}) < p_k/q_k<\acc_\bullet(\al_i)$ for all $i \geq k.$
For \(H\), Lemma~\ref{lem:md13} says that the sign of
\(\alpha-1/3\) is the sign of \(t\). Since the triples are correctly
ordered, \(\eps t>0\), so all relevant \(\alpha\)-values lie on the
increasing branch of \(\acc_H\) when \(\eps=1\), and on the decreasing
branch when \(\eps=-1\). For \(P\), Lemma~\ref{lem:md13} says that the sign of \(\alpha-1\), where
\(\alpha=e/f\), is the sign of \(t\). Thus the same argument applies using
the monotonicity of \(\acc_P\) on \((0,1)\) and \((1,\infty)\).
Therefore, by Lemma~\ref{lem:accMonotone}, we obtain
\[
\eps\alpha_{xk}<\eps\alpha_i.
\]

\end{proof}

\begin{lemma}\label{lem:affineDegenerationLimitH}
 For $\bullet=H,P$, assume that
\[
        Q_k:=Q_{\bullet,b}(y^k\Tt)=OXV_kY_k
\]
is well-defined for all $k$, and that the affine length of the edge
$XV_k$ tends to zero. Then, in the case of $H_b$, we must have $b=\lim_{k \to \infty} \frac{m_k}{d_k},$ and in the case of $P_b,$ we must have $b=\lim_{k \to \infty} \frac{e_k}{f_k}$. 
\end{lemma}

\begin{proof}
Let $y^k\Tt:=(\bm{x}_{k-1},\bm{x}_k,\bm{x}_\rho).$ In the case of $H_b$, by Lemma~\ref{lem:XVVY}, 
there is an $\eps\in\{1,-1\}$ such that
\[
        |XV_k|
        =
        \eps\frac{m_{k-1}-d_{k-1}b}{q_\rho q_k}.
\]
Equivalently,
\[
        \frac{m_{k-1}}{d_{k-1}}-b
        =
        \eps\,|XV_k|\frac{q_\rho q_k}{d_{k-1}}.
\]
It remains to note that the factor $q_k/d_{k-1}$ is bounded. Here, $q_k,d_k$ have the same recursion parameter. Since the entries are positive and the recursion parameter is \(>2\), both
sequences grow like the dominant root of the recursion, so the ratio
\(q_k/d_{k-1}\) is bounded. The case of $P_b$ is identical as
\[ |XV_k|=\eps \frac{e_{k-1}-f_{k-1}b}{q_\rho q_k}.\]
\end{proof}
\begin{lemma}\label{lem:yTailXVisible}
Let $\Tt$ be a correctly ordered recursive triple. 
Assume that
\[
        Q_k:=Q_{\bullet,b}(y^k\Tt)=OXV_kY_k
\]
is well-defined for all $k$ and $Q_k$ affinely degenerates to a triangle for $\bullet=H,P$. In the case of $H_b$, we assume $b \neq 1/3,$ and in the case of $P_b$ we assume $b \neq 1.$ Then $xQ_k$ is well-defined for all sufficiently large $k$.

\end{lemma}

\begin{proof}
As $Q_k$ degenerates to a triangle and $X$ is constant, we must have $|XV_k| \to 0.$
Define
\[
        y^k\Tt=(\bm{x}_{k-1},\bm{x}_k,\bm{x}_\rho).
\]
Let
\[
        A:=|OX|,
        \qquad
        B:=\lim_{k\to\infty}|OY_k|.
\]
By the affine triangular degeneration, the limiting triangle has vertices
\[
        O=(0,0),\qquad X=(A,0),\qquad Y_\infty=(0,B).
\]
Let
\[
        z_\infty:=\acc_\bullet(b).
\]
Then, by the proof of Lemma~\ref{lem:fullFillingAccPt}, for the limiting triangle,
\[
        \left\{\frac{A}{B},\frac{B}{A}\right\}
        =
        \left\{z_\infty,\frac{1}{z_\infty}\right\}.
\]
For each $k$, the nodal ray at $X$ of $Q_k$ is constant and $\vec n_X=(p_\rho-6q_\rho,\ q_\rho).$ Note, if $\ov{s}Q_k$ is defined for just some $k,$ this immediately implies that $p_\rho-6q_\rho<0$. 

Suppose, for contradiction, that $\ov{s}Q_k$ is well-defined for infinitely
many $k$. Passing to an infinite subsequence if necessary, the nodal ray at
$X$ intersects the side $OY_k$ for all $k$ in this subsequence. Taking the
limit, the same ray must intersect the segment $OY_\infty$. Hence, we get
\begin{equation} \label{eq:bound6prho}
        A \frac{q_\rho}{6q_\rho-p_\rho}\leq B \iff  \frac{p_\rho}{q_\rho}
        \leq
        6-\frac{A}{B}.
\end{equation}

On the other hand, since the triples $y^k\Tt$ are correctly ordered,
Lemma~\ref{lem:pqAlOrder2} gives
\[
        \acc_\bullet(\alpha_k)<\frac{p_\rho}{q_\rho}
\]
for every $k$. Taking $k\to\infty$ and using that $\alpha_k\to b$ by Lemma~\ref{lem:affineDegenerationLimitH}, we obtain
\[
        z_\infty\leq \frac{p_\rho}{q_\rho}.
\]

We now compare the two bounds. In the case of $H_b$ we have $b \neq 1/3$ and in the case of $P_b$ we have $b \neq 1,$ hence in either case, the accumulation point
satisfies
\[
        z_\infty>3+2\sqrt{2} \iff  z_\infty+\frac{1}{z_\infty}>6
\]

Using this and the fact that $A/B$ is either $z_\infty$ or $1/z_\infty,$ we conclude that 

\[
        6-\frac{A}{B}<z_\infty \leq p_\rho/q_\rho,
\]
but this contradicts \eqref{eq:bound6prho}.

Therefore $\ov{s}Q_k$ can be well-defined for only
finitely many $k$. 

If the nodal ray at \(X\) hit the vertex \(Y_k\) for infinitely many \(k\),
then passing to a subsequence and taking the limit would give
\[
        \frac{p_\rho}{q_\rho}=6-\frac{A}{B}.
\]
But we have shown
\[
        6-\frac{A}{B}<z_\infty\leq \frac{p_\rho}{q_\rho},
\]
so this equality is impossible. Hence the vertex-hit case occurs only
finitely many times. Therefore $xQ_k$ is
well-defined for all sufficiently large $k$.
\end{proof}

\subsection{Proof of Theorem~\ref{thm:ATFvis} and Corollary~\ref{cor:ATFvis}}

Before proving Theorem~\ref{thm:ATFvis} and Corollary~\ref{cor:ATFvis}, we give an illustrative example. 
\begin{example} \label{ex:visStair}

In this example, we demonstrate the visible embeddings constructed in Theorem~\ref{thm:ATFvis} for one particular $b$-value such that $H_b$ has an infinite staircase. 
We consider the example of the increasing staircase given by $y^k\ov{x}^2Q_{H_b}^0:=OXV_kY_k$. Let 
\[ \Tt_k:=y^{k}\ov{x}^2\Tt_0^H:=(\bm{x}_k,\bm{x}_{k+1},(8,1,5))\]
where $\bm{x}_0=(1, 1, 2)$ and $\bm{x}_1=(6,1,3).$ It can be easily checked that these triples are correctly ordered for each $k$. Based on the work of Bertozzi et al. in \cite{ICERM}, letting $b:=\lim_{k \to \infty} m_k/d_k=\frac{11-\sqrt{21}}{10} \approx 0.6417$, the ellipsoid embedding function for $H_b$ has an increasing staircase. To prove this, they showed there is a nonsmooth point at each $p_k/q_k$ for $\bm{x}_k=(p_k,q_k,t_k).$ Note that $m_k/d_k$ is a decreasing sequence, so by Lemma~\ref{lem:bDefH1}, the ATF mutation sequence $y^k\ov{x}^2Q_{H_b}^0$ is well-defined for this $b$-value for all $k$.  As $|XV_k|=\frac{m_{k}-d_{k}b}{q_{k}}$, we have that for this $b$-value, $\lim_{k \to \infty} |XV_k|=0$. 

For each $k,$ the mutation $sxQ_{H_b}^0(\Tt_k)$ is well-defined, which by Corollary~\ref{cor:visStair} implies that there is an ATF-visible staircase for this $b$-value. Note that $xQ_{H_b}(\Tt_k)$ is well-defined as for this example $p_\rho/q_\rho=8>6,$ so $\vn_X$ has positive slope. Then, $sxQ_{H_b}(\Tt_k)$ is well-defined as letting $x\Tt_k=(\bm{x}_k,\bm{x}_{x,k+1},(8,1,5)),$ we have that $sxQ_{H_b}(\Tt_k)$ will be well-defined if $m_{x,k+1}/d_{x,k+1}<b$. 

We outline exactly what these embeddings are here. We define the triples: 
\begin{align*}
	\Tt_k&:=(\bm{x}_k,\bm{x}_{k+1},(8,1,5)) \\
	x\Tt_k&:=(\bm{x}_k,\bm{x}_{x,k+1},\bm{x}_{k+1}) \\
	sx\Tt_k&:=(\bm{x}_{x,k+1},\bm{x}_{k+1},(6p_k-q_k,p_k,-t_k)) 
\end{align*} 
Note that for the class $(6p_k-q_k,p_k,-t_k)$ with solutions $(d;m)$ has $m'=-m_k$ and $d'=-d_k$ by Lemma~\ref{lem:symTriple}. Hence, for 
$Q(sx\Tt_k)=O\tilde{X}\tilde{V}\tilde{Y},$ we have that 

\[|O\tilde{Y}|=\frac{d_{x,k+1}-m_{x,k+1}b}{q_{x,k+1}}=\frac{z}{\mu_{\bm{x}_{x,k+1},b}(z)} \quad \text{if $z \in (p_{x,k+1}/q_{x,k+1}-\eps,p_{x,k+1}/q_{x,k+1})$}\] 
and 
\[ |O\tilde{X}|=\frac{d_k-m_kb}{p_k}=\frac{1}{\mu_{\bm{x}_k,b}(z)} \quad \text{if $z \in (p_k/q_k,p_k/q_k+\eps)$} .\]

The embedding we obtain for $O\tilde{X}\tilde{V}\tilde{Y}$ is a nonsmooth corner between $p_k/q_k$ and $p_{x,k+1}/q_{x,k+1}$ given by the intersection of the obstructions from \[(p_k,q_k,t_k) \quad \text{and} \quad (p_{x,k+1},q_{x,k+1},t_{x,k+1}).\] This exemplifies that for $c_{H_b}(z),$ the function before the accumulation point is not completely determined by the obstructions from the classes $(p_k,q_k,t_k)$, as the obstructions from the $x$-mutations above also appear.

\end{example}
 
We can now prove Theorem~\ref{thm:ATFvis}
\begin{proof}[Proof of Theorem~\ref{thm:ATFvis}]
We prove the case $\bullet=H$ and $w_k=y^kw'$. The proof for $\bullet=P$
is identical after replacing $\alpha=m/d$ by the corresponding ratio for
$P$ and using the analogous well-definedness criterion. 

Let
\[
        \Tt:=w'\Tt_0^H.
\]
Write
\[
        y^k\Tt=(\bm{x}_{k-1},\bm{x}_k,\bm{x}_\rho) \quad \text{and} \quad  xy^k\Tt=(\bm{x}_{k-1},\bm{x}_{xk},\bm{x}_k).
\]
For each entry $\bm{x}_i$, write
\[
        \alpha_i=\frac{m_i}{d_i},
        \qquad
        \alpha_{xk}=\frac{m_{xk}}{d_{xk}}.
\]
Let $\eps=1$ in the sign-matching case and $\eps=-1$ in the
sign-alternating case.

By Theorem~\ref{thm:ATFMut}, we have
\[
        w_kQ_{H_b}^0
        =
        y^kw'Q_{H_b}^0
        =
        Q_{H_b}(y^k\Tt).
\]
Denote this quadrilateral by $Q_k:=OXV_kY_k.$ By Lemma~\ref{lem:affineDegenerationLimitH}, we have $\al_k \to b.$ We first show $b$ is irrational. This follows from \cite[Lemma 3.1.4]{MM} as $b=\lim_{k \to \infty}m_k/d_k$ and $m_k,d_k$ satisfy the same linear recursion with recursion parameter strictly larger than $2$ due to the correctly ordered assumption.

By Lemma~\ref{lem:yTailXVisible}, we have for sufficiently large $k,$ $xQ_k$ is well-defined. We show that for such $k,$ we have $sxQ_k$ is well-defined. By Lemma~\ref{lem:moreAlOrder}, applied to the correctly
ordered triple $\Tt$, we have
\[
        \eps\alpha_{xk}<\eps\alpha_i
\]
for all $i\geq k$. Taking $i\to\infty$ and using $\alpha_i\to b$, we get
\[
        \eps\alpha_{xk}\leq \eps b.
\]
Since $\alpha_{xk}$ is rational and $b$ is irrational, the inequality is
strict:
\[
        \eps\alpha_{xk}<\eps b.
\]
Applying Lemma~\ref{lem:definedH} to the triple
\[
        xy^k\Tt=(\bm{x}_{k-1},\bm{x}_{xk},\bm{x}_k),
\]
this inequality implies that
\[
        sxQ_k=sxy^kw'Q_{H_b}^0
\]
is well-defined.

Therefore, for sufficiently large $k$,
$sxQ_k$ is well-defined. For such $k,$ let \[\tilde{Q}_k:=O\tilde{X}_k\tilde{V}_k\tilde{Y}_k:=xQ_k\] and let \[\ell_k:=|O\tilde{Y}_k|=\frac{d_{k-1}-m_{k-1}b}{q_{k-1}}.\]
We claim the $\ell_k$ are pairwise distinct.
For contradiction, suppose \[
        \frac{d_i-m_i b}{q_i}
        =
        \frac{d_j-m_j b}{q_j}.
\]
Then
\[
        d_iq_j-d_jq_i
        =
        b(m_iq_j-m_jq_i).
\]
As $b$ is irrational, this implies $d_iq_j-d_jq_i=m_iq_j-m_jq_i=0,$ so $m_i/q_i=m_j/q_j$ and $d_i/q_i=d_j/q_j$. Using that $p=3d-m-q,$ this implies $p_i/q_i=p_j/q_j$. As the sequence $\{p_k/q_k\}$ is strictly increasing, we conclude the lengths $\ell_k$ are pairwise distinct.

Therefore, 
Corollary~\ref{cor:visStair} applies and gives infinitely many visible
nonsmooth points of $c_{H_b}$.
\end{proof}

If $c_{\bullet,b}(z)$ has infinitely many nonsmooth points below (resp. above) the accumulation point, we call this an {\bf increasing staircase} (resp. {\bf decreasing staircase}). We now restate Corollary~\ref{cor:ATFvis} more precisely. 

\begin{cor} \label{cor:ATFvis2}
For all $b$-values such that $c_{H_b}(z)$ has an increasing staircase and no decreasing staircase, the function has an ATF-visible staircase. 
\end{cor}

To prove the corollary, we first give some background from \cite{MMW,MPW} about which $b$-values have an increasing staircase and no decreasing staircase. The special rational case is $b=1/3.$ This has an ATF-visible staircase from \cite{AADT,McSi}.

We now review the irrational case. Assume $\Tt$ is a triple of the form $wS^iR^\delta(\mathcal{B}_n)$, where $\delta\in\{0,1\}$, $n\geq 0$ is even, $w$ is a finite word in $x,y$, and $i\geq 0$ except in the case $n=0$ and $\delta=1$, where we require $i\geq 1$.
By \cite{MMW,MPW}, taking $y^k\Tt=(\bm{x}_{k-1},\bm{x}_k,\bm{x}_\rho)$ we have that the values $b=\lim_{k\to \infty}\frac{m_k}{d_k}$ are precisely the irrational $b$-values with an increasing staircase and no decreasing staircase. Additionally, by \cite[Corollary 4.2.6]{MM}, the sequence $\frac{m_{k}}{d_{k}}$ strictly decreases if $\Tt$ is sign-matching and strictly increases if $\Tt$ is sign-alternating. 

Before proving the corollary, we give an example of why the corollary is specifically stated for increasing staircases. 

\begin{example}
    In contrast to Example~\ref{ex:visStair}, we explain why Corollary~\ref{cor:ATFvis2} is specific to increasing staircases and does not work for decreasing staircases. 

 We consider the example $x^ky\ov{x}^2Q_{H_b}^0=OX_kV_kY.$ On the triples, we have
   \[ x^ky\ov{x}^2\Tt_0^H:=((6,1,3),(p_{k+1},q_{k+1},t_{k+1}),(p_{k},q_{k},t_{k}))\]
   where $(p_1,q_1,t_1)=(29,4,13)$ and $(p_0,q_0,t_0)=(8,1,5).$
   By Bertozzi et al. in \cite{ICERM}, each of these $p_k/q_k$ is a nonsmooth point in a decreasing staircase for $b:=\lim_{k \to \infty} m_k/d_k=\frac{21 + 3\sqrt{5}}{44}\approx 0.63.$
   
   Note, here neither the limit of $|X_kV_k|$ nor $|V_kY|$ is zero for $b:=\lim m_k/d_k$ because 
   \[ |V_kY|=\frac{m_k'-d_k'b}{q_k},\] so this will be zero for $b':=\lim_{k \to \infty} m_k'/d_k'>1$. As $b'>1$, we have that $x^ky\ov{x}^2Q_{b'}^H$ is not well-defined.

  We can also check that for $y\ov{x}^2Q_{H_b}^0$ to be well-defined this requires $1/2<b<2/3$. In this example we always have that $p_k/q_k>6$ implying that $x^k$ will be defined for all $k$ in $1/2<b<2/3.$ Thus, 
   $x^ky\ov{x}^2Q_{H_b}^0$ is well-defined precisely when $1/2<b<2/3.$ 

For any $w$ consisting of $x,y,$ we have that $xwx^ky\ov{x}^2Q_{H_b}^0$ is always well-defined as $wx^ky\ov{x}^2\Tt_0^H$ always has $p_\rho/q_\rho \geq 6,$ so $\vn_X$ will have positive slope. Hence, the mutation $\ov{s}wx^ky\ov{x}^2Q_{H_b}^0$ will never be well-defined when $w$ is a word of $x,y.$

Further, there is no $w$ in $x,y$ such that 
\[ swx^ky\ov{x}^2Q_{H_b}^0\] 
is well-defined. This is because assuming that  $wx^ky\ov{x}^2Q_{H_b}^0$ is well-defined, then $ywx^ky\ov{x}^2Q_{H_b}^0$ is well-defined if $b<m_\mu/d_\mu$ where $m_\mu/d_\mu$ is the middle entry of $wx^ky\ov{x}^2\Tt_0^H$. This will hold for all words $w$ in $x,y$ by a similar argument to that of Lemma~\ref{lem:moreAlOrder}. 

Hence, a similar method to the case of the corresponding increasing staircase to get the visible embeddings does not work for the decreasing case.

\end{example}

We now prove the corollary. 

\begin{proof}[Proof of Corollary \ref{cor:ATFvis2}]
    Let $\Tt$ denote a triple such that $y^k\Tt$ corresponds to an increasing staircase that is not a decreasing staircase as explained above. Write $y^k\Tt=(\bm{x}_{k-1},\bm{x}_k,\bm{x}_\rho)$. By Theorem~\ref{thm:HTrip}, there is a word $w'$ such that $w'\Tt_0^H=\Tt$. Choose $b=\lim_{k \to \infty} \frac{m_{k}}{d_{k}}.$ Further, by Theorem~\ref{thm:HTrip}, $y^kw'Q_{H_b}^0=:OX_kV_kY_k$ is well-defined for all $k.$ 
    By Theorem~\ref{thm:ATFMut}, we have
\[
y^kw'Q_{H_b}^0=Q_{H_b}(y^kw'\Tt_0^H).
\]
The affine length formula gives
\[
|X_kV_k|=\eps\frac{m_{k-1}-d_{k-1}b}{q_\rho q_k}.
\]
Since $b=\lim_{k\to\infty}m_k/d_k$ and the relevant ratios $q_k/d_{k-1}$ are bounded, it follows that $|X_kV_k| \to 0$. Thus the sequence $y^kw'Q_{H_b}^0$ degenerates affinely to a triangle.

    Further, the slopes $p_k/q_k$ accumulate to $\acc(b)$, and as $|t_\rho| \geq 3,$ this is irrational.  For $k \geq 1,$ the ordering condition \(3+2\sqrt2<p_\lambda/q_\lambda<p_\mu/q_\mu<p_\rho/q_\rho\) follows from Lemma~\ref{lem:pOverq6} and the ordering results in~\cite[Definition 2.1.6]{MMW}. The remaining coordinatewise inequalities in Definition~\ref{def:correctlyOrdered} are exactly the positivity and ordering requirements built into the recursive triples used in~\cite[Definition 2.1.6]{MMW}.

    Hence, by Theorem~\ref{thm:ATFvis}, we conclude that each of these $b$-values has an ATF-visible staircase.  
\end{proof}

\end{document}